\documentclass[reqno,11pt]{amsart}
\usepackage{amssymb}
\usepackage{mathrsfs}
\usepackage{txfonts}
\usepackage{amsfonts}
\usepackage{amstext}
\usepackage{amssymb}
\usepackage{amsmath}
\usepackage{graphicx}
\usepackage{hyperref}
\usepackage{url}

\hypersetup{
        colorlinks   = true,
        citecolor    = blue,
        linkcolor    = blue,
        urlcolor     = blue
}
\allowdisplaybreaks
\newtheorem{theorem}{Theorem}[section]
\newtheorem{proposition}[theorem]{Proposition}
\newtheorem{lemma}[theorem]{Lemma}

\newtheorem{remark}[theorem]{Remark}
\newtheorem{definition}[theorem]{Definition}

\newtheorem{example}[theorem]{Example}

\DeclareMathOperator*{\einf}{einf}
\DeclareMathOperator*{\esup}{esup}

\DeclareMathOperator*{\eosc}{eosc}

\numberwithin{equation}{section}
\newcounter{counterConstant}

\makeatletter

\newcommand{\Rmnum}[1]{\expandafter\@slowromancap\romannumeral #1@}
\@namedef{subjclassname@2020}{\textup{2020} Mathematics Subject Classification}
\makeatother

\begin{document}
\title[Weak Elliptic Harnack Equality]{The weak elliptic Harnack inequality
revisited}
\author[Hu]{Jiaxin Hu}
\address{The Department of Mathematical Sciences, Tsinghua University,
Beijing 100084, China.}
\email{hujiaxin@tsinghua.edu.cn}
\author[Yu]{Zhenyu Yu}
\address{The Department of Mathematical Sciences, Tsinghua University,
Beijing 100084, P.R. China.}
\email{yuzy18@mails.tsinghua.cn}
\thanks{\noindent Supported by the National Natural Science Foundation of
China (No. 11871296).}
\date{\today }

\begin{abstract}
In this paper we firstly derive the weak elliptic Harnack inequality from
the generalized capacity condition, the tail estimate of jump measure and
the Poincar\'{e} inequality, for any regular Dirichlet form without killing
part on a measure metric space, by using the lemma of growth and the
John-Nirenberg inequality. We secondly show several equivalent
characterizations of the weak elliptic Harnack inequality for any (not
necessarily regular) Dirichlet form. We thirdly present some consequences of
the weak elliptic Harnack inequality.\bigbreak

\centering{\textit{Dedicated to the memory of Professor Ka-Sing Lau. }}
\end{abstract}

\maketitle
\tableofcontents

\section{Introduction and main results}

In 1961, Moser showed in \cite{Moser.1961.CPAM577} that the following \emph{%
elliptic Harnack inequality}, denoted by $(\mathrm{H})$, is true: for any
compact $D^{\prime }$ in a domain $D\subset \mathbb{R}^{n}$ and for any
function $u$ which is non-negative, harmonic (with respect to the symmetric,
uniformly elliptic divergence-form operator) in $D$, we have 
\begin{equation*}
\sup_{D^{\prime }}u\leq C\inf_{D^{\prime }}u,
\end{equation*}%
where $C=C(D^{\prime },D)\geq 1$ is a constant depending only on $D^{\prime
},D$. The importance of this inequality is that the constant $C$ is
independent of function $u$ (but may depend on two domains $D^{\prime }$, $D$%
). If further $D^{\prime },D$ are two concentric balls, for example, if $%
D=B(x,R)$ and $D^{\prime }=B(x,R/2)$, then 
\begin{equation}
\sup_{B(x,R/2)}u\leq C\inf_{B(x,R/2)}u,  \label{11}
\end{equation}%
where the constant $C\geq 1$ is independent not only of function $u$, but
also of ball $B$. The inequality (\ref{11}) says that a function, which is
both non-negative and harmonic in a ball, is nearly constant around the
center. The reader may consult a book \cite[Theorem 2.1.1]%
{Saloff-Coste.2002.190} for more details.

A symmetric, uniformly elliptic operator gives arise to a strongly local,
regular Dirichlet form in the Hilbert space $L^{2}(\mathbb{R}^{n},dx)$ (see
for example \cite[Chapter 1]{FukushimaOshimaTakeda.2011.489}\ on the basic
theory of Dirichlet forms on a Hilbert space). The elliptic Harnack
inequality plays an important role in analysis, for example, in showing the
uniformly local H\"{o}lder continuity of harmonic functions, or in obtaining
the lower estimate of the heat kernel, for a given Dirichlet form on a
metric space.

Since the Moser's celebrated paper \cite{Moser.1961.CPAM577}, there has been
an increasing interest in the study on the Harnack inequality for local
Dirichlet forms. In 1972, Bombieri and Giusti \cite{Bombieri.1972.IM24} used
the geometric analysis to prove a Harnack inequality for elliptic
differential equations on minimal surfaces. In 1980, Safonov \cite%
{Safonov.1983.ZNS272} obtained the elliptic Harnack inequalities for partial
differential operators in non-divergence form. After that, the elliptic
Harnack inequality was extended in various settings, see for example, by
Benedetto and Trudinger \cite[Theorem 3]{BenedettoTrudinger1984} in 1984 for
De Giorgi classes on Euclidean spaces, by Biroli and Mosco \cite%
{BiroliMosco.1995.AANLCSFMNRL9MA37} in 1995 for a certain class of local
Dirichlet forms on discontinuous media, by Strum \cite[Proposition 3.2]%
{Sturm.1996.JMPA9273} in 1996 for time-dependent local Dirichlet forms on
compact metric spaces, and by Cabr\'{e} \cite{Cabre.1997.CPAM623} in 1997
for non-divergence elliptic operators on Riemannian manifolds with
non-negative curvature. In 2005, Barlow \cite[Theorem 2]{Barlow.2005.BLMS200}
showed that the elliptic Harnack inequality is equivalent to an annulus-type
Harnack inequality for Green's functions in the context of random walks on
graphs. In 2015, Grigor'yan, Hu and Lau \cite{GrigoryanHuLau.2015.JMSJ1485}
gave an equivalent characterization for the elliptic Harnack inequality and
the mean exit time estimate combined, for any strongly local, regular
Dirichlet form on a metric measure space, by using a more general Poincar%
\'{e} inequality and the generalized capacity inequality (see also an
earlier work \cite{GrigoryanHu.2014.CJM641}). In 2018, Barlow and Murugan 
\cite{BarlowMurugan.2018} showed that the elliptic Harnack inequality is
stable under bounded perturbations for strongly local, regular Dirichlet
forms on a length metric space, but assuming the existence of Green
function. Recently, this result has been improved by Barlow, Chen and
Murugan in \cite{BarlowChenMurugan.2020}, without assuming the existence of
Green functions and a length but assuming the relative ball-connectedness.

The Harnack inequality above is investigated only for local Dirichlet forms.
In recent years, the people have begun to study the elliptic Harnack
inequality for non-local operators or non-local Dirichlet forms. It can be
imagined that the classical Harnack inequality like the version (\ref{11})
no longer holds for non-local operators (see, for example \cite[Section 3]%
{BassChen2010} and \cite[Theorem 2.2]{Dyda.2020.AP317} for $\alpha $-stable
processes). Instead, a weak Harnack inequality different from (\ref{11})
should take place. In this direction, the reader may refer to \cite[Theorem
1.2]{Castro.2014.JFA1807}, \cite{DiKuusiPalatucci.2016.AIHPANL1279} and \cite%
[Theorem 1.6]{Dyda.2020.AP317} for non-local integro-differential operators, 
\cite{Maldonado.2017.CVPDE103} for the fractional non-local linearized
Monge-Amp\`{e}re equation, and \cite{ChenKumagaiWang.2016.EH} for pure jump
type Dirichlet forms.

In this paper, we are concerned with the weak elliptic Harnack inequality
under a more general framework (We do not touch the parabolic Harnack
inequality in this paper). Our underlying space is a metric measure space,
which may be bounded or unbounded, and our Dirichlet form is mixed, which
may be local or non-local, whose jump kernel may not exist. The main results
of this paper are as follows:

\begin{itemize}
\item to establish the weak elliptic Harnack inequality for local or
non-local regular Dirichlet forms (Theorem \ref{thm:M1} below);

\item to study the relationship among different versions of the weak
elliptic Harnack inequality appearing in the literature (Theorem \ref%
{thm:M10} below).
\end{itemize}

Let us state our framework of this paper. Let $(M,d)$ be a locally compact
separable metric space and $\mu $ be a Radon measure on $M$ with full
support. The triple $(M,d,\mu )$ is called a \emph{metric measure space}.
Denote by $B\left( x,r\right) $ an open metric ball of radius $r>0$ centered
at $x$, that is,%
\begin{equation*}
B_{r}(x):=B\left( x,r\right) :=\{y\in M:d(y,x)<r\},
\end{equation*}%
and its volume function is denoted by 
\begin{equation*}
V\left( x,r\right) :=\mu \left( B\left( x,r\right) \right) .
\end{equation*}%
For a ball $B=$ $B\left( x,r\right) $ and $\lambda >0$, the letter $\lambda
B:=B\left( x,\lambda r\right) $ denotes the concentric ball of $B$. In this
paper, we assume that every ball $B(x,r)$ is precompact.

Note that a ball in a metric space may not have a unique centre and radius,
and even if the centre is fixed, the radius may not be unique. For this
reason we always require a ball to have a fixed centre and radius in this
paper. When we pick up a ball $B(y,s)$ contained in a bigger ball $B(x,r)$,
we always assume that its radius $s$ is less than $2r$. Let $\overline{R}$
be any number in $(0,\mathrm{diam}$($M)]$. Since the metric space considered
in this paper may be bounded or unbounded, the number $\overline{R}$ may be
finite or infinite.

We say that the \emph{volume doubling condition} $(\mathrm{VD})$ holds if
there exists a constant $C_{\mu }\geq 1$ such that for all $x\in M$ and $r>0$%
, 
\begin{equation}
V(x,2r)\leq C_{\mu }V(x,r).  \label{eq:vol_1}
\end{equation}%
It is known that if condition $(\mathrm{VD})$ holds, then there exists a
positive number $d_{2}$ such that for all $x,y\in M$ and all $0<r\leq
R<\infty $, 
\begin{equation}
\frac{V(x,R)}{V(y,r)}\leq C_{\mu }\left( \frac{d(x,y)+R}{r}\right) ^{d_{2}}
\label{eq:vol_3}
\end{equation}%
with the same constant $C_{\mu }$ in (\ref{eq:vol_1}), see for example \cite[%
Proposition 5.1]{GrigoryanHu.2014.MMJ505}.

We say that the \emph{reverse} \emph{volume doubling condition} $(\mathrm{RVD%
})$ holds if there exist two positive constants $C_{d}\leq 1$ and $d_{1}$
such that for all $x\in M$ and $0<r\leq R<\overline{R}$ 
\begin{equation}
\frac{V(x,R)}{V(x,r)}\geq C_{d}\left( \frac{R}{r}\right) ^{d_{1}}.
\label{eq:vol_2}
\end{equation}

Let $w:M\times \lbrack 0,\infty )\rightarrow \lbrack 0,\infty )$ be a map
such that $w(x,\cdot )$ is continuous, strictly increasing, $w(x,0)=0$, for
any fixed $x$ in $M$. Assume that there exist positive constants $%
C_{1},C_{2} $ and $\beta _{2}\geq \beta _{1}$ such that for all $0<r\leq
R<\infty $ and all $x,y\in M$ with $d(x,y)\leq R$, 
\begin{equation}
C_{1}\left( \frac{R}{r}\right) ^{\beta _{1}}\leq \frac{w(x,R)}{w(y,r)}\leq
C_{2}\left( \frac{R}{r}\right) ^{\beta _{2}}.  \label{eq:vol_0}
\end{equation}%
For convenience, we write for any metric ball $B=B(x,R)$ 
\begin{equation*}
w(B):=w(x,R).
\end{equation*}%
Note that the symbol $w(B)$ is \emph{sensitive} to the center and radius of
ball $B$.

Denote the norm in $L^p:=L^p(M,\mu)\ (1\leq p<\infty)$ by 
\begin{equation*}
||u||_p:=\left(\int_M|u(x)|^p\mu(dx)\right)^{1/p},
\end{equation*}
and $||u||_{L^\infty}:=\esup_{x\in M}|u(x)|$, where $\esup$ is the essential
supremum.

Let $(\mathcal{E},\mathcal{F})$ be a regular Dirichlet form in $L^{2}$
without killing part, that is, 
\begin{equation}
\mathcal{E}(u,v)=\mathcal{E}^{(L)}(u,v)+\mathcal{E}^{(J)}(u,v),
\label{eq:DF-LJ}
\end{equation}%
where $\mathcal{E}^{(L)}$ is the \emph{local part} (or \emph{diffusion part}%
) and $\mathcal{E}^{(J)}$ is the \emph{jump part}. Let $\mathcal{F}_{\mathrm{%
loc}}$ be a space of all measurable functions $u$ on $M$ such that for every
precompact open subset $U$ of $M$, there exists \newline
some function $v\in \mathcal{F}$ such that $u=v$ for $\mu $-almost
everywhere in $U$. Then, there exists a unique Radon measure $d\Gamma
^{(L)}\langle u\rangle :=d\Gamma ^{(L)}\langle u,u\rangle $ such that 
\begin{equation*}
\mathcal{E}^{(L)}(u,u)=\int_{M}d\Gamma ^{(L)}\langle u\rangle
\end{equation*}%
for $u\in \mathcal{F}_{\mathrm{loc}}\cap L^{\infty }$, see for example \cite[%
Lemma 3.2.3, and the first two paragraphs on p.130]%
{FukushimaOshimaTakeda.2011.489}, wherein the symbols $\mathcal{E}^{(\mathrm{%
c})}=\mathcal{E}^{(L)}$ and $d\mu _{\langle u,u\rangle }^{(\mathrm{c}%
)}=2d\Gamma ^{(L)}\langle u,u\rangle $ are used instead. For the jump part,
there exists a unique Radon measure $J(dx,dy)$ defined on $M\times
M\setminus $diag such that 
\begin{equation}
\mathcal{E}^{(J)}(u,u)=\underset{M\times M\setminus \text{diag}}{\iint }%
(u(x)-u(y))^{2}J(dx,dy)  \label{eq:jump-DF}
\end{equation}%
for all continuous functions $u\in \mathcal{F}$ with compact supports on $M$%
. For simplicity, we let the measure $J=0$ on $\mathrm{diag}$ and will drop $%
\mathrm{diag}$ in expression $M\times M\setminus $diag in (\ref{eq:jump-DF})
when no confusion arises. In the sequel, set%
\begin{equation*}
\mathcal{E}(u):=\mathcal{E}(u,u)
\end{equation*}%
for convenience.

For any non-empty open subset $\Omega $ of $M$, let $C_{0}(\Omega )$ be a
space of all continuous functions with compact supports in $\Omega $. Denote
by $\mathcal{F}(\Omega )$ the closure of $\mathcal{F}\cap C_{0}(\Omega )$ in
the norm of 
\begin{equation*}
\sqrt{\mathcal{E}(\cdot ,\cdot )+(\cdot ,\cdot )}.
\end{equation*}%
Recall that for any non-empty open subset $\Omega $ of $M$, the form $(%
\mathcal{E},\mathcal{F}(\Omega ))$ is a regular Dirichlet form in $%
L^{2}(\Omega )$ if $(\mathcal{E},\mathcal{F})$ is regular. Let $\left\{
P_{t}^{\Omega }\right\} _{t\geq 0}$ be the heat semigroup associated with $(%
\mathcal{E},\mathcal{F}(\Omega ))$. Let 
\begin{equation*}
\mathcal{F}^{\prime }:=\{v+a:\ v\in \mathcal{F},\ a\in \mathbb{R}\}
\end{equation*}%
be a vector space that contains constant functions. We extend the domain of $%
\mathcal{E}$ to $\mathcal{F}^{\prime }$ as follows: for all $u,v\in \mathcal{%
F}$ and $a,b\in \mathbb{R}$, set 
\begin{equation*}
\mathcal{E}(u+a,v+b):=\mathcal{E}(u,v).
\end{equation*}%
We point that the extension is well defined by using (\ref{eq:DF-LJ}).

Let $U\Subset V$ (that means $U$ is precompact and the closure of $U$ is
contained in $V$) be two non-empty open subsets of $M$. We say that a
measurable function $\phi $ is a \emph{cutoff function} for $U\Subset V$,
denoted by $\phi \in \mathrm{cutoff}(U,V)$, if $\phi \in \mathcal{F}$, and 
\begin{align*}
\phi & =1\ \ \text{on}\ \ U, \\
\phi & =0\ \ \text{on}\ \ V^{c}, \\
0& \leq \phi \leq 1\ \ \text{on}\ \ M.
\end{align*}%
It is known that if $(\mathcal{E},\mathcal{F})$ is regular, the set $\mathrm{%
cutoff}(U,V)$ is non-empty for any two non-empty open subsets $U\Subset V$
of $M$.

We introduce conditions $(\mathrm{Gcap})$ and $(\mathrm{Cap}_{\leq })$.

\begin{definition}[condition $(\mathrm{Gcap})$]
We say that \emph{condition }$(\mathrm{Gcap})$ holds if for any $u\in 
\mathcal{F}^{\prime }\cap L^{\infty }$ and any two concentric metric balls $%
B_{0}:=B(x_{0},R)$, $B:=B(x_{0},R+r)$ with $0<R<R+r<\overline{R}$, there
exists some $\phi \in \mathrm{cutoff}(B_{0},B)$ such that 
\begin{equation}
\mathcal{E}(u^{2}\phi ,\phi )\leq \frac{C}{w(x_{0},r)}\int_{B}u^{2}d\mu ,
\label{eq1}
\end{equation}%
where $C>0$ is a constant independent of $u,B_{0},B$, but $\phi $ may depend
on $u$.
\end{definition}

\begin{definition}[condition $(\mathrm{Cap}_{\leq })$]
We say that \emph{condition }$(\mathrm{Cap}_{\leq })$ holds if there exists
a constant $C>0$ such that for all balls $B$ of radius $R$ less than $%
\overline{R}$ 
\begin{equation}
\mathrm{cap}((2/3)B,B)\leq C\frac{\mu (B)}{w(B)},  \label{eq2}
\end{equation}%
where the capacity $\mathrm{cap}(A,\Omega )$ for any two open subsets $%
A\Subset \Omega $ of $M$ is defined by 
\begin{equation*}
\mathrm{cap}(A,\Omega )\coloneqq\inf \{\mathcal{E}(\varphi ,\varphi ):\
\varphi \in \mathrm{cutoff}(A,\Omega )\}.
\end{equation*}
\end{definition}

Clearly, condition $(\mathrm{Gcap})$ implies condition $(\mathrm{Cap}_{\leq
})$ by taking $u=1$ in (\ref{eq1}) and by using the second inequality in (%
\ref{eq:vol_0}).

\begin{definition}[condition $(\mathrm{FK})$]
We say that condition $(\mathrm{FK})$ holds if there exist three positive
constants $C_{F}$, $\nu $ and $\sigma \in (0,1)$ such that for any ball $%
B:=B(x,r)$ with $0<r<\sigma \overline{R}$ and any non-empty open subset $%
D\subset B$, 
\begin{equation}
\lambda _{1}(D)\geq \frac{C_{F}^{-1}}{w(B)}\left( \frac{\mu (B)}{\mu (D)}%
\right) ^{\nu },  \label{eq:vol_4}
\end{equation}%
where $\lambda _{1}(D)$ is defined by 
\begin{equation*}
\lambda _{1}(D):=\inf_{u\in \mathcal{F}(D)\backslash \{0\}}\frac{\mathcal{E}%
(u,u)}{||u||_{2}^{2}}.
\end{equation*}
\end{definition}

\noindent Without loss of generality, we can assume that $0<\nu <1$ by
noting that $\frac{\mu (B)}{\mu (D)}\geq 1$.

\begin{definition}[condition $(\mathrm{PI})$]
We say that condition $(\mathrm{PI})$ holds if there exist two constants $%
\kappa $ $\geq 1$, $C$ $>0$ such that for any metric ball $B:=B(x_{0},r)$
with $0<r<\overline{R}/\kappa $ and any $u\in \mathcal{F}^{\prime }\cap
L^{\infty }$,%
\begin{equation}
\int_{B}(u-{u_{B}})^{2}d\mu \leq Cw(B)\left\{ \int_{\kappa B}d\Gamma
^{(L)}\langle u\rangle +\int_{(\kappa B)\times (\kappa
B)}(u(x)-u(y))^{2}J(dx,dy)\right\} ,  \label{eq:vol_5}
\end{equation}%
where $u_{B}$ is the average of the function $u$ over $B$, that is,%
\begin{equation*}
u_{B}=\frac{1}{\mu (B)}\int_{B}ud\mu \eqqcolon\fint_{B}ud\mu
\end{equation*}
\end{definition}

For a transition kernel $J(x,E)$ defined on $M\times \mathcal{B}(M)$ where $%
\mathcal{B}(M)$ is the collection of all Borel subsets of $M$, denote by%
\begin{equation}
J(x,E)\coloneqq\int_{E}J(x,dy).  \label{J1}
\end{equation}

We introduce condition $(\mathrm{TJ})$.

\begin{definition}[condition $(\mathrm{TJ})$]
We say that condition $(\mathrm{TJ})$ holds if there exists a transition
kernel $J(x,E)$ on $M\times \mathcal{B}(M)$ such that, for any point $x$ in $%
M$ and any $R>0$,%
\begin{equation}
\begin{aligned} &J(dx,dy) =J(x,dy)\mu (dx)\text{ \ and } \\ &J(x,B(x,R)^{c})
\leq \frac{C}{w(x,R)} \end{aligned}  \label{eq:vol_16}
\end{equation}%
for a non-negative constant $C$ independent of $x,R$.
\end{definition}

For an open subset $\Omega $ of $M$ and a function $f\in L^{2}(\Omega )$, we
say that a function $u\in \mathcal{F}$ is $f$-\emph{superharmonic} (resp. $f$%
-\emph{subharmonic}) in $\Omega $ if for any non-negative $\varphi \in 
\mathcal{F}(\Omega )$, 
\begin{equation}
\mathcal{E}(u,\varphi )\geq (f,\varphi )\text{ \ (resp. }\mathcal{E}%
(u,\varphi )\leq (f,\varphi )\text{).}  \label{f-har}
\end{equation}%
We say that a function $u\in \mathcal{F}$ is $f$-\emph{harmonic} in $\Omega $
if $u$ is both $f$-{superharmonic} and $f$-{subharmonic} in $\Omega $. If $%
f\equiv 0$, an $f$-superharmonic is shortened \emph{superharmonic}, and a
similar notion applies to an $f$-subharmonic or an $f$-harmonic.

\noindent For any two open subsets $U\Subset \Omega $ of $M$ and any
measurable function $v$, denote by 
\begin{equation}
T_{U,\Omega }(v)\coloneqq\esup_{x\in U}\int_{\Omega ^{c}}|v(y)|J(x,dy).
\label{tj2}
\end{equation}%
\noindent We introduce \emph{condition }$($\textrm{wEH}$)$, the \emph{weak
elliptic Harnack inequality.}

\begin{definition}[condition $($\textrm{wEH}$)$]
\label{df:wehi1}We say that condition $(\mathrm{wEH})$ holds if there exist
four universal constants $p,\delta ,\sigma $ in $(0,1)$ and $C_{H}\geq 1$
such that, for any two concentric balls $B_{r}:=B(x_{0},r)\subset
B_{R}:=B(x_{0},R)$ with $0<r\leq \delta R$, $R<\sigma \overline{R}$, any
function $f\in L^{\infty }(B_{R})$, and for any $u\in \mathcal{F}^{\prime
}\cap L^{\infty }$ that is non-negative, $f$-superharmonic in $B_{R}$, 
\begin{equation}
\left( \fint_{B_{r}}u^{p}d\mu \right) ^{1/p}\leq C_{H}\left( \einf_{{B}%
_{r}}u+w(B_{r})\left( T_{\frac{3}{4}B_{R},B_{R}}(u_{-})+\Vert f\Vert
_{L^{\infty }(B_{R})}\right) \right) ,  \label{eq:vol_65}
\end{equation}%
where $u_{-}:=0\vee (-u)$ is the negative part of function $u$, and $T_{%
\frac{3}{4}B_{R},B_{R}}$ is defined by (\ref{tj2}), that is,%
\begin{equation*}
T_{\frac{3}{4}B_{R},B_{R}}(u_{-})=\esup_{x\in \frac{3}{4}B_{R}}\int_{M%
\setminus B_{R}}u_{-}(y)J(x,dy).
\end{equation*}%
We remark that the constants $p,\delta ,\sigma ,C_{H}$ are all independent
of $x_{0},R,r,f$ and $u$.
\end{definition}

\begin{remark}
\label{R1}If $u$ is superharmonic, non-negative in $B_{R}$, then (\ref%
{eq:vol_65}) reads%
\begin{equation}
\left( \fint_{B_{r}}u^{p}d\mu \right) ^{1/p}\leq C_{H}\left( \einf_{{B}%
_{r}}u+w(B_{r})T_{\frac{3}{4}B_{R},B_{R}}(u_{-})\right) .  \label{12-1}
\end{equation}%
If the form $(\mathcal{E},\mathcal{F})$ is strongly local and $u$ is
harmonic, non-negative in $B_{R}$, then (\ref{eq:vol_65}) becomes%
\begin{equation}
\left( \fint_{B_{r}}u^{p}d\mu \right) ^{1/p}\leq C_{H}\einf_{{B}_{r}}u,
\label{12-2}
\end{equation}%
and in this situation, we in fact have that the weak Harnack inequality (\ref%
{12-2}) is equivalent to the strong Harnack inequality (\ref{11}), since the
inequality (\ref{12-2}) is equivalent to the following%
\begin{equation}
\einf_{B_{r}}u\geq a\exp \left( -\frac{C}{\omega _{B_{r}}(\{u\geq a\})}%
\right) \text{ \ for any }a>0  \label{12-3}
\end{equation}%
by using the equivalence $(\mathrm{wEH})\Leftrightarrow (\mathrm{wEH2})$ in
Theorem \ref{thm:M10} below where condition $(\mathrm{wEH2})$ will be stated
in Definition \ref{df:wehi3} and by using the fact that (\ref{12-3}) $%
\Rightarrow (\mathrm{H})$ in \cite[from Corollary 7.3 to Theorem 7.8 on
pages 1525-1535]{GrigoryanHuLau.2015.JMSJ1485}.
\end{remark}

The weak elliptic Harnack inequality says that for any function $u$, which
is non-negative and superharmonic in a ball $B_{R}$, its mean value over a
smaller concentric ball $B_{r}$ in the $L^{p}$-quantity (not a norm) for a
small $p\in (0,1)$, can be controlled by its essential infimum over the
smaller ball $B_{r}$, plus a tail estimate outside the ball $B_{R}$.

The main results of this paper are stated in Theorems \ref{thm:M1} and \ref%
{thm:M10} below.

\begin{theorem}
\label{thm:M1}Let $(\mathcal{E},\mathcal{F})$ be a regular Dirichlet form in 
$L^{2}(M,\mu )$ without killing part. Then%
\begin{equation}
(\mathrm{VD})+(\mathrm{RVD})+(\mathrm{Gcap})+(\mathrm{TJ})+(\mathrm{PI}%
)\Rightarrow (\mathrm{wEH}).  \label{20}
\end{equation}
\end{theorem}

We will prove Theorem \ref{thm:M1} at the end of Section \ref{sec:pf}. For
this, we need to show the following implications: 
\begin{eqnarray}
(\mathrm{VD})+(\mathrm{RVD})+(\mathrm{PI}) &\Rightarrow &(\mathrm{FK})\text{
\ (see Section \ref{sec:FK}),}  \label{21} \\
(\mathrm{VD})+(\mathrm{FK})+(\mathrm{Gcap})+(\mathrm{TJ}) &\Rightarrow &(%
\mathrm{LG})\text{ \ (see Section \ref{Sect-LG}),}  \label{21-1} \\
(\mathrm{VD})+(\mathrm{LG})+(\mathrm{Cap}_{\leq })+(\mathrm{PI})
&\Rightarrow &(\mathrm{wEH})\text{ (see Section \ref{sec:pf}),}  \label{22}
\end{eqnarray}%
where condition $(\mathrm{LG})$ is a refinement of the \emph{lemma of growth}
to be stated in Lemma \ref{lemma:LG} below.

We remark that if the metric space $(M,d)$ is unbounded and the scaling
function $w(x,r)$ is independent of point $x$, a similar implication to (\ref%
{21}) was obtained for strongly local Dirichlet forms (cf. \cite[Theorem $%
5.1 $]{GrigoryanHuLau.2015.JMSJ1485}), and for purely jump Dirichlet forms
(cf. \cite[Propositions $7.3$ and $7.4$]{ChenKumagaiWang.2016.Ae}). Here we
generalize this result to the case when the scaling function $w(x,\cdot )$
may depend on point $x$ and the metric space may be bounded or unbounded.

Our Theorem \ref{thm:M1} is an extension of a similar result in \cite[%
Theorem 3.1]{ChenKumagaiWang.2016.EH} in the sense that, instead of assuming
condition $(\mathrm{TJ})$ in this paper, the following stronger hypothesis
than condition $(\mathrm{TJ})$ was assumed in \cite{ChenKumagaiWang.2016.EH}%
: the jump kernel $J(x,y)$ exists and satisfies the following \emph{pointwise%
} upper estimate 
\begin{equation*}
J(x,y)\leq \frac{C}{V(x,d(x,y))w(x,d(x,y))}
\end{equation*}%
for $\mu \times \mu $-almost all $(x,y)$ in $M\times M\setminus \mathrm{diag}
$. Also the metric space $(M,d)$ considered in \cite{ChenKumagaiWang.2016.EH}
is assumed to be unbounded. We emphasize that we do not assume the jump
kernel $J(x,y)$ exists, neither the boundedness of the metric space.

Liu and Murugan \cite[Theorem 1.2]{LiuMurugan2020} show that the parabolic
Harnack inequality implies the existence of the jump kernel $J(x,y)$ for a
pure jump regular Dirichlet form. A natural question arises whether the weak
elliptic Harnack inequality also implies the existence of the jump kernel.
The answer is negative. In fact, the paper \cite[Section 15]%
{BendikovGrigoryanHu18} has given an example on the ultra-metric space where
the jump measure satisfies both conditions $(\mathrm{PI})$ and $(\mathrm{TJ}%
) $ (noting that condition $(\mathrm{Gcap})$ automatically holds since it
follows directly from condition $(\mathrm{TJ})$ and the ultra-metric
property), but the jump kernel does not exist. By Theorem \ref{thm:M1}
above, the weak elliptic Harnack inequality is true, however, the jump
kernel does not exist in this case. We will give the details in Section \ref%
{sect-1}.

Let us explain the idea of proving the weak elliptic Harnack inequality in
Theorem \ref{thm:M1}. The proof essentially consists of the following two
steps (under the case when $f\equiv 0$).

\begin{enumerate}
\item To obtain the so-called \emph{measure-to-point lemma} as follows: for
some $\varepsilon \in (0,1)$ and for any non-negative superharmonic function 
$u$ in a ball $B$, there exists a constant $\eta >0$, depending only on $%
\varepsilon $ but independent of the ball $B$ and the function $u$, such
that 
\begin{equation}
\frac{\mu (B\cap \{u>1\})}{\mu (B)}\geq \varepsilon \text{ \ }\Rightarrow 
\text{ \ }\inf_{\frac{1}{2}B}u\geq \eta .  \label{eq:vol_20}
\end{equation}

\item To obtain the so-called \emph{crossover lemma} as follows: there exist
three universal numbers $p,\delta $ in $(0,1)$ and $C>0$ such that, for any
non-negative superharmonic function $u$ in a ball $B_{R}$, any concentric
ball $B_{r}$ of $B_{R}$ with $0<r\leq \delta R$ and for any positive number $%
\lambda \geq w(B_{R})T_{\frac{3}{4}B_{R},B_{R}}(u_{-})$, 
\begin{equation}
\left( \fint_{B_{r}}(u+\lambda )^{p}d\mu \right) ^{1/p}\left( \fint%
_{B_{r}}(u+\lambda )^{-p}d\mu \right) ^{1/p}\leq C.  \label{61}
\end{equation}
\end{enumerate}

\noindent The implication (\ref{eq:vol_20}) says that, if the occupation 
\emph{measure} of a superlevel set 
\begin{equation*}
\{u\geq a\}\text{ \ for }a>0
\end{equation*}%
in a ball $B$ for a function $u$, which is non-negative, superharmonic in $B$%
, is bounded from below by a constant $\varepsilon $, then the function $u$
should be also bounded from below by a positive number $\eta a$ at almost
all \emph{points} near the center.

The measure-to-point lemma is essentially the same as the \emph{Lemma of
growth} introduced by Landis in \cite{Landis.1963.UMN3}, \cite%
{Landis.1998.203} in studying solutions of elliptic second order PDEs (local
Dirichlet form) in $\mathbb{R}^{n}$. This Lemma of growth has been
reformulated and extended to the case for pure jump type (non-local)
Dirichlet forms on the metric measure space in \cite[Lemma 4.1]%
{GrigoryanHuHu.2018.AM433}, see also a forthcoming paper \cite{ghh21TP} for
mixed (either local or non-local) Dirichlet forms defined by (\ref{eq:DF-LJ}%
) without killing part (cf. Lemma \ref{lemma:LG} below). An alternative
version of Lemma of growth for pure jump type (non-local) Dirichlet forms on
metric space was stated in \cite[Proposition 3.6]{ChenKumagaiWang.2016.EH}.

We remark that the measure-to-point lemma is originated from the work by
Moser \cite[Theorem 2]{Moser.1960.CPAM457}, and developed by Krylov and
Safonov \cite{KrylovSafonov.1979.DANS18}, \cite{KRYLOVSAFONOV.1980.IANS161}, 
\cite{Safonov.1983.ZNS272}. The reader may consult the reference \cite[%
Section 3]{MaldonadoDiego.2017.PAMS3981} for the classical case.

Once the measure-to-point lemma has been established, one needs further to
show the crossover lemma (\ref{61}), where the Poincar\'{e} inequality comes
into a stage. To achieve (\ref{61}), one needs to show that, for any $u\in 
\mathcal{F^{\prime }}\cap L^{\infty }$ that is superharmonic and
non-negative in a ball $B$ and for any positive number $\lambda $ bounded
from below by a tail (in the case of local Dirichlet forms, any number $%
\lambda >0$ will be fine), the logarithm function 
\begin{equation*}
\ln ({u+\lambda )}
\end{equation*}%
belongs to the space $\mathrm{BMO}(\delta B)$, for some number $\delta \in
(0,1)$ that is independent of $u,\lambda $ and ball $B$. After that, the
rest of the proof is standard: one makes use of Lemma \ref{cor:JN} in
Appendix for an exponential function%
\begin{equation*}
\exp \left( \frac{c}{b}g\right) \text{ \ for any }b\geq ||g||_{\mathrm{BMO}}
\end{equation*}%
for $g:=\ln ({u+\lambda )}$, which is valid from the John-Nirenberg
inequality (see Lemma \ref{lemma:JN} in Appendix), and we are eventually led
to the desired crossover lemma (\ref{61})\ (see Lemma \ref{thm:M4} below).

Besides the version of the weak elliptic Harnack inequality stated in
Definition \ref{df:wehi1}, there are several other versions in the
literature, see for example \cite[Proposition 3.6]{ChenKumagaiWang.2016.EH}, 
\cite[Lemma 7.2]{GrigoryanHuLau.2015.JMSJ1485}, \cite[Lemma 4.5]%
{GrigoryanHuHu.2018.AM433}. We list all of them in Section \ref{sec:wEHI}
and term as conditions $(\mathrm{wEH}1)$, $(\mathrm{wEH}2)$, $(\mathrm{wEH}%
3) $, $(\mathrm{wEH}4)$. We shall show that the first three conditions $(%
\mathrm{wEH}1)$, $(\mathrm{wEH}2)$, $(\mathrm{wEH}3)$ are equivalent one
another, each of which implies condition $(\mathrm{wEH}4)$.

\begin{theorem}
\label{thm:M10}Let $(\mathcal{E},\mathcal{F})$ be a Dirichlet form in $%
L^{2}(M,\mu )$. If condition $(\mathrm{VD})$ holds, then 
\begin{eqnarray}
(\mathrm{wEH}) &\Leftrightarrow &(\mathrm{wEH}1)\Leftrightarrow (\mathrm{wEH}%
2)\Leftrightarrow (\mathrm{wEH}3)  \label{56} \\
&\Rightarrow &(\mathrm{wEH}4).  \label{57}
\end{eqnarray}
\end{theorem}

We will prove Theorem \ref{thm:M10} at the end of Section \ref{sec:wEHI}.

\section[FK]{Faber-Krahn inequality and Dirichlet heat kernel}

\label{sec:FK}In this section, we show that for a regular Dirichlet form
without killing part on a metric space, if the measure satisfies conditions $%
(\mathrm{VD})$ and $(\mathrm{RVD})$, then the Poincar\'{e} inequality
implies the Faber-Krahn inequality. Although this conclusion is known to the
expert, there is no a direct proof in the literature, and we will give a
self-contained proof for convenience. Here we do not assume the existence of
the jump kernel, neither the independence of point for the scaling function $%
w$. Our result can be viewed as an extension of the previous work \cite[%
Theorem 5.1]{GrigoryanHuLau.2015.JMSJ1485} for a local Dirichlet form for
the doubling measure, and \cite[Lemmas 5.2, 5.3]{BendikovGrigoryanHu18} for
a non-local Dirichlet form for the Ahlfors-regular measure. See also \cite[%
Proposition 3.4.1]{BoutayebCoulhonSikora.2015.AM302}. As a by-product, we
derive that the Dirichlet heat kernel $p_{t}^{B}(x,y)$ exists and satisfies
an upper bound, for any ball $B$ of radius less than $\sigma \overline{R}$.

We introduce condition $(\mathrm{Nash}_{B})$, which is the \emph{Nash
inequality on a ball }$B$.

\begin{definition}[condition $(\mathrm{Nash}_{B})$]
We say that condition $(\mathrm{Nash}_{B})$ holds if there exist three
positive constants $\sigma \in (0,1)$ and $\nu ,C$ such that for any metric
ball $B$ of radius $r\in (0,\sigma \overline{R})$ and any $u\in \mathcal{F}%
(B)$, 
\begin{equation}
||u||_{2}^{2+2\nu }\leq \frac{C}{\mu (B)^{\nu }}||u||_{1}^{2\nu }\left(
||u||_{2}^{2}+w(B)\mathcal{E}(u,u)\right) .  \label{eq:vol_6}
\end{equation}%
We remark that constants $C$ and $\nu ,\sigma $ are all independent of ball $%
B$ and function $u$.
\end{definition}

We show that the Poincar\'{e} inequality implies the Nash inequality on a
ball.

\begin{lemma}
\label{thm:M2}Assume that $(\mathcal{E},\mathcal{F})$ is a regular Dirichlet
form in $L^{2}(M,\mu )$ without killing part. If conditions $(\mathrm{VD})$
and $(\mathrm{PI})$ are satisfied, then condition $(\mathrm{Nash}_{B})$
holds, that is, 
\begin{equation*}
(\mathrm{VD})+(\mathrm{PI})\Rightarrow (\mathrm{Nash}_{B}).
\end{equation*}
\end{lemma}

\begin{proof}
Since the proof is quite long, we divided into two steps.

\emph{Step} $1$. We show that there exists a constant $C>0$ such that for
all $s>0$ and all $u\in \mathcal{F}\cap L^{1}$ with $||u||_{1}>0$ 
\begin{equation}
||u_{s}||_{2}^{2}\leq \frac{C||u||_{1}^{2}}{\underset{z\in \mathrm{supp\,}(u)%
}{\inf }V(z,s)},  \label{eq:vol_10}
\end{equation}%
where $u_{s}(x)$ is the average of function $u$ over a ball $B(x,s)$, that
is, 
\begin{equation*}
u_{s}(x)=\frac{1}{V(x,s)}\int_{B(x,s)}u(z)\mu (dz)\ \text{\ for}\ x\in M,s>0.
\end{equation*}%
The proof is motivated by \cite[Theorem 2.4]{Saloff-Coste.1992.IMRN27}. At
this step, we do not need condition $(\mathrm{PI})$. To this end, let $%
||u||_{1}>0$, and denote by 
\begin{equation*}
A_{s}:=\{x\in M:d(x,\mathrm{supp\,}(u))<s\},
\end{equation*}%
the $s$-neighborhood of the support of $u$. Clearly, we see that $%
u_{s}(x)\equiv 0$ when $x$ lies outside the set $A_{s}$, since $u(z)=0$ for $%
z\in B(x,s)\subset M\setminus \mathrm{supp\,}(u)$. It follows that 
\begin{equation}
||u_{s}||_{\infty }\leq \frac{||u||_{1}}{\underset{x\in A_{s}}{\inf }V(x,s)}%
\leq \frac{2^{d_{2}}C_{\mu }||u||_{1}}{\underset{x\in \mathrm{supp\,}(u)}{%
\inf }V(x,s)},  \label{eq:vol_8}
\end{equation}%
where we have used the fact that for any $x\in A_{s}$, 
\begin{equation*}
\frac{1}{\underset{x\in A_{s}}{\inf }V(x,s)}\leq \frac{2^{d_{2}}C_{\mu }}{%
\underset{x\in \mathrm{supp\,}(u)}{\inf }V(x,s)},
\end{equation*}%
since there exists a point $z\in \mathrm{supp\,}(u)$ such that $d(z,x)<s$,
and thus by (\ref{eq:vol_3}) 
\begin{equation}
\frac{V(z,s)}{V(x,s)}\leq C_{\mu }\left( \frac{d(z,x)+s}{s}\right)
^{d_{2}}\leq C_{\mu }\left( \frac{s+s}{s}\right) ^{d_{2}}=2^{d_{2}}C_{\mu },
\label{13}
\end{equation}%
from which,%
\begin{equation*}
\frac{1}{\underset{x\in A_{s}}{\inf }V(x,s)}=\underset{x\in A_{s}}{\sup }%
\frac{1}{V(x,s)}\leq 2^{d_{2}}C_{\mu }\underset{z\in \mathrm{supp\,}(u)}{%
\sup }\frac{1}{V(z,s)}=\frac{2^{d_{2}}C_{\mu }}{\underset{x\in \mathrm{supp\,%
}(u)}{\inf }V(x,s)}.
\end{equation*}

On the other hand, 
\begin{align}
||u_{s}||_{1}& \leq \int_{A_{s}}\frac{1}{V(x,s)}\left(
\int_{B(x,s)}|u(z)|\mu (dz)\right) \mu (dx)  \notag \\
& =\int_{A_{s}}\frac{1}{V(x,s)}\left( \int_{\mathrm{supp\,}%
(u)}|u(z)|1_{B(x,s)}(z)\mu (dz)\right) \mu (dx)  \notag \\
& =\int_{\mathrm{supp\,}(u)}|u(z)|\left( \int_{A_{s}}\frac{1_{B(x,s)}(z)}{%
V(x,s)}\mu (dx)\right) \mu (dz)  \notag \\
& =\int_{\mathrm{supp\,}(u)}|u(z)|\left( \int_{A_{s}\cap B(z,s)}\frac{1}{%
V(x,s)}\mu (dx)\right) \mu (dz)  \notag \\
& \leq \int_{\mathrm{supp\,}(u)}|u(z)|\frac{V(z,s)}{\underset{x\in B(z,s)}{%
\inf }V(x,s)}\mu (dz)  \notag \\
& =\int_{\mathrm{supp\,}(u)}|u(z)|\underset{x\in B(z,s)}{\sup }\frac{V(z,s)}{%
V(x,s)}\mu (dz)\leq 2^{d_{2}}C_{\mu }||u||_{1},  \label{eq:vol_76}
\end{align}%
since for any $z\in \mathrm{supp\,}(u)$ and any $x\in B(z,s)$, 
\begin{equation*}
\frac{V(z,s)}{V(x,s)}\leq 2^{d_{2}}C_{\mu }
\end{equation*}%
by virtue of (\ref{13}). Therefore, it follows from (\ref{eq:vol_8}), (\ref%
{eq:vol_76}) that 
\begin{equation*}
||u_{s}||_{2}^{2}\leq ||u_{s}||_{\infty }||u_{s}||_{1}\leq \frac{%
(2^{d_{2}}C_{\mu })^{2}||u||_{1}^{2}}{\underset{z\in \mathrm{supp\,}(u)}{%
\inf }V(z,s)},
\end{equation*}%
thus showing (\ref{eq:vol_10}) with $C:=(2^{d_{2}}C_{\mu })^{2}$.

\emph{Step} $2$. We show that condition $(\mathrm{Nash}_{B})$ holds. We
assume that condition $(\mathrm{PI})$ holds.

Fix a ball $B:=B(x_{0},r)$ with $r\in (0,\frac{\overline{R}}{\kappa })$,
where constant $\kappa $ is the same as in condition $(\mathrm{PI})$. Let $%
s\in (0,\frac{\overline{R}}{2\kappa })$ be a number to be determined later
on, and fix a function $u\in \mathcal{F}(B)\cap L^{1}(M,\mu )$. Since $M$ is
separable, there is a countable family of points $\{y_{i}\}_{i=1}^{\infty }$
such that $M\subset \bigcup_{i=1}^{\infty }B(y_{i},s)$. By the doubling
property, we can find a subsequence $\{x_{i}\}_{i=1}^{\infty }\subset
\{y_{i}\}_{i=1}^{\infty }$ such that $M=\bigcup_{i=1}^{\infty }B_{i}$ with $%
B_{i}:=B(x_{i},s)$, and $\{\frac{1}{5}B_{i}\}_{i=1}^{\infty }$ are pairwise
disjoint (see \cite[Theorem 1.16]{Heinonen.2001.140}). The over-lapping
number $\sum_{i=1}^{\infty }1_{2\kappa B_{i}}$ is bounded by some integer $%
N_{0}$ depending only on $\kappa $ and $C_{\mu }$, that is, $%
\sum_{i=1}^{\infty }1_{2\kappa B_{i}}\leq N_{0}$. From this, we have for any
measurable function $g\geq 0$, 
\begin{align}
\sum_{i=1}^{\infty }\iint_{(2\kappa B_{i})\times M}g(x,y)J(dx,dy)&
=\iint_{M\times M}g(x,y)\sum_{i=1}^{\infty }1_{2\kappa B_{i}}(x)J(dx,dy) 
\notag \\
& \leq N_{0}\iint_{M\times M}g(x,y)J(dx,dy).  \label{eq:vol_9}
\end{align}%
We estimate the term $||u-u_{s}||_{2}^{2}$ by 
\begin{align}
||u-u_{s}||_{2}^{2}& \leq \sum_{i=1}^{\infty
}\int_{B_{i}}|u(x)-u_{s}(x)|^{2}\mu (dx)  \notag \\
& \leq 2\sum_{i=1}^{\infty }\left(
\int_{B_{i}}(|u(x)-u_{2B_{i}}|^{2}+|u_{2B_{i}}-u_{s}(x)|^{2})\mu (dx)\right) %
\eqqcolon2(I_{1}+I_{2}).  \label{eq:vol_11}
\end{align}%
For $I_{1}$, we have by condition $(\mathrm{PI})$, 
\begin{align}
I_{1}& =\sum_{i=1}^{\infty }\int_{B_{i}}|u(x)-u_{2B_{i}}|^{2}\mu (dx)\leq
\sum_{i=1}^{\infty }\int_{2B_{i}}|u(x)-u_{2B_{i}}|^{2}\mu (dx)  \notag \\
& \leq C\sum_{i=1}^{\infty }w(x_{i},2s)\left\{ \int_{2\kappa B_{i}}d\Gamma
^{(L)}\langle u\rangle +\iint_{(2\kappa B_{i})\times (2\kappa
B_{i})}(u(x)-u(y))^{2}J(dx,dy)\right\} .  \label{eq:vol_12}
\end{align}%
For $I_{2}$, note that for any $x\in B_{i}=B(x_{i},s)$, the function $%
1_{B(x,s)}(z)=0$ when $z\in (2B_{i})^{c}\subset B(x,s)^{c}$. Using the
Cauchy-Schwarz inequality and condition $(\mathrm{PI})$, we have for any $%
x\in B_{i}$ 
\begin{align}
|u_{s}(x)-u_{2B_{i}}|^{2}& =\Big|\int_{M}\frac{1_{B(x,s)}(z)}{V(x,s)}%
(u(z)-u_{2B_{i}})\mu (dz)\Big|^{2}\leq \int_{M}\frac{1_{B(x,s)}(z)}{V(x,s)}%
|u(z)-u_{2B_{i}}|^{2}\mu (dz)  \notag \\
& \leq \int_{2B_{i}}\frac{1}{V(x,s)}|u(z)-u_{2B_{i}}|^{2}\mu (dz)\leq \frac{%
2^{d_{2}}C_{\mu }}{V(x_{i},s)}\int_{2B_{i}}|u(z)-u_{2B_{i}}|^{2}\mu (dz) 
\notag \\
& \leq \frac{Cw(x_{i},2s)}{V(x_{i},s)}\left\{ \int_{2\kappa B_{i}}d\Gamma
^{(L)}\langle u\rangle +\iint_{(2\kappa B_{i})\times (2\kappa
B_{i})}(u(x)-u(y))^{2}J(dx,dy)\right\} ,  \label{eq:vol_73}
\end{align}%
where we have used the fact that for any $x\in B_{i}$, 
\begin{equation*}
\frac{V(x_{i},s)}{V(x,s)}\leq C_{\mu }\left( \frac{d(x_{i},x)+s}{s}\right)
^{d_{2}}\leq 2^{d_{2}}C_{\mu }
\end{equation*}%
by virtue of (\ref{eq:vol_3}). Therefore, it follows that 
\begin{align}
I_{2}& =\sum_{i=1}^{\infty }\int_{B_{i}}(u_{2B_{i}}-u_{s}(x))^{2}\mu (dx) 
\notag \\
& \leq \sum_{i=1}^{\infty }\int_{B_{i}}\frac{Cw(x_{i},2s)}{V(x_{i},s)}%
\left\{ \int_{2\kappa B_{i}}d\Gamma ^{(L)}\langle u\rangle +\iint_{(2\kappa
B_{i})\times (2\kappa B_{i})}(u(x)-u(y))^{2}J(dx,dy)\right\} \mu (dx)  \notag
\\
& =C\sum_{i=1}^{\infty }w(x_{i},2s)\left\{ \int_{2\kappa B_{i}}d\Gamma
^{(L)}\langle u\rangle +\iint_{(2\kappa B_{i})\times (2\kappa
B_{i})}(u(x)-u(y))^{2}J(dx,dy)\right\} .  \label{eq:vol_13}
\end{align}%
Combining (\ref{eq:vol_12}) and (\ref{eq:vol_13}), we conclude from (\ref%
{eq:vol_11}) that 
\begin{align}
||u-u_{s}||_{2}^{2}& \leq 2(I_{1}+I_{2})  \notag \\
& \leq C\sum_{i=1}^{\infty }w(x_{i},2s)\left\{ \int_{2\kappa B_{i}}d\Gamma
^{(L)}\langle u\rangle +\iint_{(2\kappa B_{i})\times (2\kappa
B_{i})}(u(x)-u(y))^{2}J(dx,dy)\right\}   \label{eq:vol_74}
\end{align}%
for a positive constant $C$ depending only on the constants from condition $(%
\mathrm{VD})$ (independent of $u,s,B_{i}$).

Since $u\in \mathcal{F}(B)$, if $2\kappa B_{i}\subset B^{c}$, we see that $%
u(x)=u(y)=0$ when $x,y\in 2\kappa B_{i}$, thus $1_{2\kappa B_{i}}d\Gamma
^{(L)}\langle u\rangle =0$, and so the integral in the above summation
vanishes. In other words, the summation in (\ref{eq:vol_74}) is taken only
over the indices $i$ such that $(2\kappa B_{i})\cap B\neq \emptyset $. Set 
\begin{equation}
Q:=\sup_{i:(2\kappa B_{i})\cap B\neq \emptyset }w(x_{i},s).
\label{eq:vol_75}
\end{equation}%
Since $w(x_{i},2s)\leq C_{2}2^{\beta _{2}}w(x_{i},s)$ by using (\ref%
{eq:vol_0}), we obtain that 
\begin{align}
& \sum_{i=1}^{\infty }w(x_{i},s)\iint_{(2\kappa B_{i})\times (2\kappa
B_{i})}(u(x)-u(y))^{2}J(dx,dy)  \notag \\
& \leq Q\sum_{i=1}^{\infty }\iint_{(2\kappa B_{i})\times
M}(u(x)-u(y))^{2}J(dx,dy)  \notag \\
& \leq Q\cdot N_{0}\iint_{M\times M}(u(x)-u(y))^{2}J(dx,dy)\text{ \ (using (%
\ref{eq:vol_9}))}  \notag \\
& =N_{0}Q\mathcal{E}^{(J)}(u,u).  \label{eq:vol_14}
\end{align}%
On the other hand, 
\begin{align}
\sum_{i=1}^{\infty }w(x_{i},s)\int_{2\kappa B_{i}}d\Gamma ^{(L)}\langle
u\rangle & \leq Q\sum_{i=1}^{\infty }\int_{2\kappa B_{i}}d\Gamma
^{(L)}\langle u\rangle =Q\int_{M}\sum_{i=1}^{\infty }1_{2\kappa
B_{i}}d\Gamma ^{(L)}\langle u\rangle  \notag \\
& \leq N_{0}Q\int_{M}d\Gamma ^{(L)}\langle u\rangle =N_{0}Q\mathcal{E}%
^{(L)}(u,u).  \label{eq:vol_17}
\end{align}%
Therefore, combining (\ref{eq:vol_14}) and (\ref{eq:vol_17}), we conclude
from (\ref{eq:vol_74}) that for all $s$ $\in (0,\frac{\overline{R}}{2\kappa }%
)$ 
\begin{equation}
||u-u_{s}||_{2}^{2}\leq C\left\{ N_{0}Q\mathcal{E}^{(L)}(u,u)+N_{0}Q\mathcal{%
E}^{(J)}(u,u)\right\} =CN_{0}Q\mathcal{E}(u).  \label{eq:vol_18}
\end{equation}%
It is left to estimate $Q$ for any $s\in (0,\frac{\overline{R}}{2\kappa })$.
We distinguish two cases when $s\leq r$ or not.

Indeed, let $z_{0}\in (2\kappa B_{i})\cap B$. By (\ref{eq:vol_0}), we have 
\begin{equation*}
\frac{w(x_{i},s)}{w(z_{0},s)}=\frac{w(x_{i},s)}{w(x_{i},2\kappa s)}\cdot 
\frac{w(x_{i},2\kappa s)}{w(z_{0},s)}\leq C_{1}^{-1}\left( \frac{s}{2\kappa s%
}\right) ^{\beta _{1}}\cdot C_{2}\left( \frac{2\kappa s}{s}\right) ^{\beta
_{2}}=c^{\prime }(\kappa ),
\end{equation*}%
whilst for $s\leq r$ 
\begin{equation*}
\frac{w(z_{0},s)}{w(x_{0},r)}\leq C_{1}^{-1}\left( \frac{s}{r}\right)
^{\beta _{1}}.
\end{equation*}%
Thus, 
\begin{equation*}
\frac{w(x_{i},s)}{w(x_{0},r)}=\frac{w(x_{i},s)}{w(z_{0},s)}\cdot \frac{%
w(z_{0},s)}{w(x_{0},r)}\leq c\left( \frac{s}{r}\right) ^{\beta _{1}}
\end{equation*}%
if $(2\kappa B_{i})\cap B\neq \emptyset $ and $s\leq r$. From this, we
obtain 
\begin{equation}
Q=\sup_{i:(2\kappa B_{i})\cap B\neq \emptyset }w(x_{i},s)\leq c^{\prime
}\left( \frac{s}{r}\right) ^{\beta _{1}}w(x_{0},r)\ \ \text{if}\ s\leq r.
\label{eq:vol_77}
\end{equation}%
Plugging (\ref{eq:vol_77}) into (\ref{eq:vol_18}), we have 
\begin{equation}
||u-u_{s}||_{2}^{2}\leq C^{\prime }\left( \frac{s}{r}\right) ^{\beta
_{1}}w(x_{0},r)\mathcal{E}(u)  \label{eq:vol_78}
\end{equation}%
if $s\leq r$. Note that if $s\leq r$, then for any $x\in \mathrm{supp\,}%
(u)\subset B(x_{0},r)$ 
\begin{equation*}
\frac{V(x_{0},r)}{V(x,s)}\leq C_{\mu }\left( \frac{d(x_{0},x)+r}{s}\right)
^{d_{2}}\leq 2^{d_{2}}C_{\mu }\left( \frac{r}{s}\right) ^{d_{2}},
\end{equation*}%
which gives by (\ref{eq:vol_10}) that 
\begin{equation}
||u_{s}||_{2}^{2}\leq \frac{C||u||_{1}^{2}}{\underset{x\in \mathrm{supp\,}(u)%
}{\inf }V(x,s)}\leq C\left( \frac{r}{s}\right) ^{d_{2}}\frac{||u||_{1}^{2}}{%
V(x_{0},r)}.  \label{eq:vol_79}
\end{equation}%
Therefore, we conclude from (\ref{eq:vol_78}), (\ref{eq:vol_79}) that for
all $0<s$ $<r\wedge \frac{\overline{R}}{2\kappa }$, 
\begin{align}
||u||_{2}^{2}& \leq 2\left( ||u-u_{s}||_{2}^{2}+||u_{s}||_{2}^{2}\right) 
\notag \\
& \leq C\left( \left( \frac{s}{r}\right) ^{\beta _{1}}w(x_{0},r)\mathcal{E}%
(u)+\left( \frac{r}{s}\right) ^{d_{2}}\frac{||u||_{1}^{2}}{V(x_{0},r)}%
\right) .  \label{eq:vol_80}
\end{align}%
On the other hand, if $r\leq s<\frac{\overline{R}}{2\kappa }$, it is clear
that 
\begin{equation}
||u||_{2}^{2}\leq \left( \frac{s}{r}\right) ^{\beta _{1}}||u||_{2}^{2}.
\label{eq:vol_81}
\end{equation}%
Summing up (\ref{eq:vol_80}) and (\ref{eq:vol_81}), we obtain for all $0<s$ $%
<\frac{\overline{R}}{2\kappa }$, 
\begin{align}
||u||_{2}^{2}& \leq C\left( \left( \frac{s}{r}\right) ^{\beta _{1}}\left(
w(x_{0},r)\mathcal{E}(u)+||u||_{2}^{2}\right) +\left( \frac{r}{s}\right)
^{d_{2}}\frac{||u||_{1}^{2}}{V(x_{0},r)}\right)  \notag \\
& \leq C2^{d_{2}}\left( \left( \frac{2s}{r}\right) ^{\beta _{1}}\left(
w(x_{0},r)\mathcal{E}(u)+||u||_{2}^{2}\right) +\left( \frac{r}{2s}\right)
^{d_{2}}\frac{||u||_{1}^{2}}{V(x_{0},r)}\right) .  \label{eq:vol_82}
\end{align}%
We minimize the right-hand side of (\ref{eq:vol_82}) in $s$ $\in (0,\frac{%
\overline{R}}{2\kappa })$, for example, by choosing $s$ such that 
\begin{equation*}
\left( \frac{2s}{r}\right) ^{\beta _{1}}\left( w(x_{0},r)\mathcal{E}%
(u)+||u||_{2}^{2}\right) =\left( \frac{r}{2s}\right) ^{d_{2}}\frac{%
||u||_{1}^{2}}{V(x_{0},r)},
\end{equation*}%
that is, 
\begin{equation}
s=\frac{r}{2}\left( \frac{||u||_{1}^{2}}{V(x_{0},r)\left( w(x_{0},r)\mathcal{%
E}(u)+||u||_{2}^{2}\right) }\right) ^{\frac{1}{\beta _{1}+d_{2}}}.
\label{14}
\end{equation}%
We postpone verifying that $s$ $\in (0,\frac{\overline{R}}{2\kappa })$.
Therefore, it follows that 
\begin{equation*}
||u||_{2}^{2}\leq C^{\prime }\left( \frac{||u||_{1}^{2}}{V(x_{0},r)}\right)
^{\frac{\beta _{1}}{\beta _{1}+d_{2}}}\left( ||u||_{2}^{2}+w(x_{0},r)%
\mathcal{E}(u)\right) ^{\frac{d_{2}}{\beta _{1}+d_{2}}},
\end{equation*}%
thus showing that 
\begin{equation*}
||u||_{2}^{2(1+\frac{\beta _{1}}{d_{2}})}\leq C\left(
||u||_{2}^{2}+w(x_{0},r)\mathcal{E}(u)\right) \left( \frac{||u||_{1}^{2}}{%
V(x_{0},r)}\right) ^{\frac{\beta _{1}}{d_{2}}},
\end{equation*}%
for all $u\in \mathcal{F}(B)\cap L^{1}$. Hence, condition $(\mathrm{Nash}%
_{B})$\ holds with $\sigma =\frac{1}{\kappa }$ and $\nu =\frac{\beta _{1}}{%
d_{2}}$.

It remains to verify that the number $s$ given by (\ref{14}) satisfies
condition $s\in (0,\frac{\overline{R}}{2\kappa })$. Indeed, by the
Cauchy-Schwarz inequality, we have for any $u\in \mathcal{F}(B)$, 
\begin{equation*}
||u||_{1}^{2}\leq V(x_{0},r)||u||_{2}^{2},
\end{equation*}%
from which, we see that, using the fact that $r\in (0,\overline{R}/\kappa )$%
, 
\begin{equation*}
s=\frac{r}{2}\left( \frac{||u||_{1}^{2}}{V(x_{0},r)\left( w(x_{0},r)\mathcal{%
E}(u)+||u||_{2}^{2}\right) }\right) ^{\frac{1}{\beta _{1}+d_{2}}}\leq \frac{r%
}{2}\left( \frac{||u||_{1}^{2}}{V(x_{0},r)||u||_{2}^{2}}\right) ^{\frac{1}{%
\beta _{1}+d_{2}}}\leq \frac{r}{2}<\frac{\overline{R}}{2\kappa }.
\end{equation*}%
The proof is complete.
\end{proof}

We derive the on-diagonal upper bound of the Dirichlet heat kernel on any
ball by using condition $(\mathrm{Nash}_{B})$. In particular, we derive the
Faber-Krahn inequality.

\begin{lemma}
\label{thm:M3}Let $(\mathcal{E},\mathcal{F})$ be a regular Dirichlet form in 
$L^{2}$. If conditions $(\mathrm{VD})$, $(\mathrm{RVD})$ and $(\mathrm{Nash}%
_{B})$ hold, then the Dirichlet heat kernel $p_{t}^{B}(x,y)$ exists and
satisfies 
\begin{equation}
\esup_{{x,y\in B}}p_{t}^{B}(x,y)\leq \frac{C}{\mu (B)}\left( \frac{w(B)}{t}%
\right) ^{1/\nu }\ \text{for\ all}\ t>0  \label{eq:vol_27}
\end{equation}%
for any ball $B$ of radius $r<\frac{\sigma \overline{R}}{A}$, where $C,A$
are two universal constants independent of $B,t$ and constant $\sigma $
comes from condition $(\mathrm{Nash}_{B})$. Consequently, we have%
\begin{equation*}
(\mathrm{VD})+(\mathrm{RVD})+(\mathrm{Nash}_{B})\Rightarrow (\mathrm{FK}).
\end{equation*}
\end{lemma}

\begin{proof}
Assume that $A>1$ is a number to be chosen, see (\ref{aa}) below. Let $%
B:=B(x_{0},r)$ with%
\begin{equation}
0<r<\frac{\sigma \overline{R}}{A}.  \label{rr1}
\end{equation}%
Since $Ar$ is less than $\sigma \overline{R}$, we can apply condition $(%
\mathrm{Nash}_{B})$ on a ball $B(x_{0},Ar)$ and obtain for any $u\in 
\mathcal{F}(B(x_{0},Ar))$ 
\begin{equation}
||u||_{2}^{2+2\nu }\leq \frac{C||u||_{1}^{2\nu }}{V(x_{0},Ar)^{\nu }}\left(
||u||_{2}^{2}+w(x_{0},Ar)\mathcal{E}(u)\right) .  \label{eq:vol_21}
\end{equation}%
Note that for all $u\in \mathcal{F}(B(x_{0},r))$, 
\begin{equation}
||u||_{1}=\int_{B(x_{0},r)}|u|d\mu \leq {V(x_{0},r)}^{1/2}||u||_{2}.
\label{eq:vol_15}
\end{equation}%
Since $\mathcal{F}(B(x_{0},r))\subset \mathcal{F}(B(x_{0},Ar))$ for any $A>1$%
, it follows (\ref{eq:vol_15}), (\ref{eq:vol_21}) that for any $u\in 
\mathcal{F}(B(x_{0},r))$ 
\begin{eqnarray*}
||u||_{2}^{2+2\nu } &\leq &\frac{C\left( {V(x_{0},r)}^{1/2}||u||_{2}\right)
^{2\nu }}{V(x_{0},Ar)^{\nu }}||u||_{2}^{2}+\frac{C||u||_{1}^{2\nu }}{%
V(x_{0},Ar)^{\nu }}w(x_{0},Ar)\mathcal{E}(u) \\
&=&C\left( \frac{V(x_{0},r)}{V(x_{0},Ar)}\right) ^{\nu }||u||_{2}^{2(1+\nu
)}+\frac{Cw(x_{0},Ar)}{V(x_{0},Ar)^{\nu }}||u||_{1}^{2\nu }\mathcal{E}(u).
\end{eqnarray*}%
By condition $(\mathrm{RVD})$, we have 
\begin{equation*}
\frac{V(x_{0},r)}{V(x_{0},Ar)}\leq \frac{1}{C_{d}A^{d_{1}}}=\left( \frac{1}{%
2C}\right) ^{1/\nu },
\end{equation*}%
provided that 
\begin{equation}
A=C_{d}^{-1/d_{1}}(2C)^{\frac{1}{\nu d_{1}}}>1.  \label{aa}
\end{equation}%
Therefore, for all $u\in \mathcal{F}(B(x_{0},r))$, 
\begin{equation}
||u||_{2}^{2+2\nu }\leq 2C\frac{w(x_{0},Ar)}{V(x_{0},r)^{\nu }}%
||u||_{1}^{2\nu }\mathcal{E}(u),  \label{228}
\end{equation}%
which gives that 
\begin{equation*}
\mathcal{E}(u)\geq \frac{1}{2C}\frac{V(x_{0},r)^{\nu }}{w(x_{0},Ar)}%
||u||_{2}^{2+2\nu }||u||_{1}^{-2\nu }\geq \frac{C^{\prime }\mu (B)^{\nu }}{%
w(B)}||u||_{2}^{2+2\nu }||u||_{1}^{-2\nu }.
\end{equation*}%
Applying \cite[Lemma 5.5]{GrigoryanHu.2014.MMJ505} with $U=B(x_{0},r)$, $%
a=C^{\prime }\frac{{\mu (B)}^{\nu }}{w(B)}$, we conclude that the Dirichlet
heat kernel $p_{t}^{B}(x,y)$ exists and satisfies (\ref{eq:vol_27}).

We will show that condition $(\mathrm{FK}$) follows from (\ref{228}).

Indeed, let $D\subset B$ be an open subset, and let $u\in \mathcal{F}(D)$.
Noting that%
\begin{equation*}
||u||_{1}^{2}\leq \mu (D)||u||_{2}^{2}
\end{equation*}%
by the Cauchy-Schwarz inequality, we see from (\ref{228}) and (\ref{eq:vol_0}%
) that 
\begin{equation*}
||u||_{2}^{2+2\nu }\leq 2Cw(x_{0},Ar)\left( \frac{\mu (D)}{\mu (B)}\right)
^{\nu }||u||_{2}^{2\nu }\mathcal{E}(u)\leq C^{\prime \prime }w(B)\left( 
\frac{\mu (D)}{\mu (B)}\right) ^{\nu }||u||_{2}^{2\nu }\mathcal{E}(u),
\end{equation*}%
thus showing that 
\begin{equation*}
\lambda _{1}(D)=\underset{u\in \mathcal{F}(D)\backslash \{0\}}{\inf }\frac{%
\mathcal{E}(u)}{||u||_{2}^{2}}\geq \frac{c^{\prime }}{w(B)}\left( \frac{\mu
(B)}{\mu (D)}\right) ^{\nu }.
\end{equation*}%
Therefore, the Faber-Krahn inequality holds for any ball $B$ of radius $r$
satisfying (\ref{rr1}).
\end{proof}

We remark that if the metric space $(M,d)$ is connected and unbounded, then
condition $(\mathrm{VD})$ implies condition $(\mathrm{RVD})$, see for
example \cite[Corollary 5.3]{GrigoryanHu.2014.MMJ505}. In this case, we have
that conditions $\mathrm{(VD)}$, $\mathrm{(PI)}$ will imply condition $%
\mathrm{(FK)}$, since condition $(\mathrm{RVD})$ is automatically true.

\section{A refinement of lemma of growth}

\label{Sect-LG}In this section we shall derive the lemma of growth for any
two concentric balls $B,\delta B$ with $0<\delta <1$, which is a refinement
of the version stated in a forthcoming paper \cite{ghh21TP}, see also \cite[%
Lemma 4.1]{GrigoryanHuHu.2018.AM433}. The lemma of growth will follow from
conditions $(\mathrm{VD})$, $(\mathrm{Gcap})$, $(\mathrm{FK})$, $(\mathrm{TJ}%
)$. The basic tool in the proof is to use the celebrated De-Giorgi iteration
technique for occupation measures (instead of for $L^{2}$-norms). Although
the idea is essentially the same as in \cite{ghh21TP}, \cite[Proof of Lemma
4.1]{GrigoryanHuHu.2018.AM433}, we sketch the proof for the reader's
convenience.

Before we address the lemma of growth, we give the following preliminary.
For each $n\geq 1$, let $F_{n}$ be a function on $[0,\infty )$ given by%
\begin{equation}
F_{n}(r)=\frac{1}{2}\left( r+\sqrt{r^{2}+\frac{1}{n^{2}}}\right) -\frac{1}{2n%
}\text{ \ for }r\in (-\infty ,\infty ).  \label{F1}
\end{equation}%
Clearly, $F_{n}(0)=0$, and for any $r\in (-\infty ,\infty )$, 
\begin{eqnarray}
&&0\leq F_{n}^{\prime }(r)=\frac{1}{2}\left( 1+\frac{r}{\sqrt{r^{2}+n^{-2}}}%
\right) \leq 1\text{,\ \ }  \notag \\
&&0\leq F_{n}^{\prime \prime }(r)=\frac{1}{2n^{2}(r^{2}+n^{-2})^{3/2}}\leq 
\frac{n}{2},  \notag \\
&&F_{n}(r)\rightrightarrows r_{+}\ \ \text{uniformly in }(-\infty ,\infty )\ 
\text{as}\ n\rightarrow \infty .  \label{700}
\end{eqnarray}

\begin{proposition}
\label{P21}Let $(\mathcal{E},\mathcal{F})$ be a regular Dirichlet form in $%
L^{2}(M,\mu )$ without killing part and let $F_{n}$ be given by (\ref{F1}).
Then for any $u\in \mathcal{F}^{\prime }\cap L^{\infty }$ and any $0\leq
\varphi \in \mathcal{F}\cap L^{\infty }$,%
\begin{equation}
\mathcal{E}(u_{+},\varphi )\leq \limsup_{k\rightarrow \infty }\mathcal{E}%
(u,F_{n_{k}}^{\prime }(u)\varphi )  \label{F2}
\end{equation}%
for a subsequence $\{n_{k}\}_{k\geq 1}$ of $\{n\}_{n\geq 1}$.
\end{proposition}

\begin{proof}
Note that the functions $F_{n}(u)$, $F_{n}^{\prime }(u)\varphi $ belong to $%
\mathcal{F}\cap L^{\infty }$ for each $n\geq 1$ by using Proposition \ref%
{P61} in Appendix. Since $\varphi \geq 0$ in $M$, we have 
\begin{equation}
\mathcal{E}(F_{n}(u),\varphi )\leq \mathcal{E}(u,F_{n}^{\prime }(u)\varphi )%
\text{ \ (}n\geq 1\text{)}  \label{F3}
\end{equation}%
by using (\ref{eq61}) in Appendix.

Write $u=v+a$ for some $v\in \mathcal{F}$ and $a\in 
%TCIMACRO{\U{211d} }%
%BeginExpansion
\mathbb{R}
%EndExpansion
$. Since $F_{n}(v+a)-F_{n}(a)$ is a normal contraction of $v\in \mathcal{F}$%
, we have 
\begin{equation*}
f_{n}:=F_{n}(u)-F_{n}(a)=F_{n}(v+a)-F_{n}(a)\in \mathcal{F}\text{ \ and \ }%
\mathcal{E}(f_{n},f_{n})\leq \mathcal{E}(v,v).
\end{equation*}%
Since $(v+a)_{+}-a_{+}$ is also a normal contraction of $v\in \mathcal{F}$,
we also have%
\begin{equation*}
f:=u_{+}-a_{+}=(v+a)_{+}-a_{+}\in \mathcal{F}\text{.}
\end{equation*}%
On the other hand, by the dominated convergence theorem, 
\begin{equation*}
f_{n}\overset{L^{2}}{\longrightarrow }f\ \ \text{as}\ n\rightarrow \infty .
\end{equation*}%
Since $f_{n}\in \mathcal{F}$ and 
\begin{equation*}
\sup_{n}\mathcal{E}(f_{n},f_{n})\leq \mathcal{E}(v,v)<\infty ,
\end{equation*}%
there exists a subsequence $\left\{ f_{n_{k}}\right\} _{k\geq 1}$ converging
to $f$ weakly in terms of the energy norm $\mathcal{E}$ by using Lemma \ref%
{lem:A1} in Appendix. Therefore, 
\begin{eqnarray*}
\mathcal{E}(u_{+},\varphi ) &=&\mathcal{E}(f+a_{+},\varphi )=\mathcal{E}%
(f,\varphi )=\lim_{k\rightarrow \infty }\mathcal{E}(f_{n_{k}},\varphi ) \\
&=&\lim_{k\rightarrow \infty }\mathcal{E}(F_{n_{k}}(u)-F_{n_{k}}(a),\varphi
)=\limsup_{k\rightarrow \infty }\mathcal{E}(F_{n_{k}}(u),\varphi )\leq
\limsup_{k\rightarrow \infty }\mathcal{E}(u,F_{n_{k}}^{\prime }(u)\varphi )
\end{eqnarray*}%
by virtue of (\ref{F3}), thus showing (\ref{F2}). The proof is complete.
\end{proof}

We recall condition $(\mathrm{LG})$, termed the \emph{lemma of growth},
which was introduced in \cite[Lemma 4.1]{GrigoryanHuHu.2018.AM433} for the
case when $w(x,r)=r^{\beta }$ and $f\equiv 0$. Note that the following
notion of lemma of growth involves a given function $f.$

\begin{definition}
\label{def-LG}For any two fixed numbers $\varepsilon ,\delta $ in $(0,1)$,
we say that \emph{condition }$\mathrm{LG}(\varepsilon ,\delta )$ holds if
there exist four constants $\sigma \in (0,1),\varepsilon _{0}\in (0,\frac{1}{%
2})$ and $\theta ,C_{L}>0$ such that, for any ball $B:=B(x_{0},R)$ with
radius $R\in (0,\sigma \overline{R})$, any function $f\in L^{\infty }(B)$,
and for any $u\in \mathcal{F^{\prime }}\cap L^{\infty }$ which is $f$%
-superharmonic and non-negative in $B$, if for some $a>0$ 
\begin{equation}
\frac{\mu (B\cap \{u<a\})}{\mu (B)}\leq \varepsilon _{0}(1-\varepsilon
)^{2\theta }(1-\delta )^{C_{L}\theta }\left( 1+\frac{w(B)\left( T_{\frac{%
3+\delta }{4}B,B}(u_{-})+||f||_{L^{\infty }(B)}\right) }{\varepsilon a}%
\right) ^{-\theta },  \label{eq:vol_317}
\end{equation}%
then 
\begin{equation}
\einf_{{\delta B}}u\geq \varepsilon a,  \label{eq:a/2}
\end{equation}%
where the tail $T_{\frac{3+\delta }{4}B,B}(u_{-})$ is defined by (\ref{tj2}%
), that is%
\begin{equation*}
T_{\frac{3+\delta }{4}B,B}(u_{-})=\esup_{x\in \frac{3+\delta }{4}%
B}\int_{M\setminus B}u_{-}(y)J(x,dy).
\end{equation*}%
For simplicity, we write condition\emph{\ }$\mathrm{LG}(\varepsilon ,\delta
) $ by condition\emph{\ }$(\mathrm{LG})$ without mentioning $\varepsilon
,\delta $.
\end{definition}

We remark that the constants $\sigma ,\varepsilon _{0},\theta ,C_{L}$ are
all independent of $\varepsilon ,\delta $. Recall that condition $(\mathrm{EP%
})$, termed the \emph{energy product} of a function $u$ with some cutoff
function $\phi $, was introduced in \cite{ghh21TP}.

\begin{definition}[Condition $(\mathrm{EP})$]
\label{def:EP}\textrm{We say that the \emph{condition} $(\mathrm{EP})$ is
satisfied if there exist two universal constants $C>0,C_{0}\geq 0$ such
that, for any three concentric balls $B_{0}:=B(x_{0},R)$, $B:=B(x_{0},R+r)$
and $\Omega :=B(x_{0},R^{\prime })$ with $0<R<R+r<R^{\prime }<\overline{R}$,
and for any $u\in \mathcal{F}^{\prime }\cap L^{\infty }$, there exists some $%
\phi \in \mathrm{cutoff}(B_{0},B)$ such that 
\begin{equation}
\mathcal{E}(u\phi )\leq \frac{3}{2}\mathcal{E}(u,u\phi ^{2})+\frac{C}{%
w(x_{0},r)}\left( \frac{R^{\prime }}{r}\right) ^{C_{0}}\int_{\Omega
}u^{2}d\mu +3\int_{\Omega \times \Omega ^{c}}u(x)u(y)\phi ^{2}(x)J(dx,dy).
\label{eq:EP}
\end{equation}
}
\end{definition}

Condition $(\mathrm{EP}$) plays an important role in deriving condition $(%
\mathrm{LG})$. The following has been proved in \cite{ghh21TP}.

\begin{lemma}[\protect\cite{ghh21TP}]
\label{L31}Assume that $(\mathcal{E},\mathcal{F})$ is a regular Dirichlet
form in $L^{2}$ without killing part. Then 
\begin{equation}
(\mathrm{Gcap})+(\mathrm{TJ})\Rightarrow (\mathrm{EP}).  \label{eq:ABB->EP-1}
\end{equation}
\end{lemma}

We shall prove the lemma of growth, where condition $(\mathrm{EP}$) is our
starting point, instead of from condition $(\mathrm{Gcap})$. The idea is
essentially adopted from \cite{GrigoryanHuHu.2018.AM433, ghh21TP}.

\begin{lemma}
\label{lemma:LG}Let $(\mathcal{E},\mathcal{F})$ be a regular Dirichlet form
in $L^{2}$ without killing part. If conditions $(\mathrm{VD})$, $(\mathrm{FK}%
)$, $(\mathrm{TJ})$ and $(\mathrm{EP})$ are satisfied, then condition $(%
\mathrm{LG})$ holds with $\theta =1/\nu $ and $C_{L}=C_{0}+\beta _{2}+d_{2}$%
, where the constants $\sigma ,\nu $ are taken same as in condition $(%
\mathrm{FK})$ and $C_{0}$ same as in condition $(\mathrm{EP})$. Namely, we
have 
\begin{equation}
(\mathrm{VD})+(\mathrm{FK})+(\mathrm{TJ})+(\mathrm{EP})\Rightarrow (\mathrm{%
LG})\text{.}  \label{60}
\end{equation}%
Consequently, 
\begin{equation}
(\mathrm{VD})+(\mathrm{Gcap})+(\mathrm{FK})+(\mathrm{TJ})\Rightarrow (%
\mathrm{LG}).  \label{76}
\end{equation}
\end{lemma}

\begin{proof}
Note that any function $u\in \mathcal{F}$ admits a \emph{quasi-continuous}
version $\widetilde{u}$ \cite[Theorem 2.1.3, p.71]%
{FukushimaOshimaTakeda.2011.489}. We will use the same letter $u$ to denote
some quasi-continuous modification of $u$. For any $u\in \mathcal{F}$ and
any open subset $\Omega $ of $M$, a function $u$ belongs to the space $%
\mathcal{F}\left( \Omega \right) $ if and only if $\widetilde{u}=0$ q.e. in $%
\Omega ^{c}$, where q.e. means \emph{quasi-everywhere} (see \cite[Corollary
2.3.1, p.98]{FukushimaOshimaTakeda.2011.489}).

We shall show the implication (\ref{60}).

\noindent Fix a ball $B:=B(x_{0},R)$ with radius $0<R<\sigma \overline{R}$
and a function $f\in L^{\infty }(B)$. Let $u\in \mathcal{F}^{\prime }\cap
L^{\infty }$ be a function that is $f$-superharmonic and non-negative in $B$%
. We will show that (\ref{eq:a/2}) is true if condition (\ref{eq:vol_317})
is satisfied for some $a>0$.

To do this, denote 
\begin{equation*}
B_{r}:=B(x_{0},r)\text{ for any }r>0,
\end{equation*}%
so that $B_{R}=B=B(x_{0},R)$. Fix four numbers $a,b$ and $r_{1},r_{2}$ such
that 
\begin{equation}
0<a<b<\infty \text{ \ and \ }\frac{r_{2}}{2}\leq r_{1}<r_{2}<R,  \label{4num}
\end{equation}%
and set%
\begin{equation*}
m_{1}=\frac{\mu (B_{r_{1}}\cap \{u<a\})}{\mu \left( B_{r_{1}}\right) }\ \ 
\text{and\ \ }m_{2}=\frac{\mu (B_{r_{2}}\cap \{u<b\})}{\mu \left(
B_{r_{2}}\right) }.
\end{equation*}%
Set also $v:=\left( b-u\right) _{+}$ and%
\begin{equation*}
\widetilde{m}_{1}:=\mu (B_{r_{1}}\cap \{u<a\}),\ \ \widetilde{m}_{2}:=\mu
(B_{r_{2}}\cap \{u<b\}).
\end{equation*}%
Let $\widetilde{B}$ be any intermediate concentric ball between $B_{r_{1}}$
and $B_{r_{2}}$, so that%
\begin{equation*}
B_{r_{1}}\subset \widetilde{B}:=B_{r_{1}+\rho }\subset B_{r_{2}}\text{ }%
(0<\rho <r_{2}-r_{1})\text{.}
\end{equation*}%
Applying condition $(\mathrm{EP})$ to the triple $B_{r_{1}}$, $\widetilde{B}$%
, $B_{r_{2}}$ and the function $v$, we see that there exists some function $%
\phi \in \mathrm{cutoff}(B_{r_{1}},\widetilde{B})$ such that 
\begin{eqnarray}
\mathcal{E}(v\phi ) &\leq &~\frac{3}{2}\mathcal{E}(v,v\phi ^{2})+\frac{C}{%
w\left( x_{0},\rho \right) }\left( \frac{r_{2}}{\rho }\right)
^{C_{0}}\int_{B_{r_{2}}}v^{2}d\mu  \notag \\
&&+3\int_{B_{r_{2}}\times B_{r_{2}}^{c}}v(x)v(y)\phi ^{2}(x)J(dx,dy).
\label{505}
\end{eqnarray}%
Without loss of generality, we can assume that $\phi $ is quasi-continuous.
Then we have%
\begin{equation}
\widetilde{m}_{1}=\int_{B_{r_{1}}\cap \{u<a\}}\phi ^{2}d\mu \leq
\int_{B_{r_{1}}}\phi ^{2}\underset{\geq 1\text{ on }\{u<a\}}{\underbrace{%
\left( \frac{(b-u)_{+}}{b-a}\right) ^{2}}}d\mu =\frac{1}{(b-a)^{2}}%
\int_{B_{r_{1}}}(\phi v)^{2}d\mu .  \label{eq:ma1Up}
\end{equation}%
Consider the set 
\begin{equation*}
E:=\widetilde{B}\cap \{u<b\}.
\end{equation*}%
By the outer regularity of $\mu $, for any $\epsilon >0$ , there is an open
set $\Omega $ such that $E\subset \Omega \subset B_{r_{2}}$ and 
\begin{equation}
\mu (\Omega )\leq \mu (E)+\epsilon \leq \widetilde{m}_{2}+\epsilon .
\label{eq:muOmegaUp}
\end{equation}

On the other hand, since $\phi =0$ q.e. outside $\widetilde{B}$ and $v=0$
outside $\{u<b\}$, we see that $\phi v=0$ q.e. in $E^{c}$. Since $\phi v\in 
\mathcal{F}$ and $\phi v=0$ q.e. in $\Omega ^{c}\subset E^{c}$, we conclude
that 
\begin{equation}
\phi v\in \mathcal{F}(\Omega ).  \label{eq:vphiFOmega}
\end{equation}%
By the definition of $\lambda _{1}\left( \Omega \right) $, we have%
\begin{equation*}
\int_{\Omega }(\phi v)^{2}d\mu \leq \frac{\mathcal{E}(\phi v)}{\lambda
_{1}(\Omega )}.
\end{equation*}%
Using again the fact that $\phi v$ vanishes outside $\Omega $ and combining
this inequality with (\ref{eq:ma1Up}), we obtain that%
\begin{equation}
\widetilde{m}_{1}\leq \frac{1}{(b-a)^{2}}\int_{B_{r_{1}}}(\phi v)^{2}d\mu
\leq \frac{1}{(b-a)^{2}}\int_{\Omega }(\phi v)^{2}d\mu \leq \frac{\mathcal{E}%
(\phi v)}{(b-a)^{2}\lambda _{1}(\Omega )}.  \label{eq:m1<}
\end{equation}%
By condition $(\mathrm{FK})$ and (\ref{eq:muOmegaUp}), 
\begin{equation}
\lambda _{1}(\Omega )\geq \frac{C_{F}^{-1}}{w(B_{r_{2}})}\left( \frac{\mu
(B_{r_{2}})}{\mu (\Omega )}\right) ^{\nu }\geq \frac{C_{F}^{-1}}{w(B_{r_{2}})%
}\left( \frac{\mu (B_{r_{2}})}{\widetilde{m}_{2}+\epsilon }\right) ^{\nu },
\label{eq:la1>}
\end{equation}%
from which, it follows by (\ref{eq:m1<}) that%
\begin{equation*}
\widetilde{m}_{1}\leq \frac{\mathcal{E}(\phi v)}{(b-a)^{2}}\cdot \frac{%
w(B_{r_{2}})}{C_{F}^{-1}}\left( \frac{\mu (B_{r_{2}})}{\widetilde{m}%
_{2}+\epsilon }\right) ^{-\nu }.
\end{equation*}%
Letting $\epsilon \rightarrow 0$, we obtain that, using the fact that $m_{2}=%
\frac{\widetilde{m}_{2}}{\mu (B_{r_{2}})}$, 
\begin{equation}
\widetilde{m}_{1}\leq \frac{C_{F}}{(b-a)^{2}}\left( \frac{\widetilde{m}_{2}}{%
\mu (B_{r_{2}})}\right) ^{\nu }\cdot w(B_{r_{2}})\mathcal{E}(\phi v)=\frac{%
C_{F}\left( m_{2}\right) ^{\nu }}{(b-a)^{2}}\cdot w(B_{r_{2}})\mathcal{E}%
(\phi v),  \label{LG1}
\end{equation}%
where the constants $\nu $ and $C_{F}$ are the same as in condition $(%
\mathrm{FK})$.

We estimate the term $\mathcal{E}(\phi v)$ on the right-hand side of (\ref%
{LG1}) by applying the inequality (\ref{505}). For this, we need to estimate
the term $\mathcal{E}(v,v\phi ^{2})$. This can be done by using the $f$%
-superharmonicity of $u$ and using condition $(\mathrm{TJ})$.

Indeed, since $v\phi \in \mathcal{F}(\Omega )\cap L^{\infty }$ and $\phi \in 
\mathcal{F}\cap L^{\infty }$, the function $v\phi ^{2}=v\phi \cdot \phi \in 
\mathcal{F}(\Omega )\subset \mathcal{F}(B)$, which is non-negative. Let $%
F_{n}$ be given by (\ref{700}) for $n\geq 1$. Since $u$ is $f$-superharmonic
in $B$ and $\Vert F_{n}^{\prime }\Vert _{\infty }\leq 1$ and since the
function $F_{n}^{\prime }(b-u)v\phi ^{2}$ is non-negative and belongs to the
space $\mathcal{F}(\Omega )$ so that it can be used as a test function, we
have 
\begin{align}
\mathcal{E}(b-u,F_{n}^{\prime }(b-u)v\phi ^{2})& =-\mathcal{E}%
(u,F_{n}^{\prime }(b-u)v\phi ^{2})\leq -(f,F_{n}^{\prime }(b-u)v\phi ^{2}) 
\notag \\
& \leq \int_{M}|f|v\phi ^{2}d\mu \leq ||f||_{L^{\infty
}(B_{r_{2}})}\int_{B_{r_{2}}}vd\mu   \notag \\
& \leq ||f||_{L^{\infty }(B_{r_{2}})}b\mu (B_{r_{2}}\cap \{u<b\})~\text{ \
(using $v\leq b1_{\{u<b\}}$)}  \notag \\
& =b||f||_{L^{\infty }(B_{r_{2}})}\widetilde{m}_{2}.  \label{504}
\end{align}%
Applying (\ref{F2}) with $u$ replaced by $b-u$ and with $\varphi =v\phi ^{2}$%
, we obtain by (\ref{504}) 
\begin{equation}
\mathcal{E}(v,v\phi ^{2})=\mathcal{E}((b-u)_{+},v\phi ^{2})\leq
\limsup_{k\rightarrow \infty }\mathcal{E}(b-u,F_{n_{k}}^{\prime }(b-u)v\phi
^{2})\leq b||f||_{L^{\infty }(B_{r_{2}})}\widetilde{m}_{2}.  \label{eq:Evv}
\end{equation}

Therefore, plugging (\ref{eq:Evv}) into (\ref{505}) and then using the facts
that%
\begin{equation*}
\mathrm{supp\,(\phi )}\subset \widetilde{B}\text{ \ and \ }%
J(dx,dy)=J(x,dy)\mu (dx),
\end{equation*}%
we see that%
\begin{align}
\mathcal{E}(v\phi )& \leq \frac{3}{2}b||f||_{L^{\infty }(B_{r_{2}})}%
\widetilde{m}_{2}+\frac{C}{w\left( x_{0},\rho \right) }\left( \frac{r_{2}}{%
\rho }\right) ^{C_{0}}\int_{B_{r_{2}}}v^{2}d\mu  \notag \\
& +3\int_{\widetilde{B}}v\left( x\right) \mu \left( dx\right) \cdot \esup%
_{x\in \widetilde{B}}\int_{B_{r_{2}}^{c}}v(y)J(x,dy).  \label{500}
\end{align}%
Since $v=\left( b-u\right) _{+}\leq b$ in $B_{r_{2}}\subset B_{R}=B$, we have%
\begin{equation*}
\int_{B_{r_{2}}}v^{2}d\mu \leq b^{2}\mu (B_{r_{2}}\cap \{u<b\})=b^{2}%
\widetilde{m}_{2}.
\end{equation*}%
Since $v=\left( b-u\right) _{+}\leq b+u_{-}$ in $M$, we also have%
\begin{eqnarray*}
\int_{\widetilde{B}}v\left( x\right) \mu \left( dx\right) \esup_{x\in 
\widetilde{B}}\int_{B_{r_{2}}^{c}}v(y)J(x,dy) &\leq &b\widetilde{m}_{2}\esup%
_{x\in \widetilde{B}}\int_{B_{r_{2}}^{c}}v(y)J(x,dy) \\
&\leq &b\widetilde{m}_{2}\esup_{x\in \widetilde{B}%
}\int_{B_{r_{2}}^{c}}(b+u_{-}(y))J(x,dy) \\
&=&b\widetilde{m}_{2}\left( b\esup_{x\in \widetilde{B}%
}\int_{B_{r_{2}}^{c}}J(x,dy)+T_{\widetilde{B},B_{r_{2}}}(u_{-})\right) .
\end{eqnarray*}%
Thus, using the fact that for any point $x_{0}$ in $M$ and any $0<\rho
<r_{2}-r_{1}$,%
\begin{equation*}
\frac{w(B_{r_{2}})}{w(x_{0},\rho )}=\frac{w(x_{0},r_{2})}{w(x_{0},\rho )}%
\leq C_{2}\left( \frac{r_{2}}{\rho }\right) ^{\beta _{2}}
\end{equation*}%
by virtue of (\ref{eq:vol_0}), it follows from (\ref{500}) that%
\begin{eqnarray}
\mathcal{E}(v\phi ) &\leq &\frac{3}{2}b||f||_{L^{\infty }(B_{r_{2}})}%
\widetilde{m}_{2}+\frac{CC_{2}}{w(B_{r_{2}})}\left( \frac{r_{2}}{\rho }%
\right) ^{C_{0}+\beta _{2}}\cdot b^{2}\widetilde{m}_{2}  \notag \\
&&+3b\widetilde{m}_{2}\left( b\esup_{x\in \widetilde{B}%
}\int_{B_{r_{2}}^{c}}J(x,dy)+T_{\widetilde{B},B_{r_{2}}}(u_{-})\right) .
\label{507}
\end{eqnarray}%
We look at the third term on the right-hand side of (\ref{507}).

Observing by (\ref{eq:vol_0}) that for any $x\in \widetilde{B}\subset
B_{r_{2}},$%
\begin{equation}
\frac{w(B_{r_{2}})}{w(x,r_{2}-r_{1}-\rho )}=\frac{w(x_{0},r_{2})}{%
w(x,r_{2}-r_{1}-\rho )}\leq C_{2}\left( \frac{r_{2}}{r_{2}-r_{1}-\rho }%
\right) ^{\beta _{2}}\text{,}  \label{508}
\end{equation}%
we have by condition $(\mathrm{TJ})$ that%
\begin{eqnarray}
\esup_{x\in \widetilde{B}}\int_{B_{r_{2}}^{c}}J(x,dy) &\leq &\esup_{x\in 
\widetilde{B}}\int_{B(x,r_{2}-r_{1}-\rho )^{c}}J(x,dy)\leq \esup_{x\in 
\widetilde{B}}\frac{C}{w(x,r_{2}-r_{1}-\rho )}  \notag \\
&\leq &\frac{CC_{2}}{w(B_{r_{2}})}\left( \frac{r_{2}}{r_{2}-r_{1}-\rho }%
\right) ^{\beta _{2}}.  \label{509}
\end{eqnarray}%
Plugging (\ref{509}) into (\ref{507}), we obtain%
\begin{eqnarray*}
\mathcal{E}(v\phi ) &\leq &\frac{3}{2}b||f||_{L^{\infty }(B_{r_{2}})}%
\widetilde{m}_{2}+\frac{CC_{2}}{w(B_{r_{2}})}\left( \frac{r_{2}}{\rho }%
\right) ^{C_{0}+\beta _{2}}\cdot b^{2}\widetilde{m}_{2} \\
&&+3b\widetilde{m}_{2}\left( b\frac{CC_{2}}{w(B_{r_{2}})}\left( \frac{r_{2}}{%
r_{2}-r_{1}-\rho }\right) ^{\beta _{2}}+T_{\widetilde{B},B_{r_{2}}}(u_{-})%
\right)  \\
&\leq &\frac{C^{\prime }b^{2}\widetilde{m}_{2}}{w(B_{r_{2}})}\left( \frac{%
r_{2}}{\rho }\right) ^{C_{0}+\beta _{2}}\left( 1+\frac{w(B_{r_{2}})\left( T_{%
\widetilde{B},B_{r_{2}}}(u_{-})+||f||_{L^{\infty }(B_{r_{2}})}\right) }{b}%
\right) ,
\end{eqnarray*}%
provided that $0<\rho \leq (r_{2}-r_{1})/2$, since in this case%
\begin{equation*}
\left( \frac{r_{2}}{r_{2}-r_{1}-\rho }\right) ^{\beta _{2}}\leq \left( \frac{%
r_{2}}{\rho }\right) ^{\beta _{2}}\leq \left( \frac{r_{2}}{\rho }\right)
^{C_{0}+\beta _{2}}.
\end{equation*}%
From this, we obtain by (\ref{LG1}) that%
\begin{equation*}
\widetilde{m}_{1}\leq C^{\prime }C_{F}\,m_{2}^{\nu }\,\widetilde{m}%
_{2}\left( \frac{b}{b-a}\right) ^{2}\left( \frac{r_{2}}{\rho }\right)
^{C_{0}+\beta _{2}}\left( 1+\frac{w(B_{r_{2}})\left( T_{\widetilde{B}%
,B_{r_{2}}}(u_{-})+||f||_{L^{\infty }(B_{r_{2}})}\right) }{b}\right) .
\end{equation*}%
Dividing this inequality by $\mu \left( B_{r_{1}}\right) $ and then using
the facts that%
\begin{equation*}
m_{1}=\frac{\widetilde{m}_{1}}{\mu (B_{r_{1}})}\text{ \ and}\ \ m_{2}=\frac{%
\widetilde{m}_{2}}{\mu (B_{r_{2}})}
\end{equation*}%
and that, for any $\frac{r_{2}}{2}\leq r_{1}<r_{2}$, 
\begin{equation*}
\frac{\mu (B_{r_{2}})}{\mu (B_{r_{1}})}\leq C_{\mu }\left( \frac{r_{2}}{r_{1}%
}\right) ^{d_{2}}\leq C_{\mu }\left( \frac{r_{2}}{r_{2}-r_{1}}\right)
^{d_{2}}\leq C_{\mu }\left( \frac{r_{2}}{\rho }\right) ^{d_{2}}\text{ (by
using (\ref{eq:vol_3}))}
\end{equation*}%
we conclude that, for all $0<\rho \leq (r_{2}-r_{1})/2$ with $\frac{r_{2}}{2}%
\leq r_{1}<r_{2}$, 
\begin{align}
m_{1}& \leq C^{\prime }C_{F}m_{2}^{1+\nu }\left( \frac{b}{b-a}\right) ^{2}%
\frac{\mu (B_{r_{2}})}{\mu (B_{r_{1}})}\left( \frac{r_{2}}{\rho }\right)
^{C_{0}+\beta _{2}}\left( 1+\frac{w(B_{r_{2}})\left( T_{\widetilde{B}%
,B_{r_{2}}}(u_{-})+||f||_{L^{\infty }(B_{r_{2}})}\right) }{b}\right)   \notag
\\
& \leq C\left( \frac{b}{b-a}\right) ^{2}\left( \frac{r_{2}}{\rho }\right)
^{C_{0}+\beta _{2}+d_{2}}\left( 1+\frac{w(B_{r_{2}})\left( T_{B_{r_{1}+\rho
},B_{r_{2}}}(u_{-})+||f||_{L^{\infty }(B_{r_{2}})}\right) }{b}\right) \cdot
m_{2}^{1+\nu },  \label{eq:tildem1m2}
\end{align}%
where $C:=C^{\prime }C_{F}C_{\mu }>0$ depends only on the constants from the
hypotheses (but is independent of the numbers $\rho ,a,b,r_{1},r_{2}$ and
the functions $f,u$). We will apply (\ref{eq:tildem1m2}) to show (\ref%
{eq:vol_317}).

In fact, let $\delta ,\varepsilon $ be any two fixed numbers in $(0,1)$.
Consider the following sequences%
\begin{equation*}
R_{k}:=\left( \delta +2^{-k}(1-\delta )\right) R\quad \text{ and }\quad
a_{k}:=\left( \varepsilon +2^{-k}(1-\varepsilon )\right) a\text{ \ for }%
k\geq 0.
\end{equation*}%
Clearly, $R_{0}=R$, $a_{0}=a$, $R_{k}\searrow \delta R$, and $a_{k}\searrow
\varepsilon a$ as $k\rightarrow \infty $, and%
\begin{equation*}
\frac{R_{k-1}}{2}<R_{k}<R_{k-1}\text{ \ for any }k\geq 1.
\end{equation*}%
Set%
\begin{equation*}
m_{k}:=\frac{\mu (B_{R_{k}}\cap \{u<a_{k}\})}{\mu (B_{R_{k}})}.
\end{equation*}%
Applying (\ref{eq:tildem1m2}) with%
\begin{align*}
& a=a_{k}\text{, \ }b=a_{k-1}\text{, \ }r_{1}=R_{k}\text{, \ }r_{2}=R_{k-1}%
\text{ \ and} \\
& \rho =\rho _{k}:=(R_{k-1}-R_{k})/2=2^{-k-1}(1-\delta )R,
\end{align*}%
we obtain for all $k\geq 1$%
\begin{equation}
m_{k}\leq CA_{k}\left( \frac{a_{k-1}}{a_{k-1}-a_{k}}\right) ^{2}\left( \frac{%
R_{k-1}}{R_{k-1}-R_{k}}\right) ^{C_{0}+\beta _{2}+d_{2}}m_{k-1}^{1+\nu },
\label{72}
\end{equation}%
where $A_{k}$ is given by 
\begin{equation*}
A_{k}:=1+\frac{w(B_{R_{k-1}})\left( T_{B_{R_{k}+\rho
_{k}},B_{R_{k-1}}}(u_{-})+||f||_{L^{\infty }(B_{R_{k-1}})}\right) }{a_{k-1}}.
\end{equation*}%
Since $B_{R_{k}+\rho _{k}}\subset B_{(3+\delta )R/4}$ for any $k\geq 1$\ by
noting that 
\begin{equation*}
R_{k}+\rho _{k}=\left( \delta +2^{-k}(1-\delta )\right) R+2^{-k-1}(1-\delta
)R\leq \frac{(3+\delta )R}{4}
\end{equation*}%
and since $u_{-}=0$ in $B=B_{R}\supseteq B_{R_{k-1}}$ by using the fact that 
$u$ is non-negative, we have%
\begin{equation*}
T_{B_{R_{k}+\rho _{k}},B_{R_{k-1}}}(u_{-})=T_{B_{R_{k}+\rho
_{k}},B_{R}}(u_{-})\leq T_{B_{(3+\delta )R/4},B_{R}}(u_{-}).
\end{equation*}%
Since $a_{k-1}\geq \varepsilon a$ and $w(B_{R_{k-1}})\leq w(B)$, it follows
that%
\begin{equation}
A_{k}\leq 1+\frac{w(B)\left( T_{B_{(3+\delta
)R/4},B}(u_{-})+||f||_{L^{\infty }(B)}\right) }{\varepsilon a}\eqqcolon A%
\text{ \ for any }k\geq 1.  \label{71}
\end{equation}%
Since for any $k\geq 1$%
\begin{equation*}
\frac{a_{k-1}}{a_{k-1}-a_{k}}=\frac{\varepsilon +2^{-\left( k-1\right)
}(1-\varepsilon )}{\left( 2^{-\left( k-1\right) }-2^{-k}\right)
(1-\varepsilon )}\leq \frac{2^{k}}{1-\varepsilon }\text{ \ and \ }\frac{%
R_{k-1}}{R_{k-1}-R_{k}}\leq \frac{2^{k}}{1-\delta },
\end{equation*}%
we obtain from (\ref{72}), (\ref{71}) that%
\begin{equation}
m_{k}\leq CA\left( \frac{2^{k}}{1-\varepsilon }\right) ^{2}\left( \frac{2^{k}%
}{1-\delta }\right) ^{C_{0}+\beta _{2}+d_{2}}m_{k-1}^{1+\nu }\eqqcolon %
DA\cdot 2^{\lambda k}\cdot m_{k-1}^{q},  \label{73}
\end{equation}%
where the constants $D$, $\lambda $, $q$ are respectively given by 
\begin{equation}
D:=C(1-\varepsilon )^{-2}(1-\delta )^{-(C_{0}+\beta _{2}+d_{2})}\text{, \ \ }%
\lambda :=2+C_{0}+\beta _{2}+d_{2}\text{ \ \ and \ \ }q:=1+\nu .  \label{74}
\end{equation}%
Iterating the inequality (\ref{73}), we have for all $k\geq 1$ 
\begin{align*}
m_{k}\leq & ~(DA)\cdot 2^{\lambda k}\cdot m_{k-1}^{q}\leq (DA)\cdot
2^{\lambda k}\cdot \left( DA\cdot 2^{\lambda (k-1)}\cdot m_{k-2}^{q}\right)
^{q} \\
=& ~(DA)^{1+q}\cdot 2^{\lambda k+\lambda q(k-1)}\cdot m_{k-2}^{q^{2}}\leq
\cdots \\
\leq & ~(DA)^{1+q+\cdots +q^{k-1}}\cdot 2^{\lambda (k+q(k-1)+\cdots
+q^{k-1})}\cdot m_{0}^{q^{k}}.
\end{align*}%
Since 
\begin{align*}
& k+q(k-1)+\cdots +q^{k-1}=\frac{q^{k+1}-(k+1)q+k}{(q-1)^{2}}\leq \frac{q}{%
\left( q-1\right) ^{2}}q^{k}, \\
& 1+q+\cdots +q^{k-1}=\frac{q^{k}-1}{q-1}\leq \frac{q^{k}}{q-1},
\end{align*}%
it follows that 
\begin{equation}
m_{k}\leq \Big((DA)^{\frac{1}{q-1}}\cdot 2^{\frac{\lambda q}{\left(
q-1\right) ^{2}}}\cdot m_{0}\Big)^{q^{k}},  \label{73-1}
\end{equation}%
from which, we conclude that if 
\begin{equation}
2^{\frac{\lambda q}{\left( q-1\right) ^{2}}}\cdot (DA)^{\frac{1}{q-1}}\cdot
m_{0}\leq \frac{1}{2},  \label{eq:tildem0small}
\end{equation}%
then 
\begin{equation}
\lim_{k\rightarrow \infty }m_{k}=0  \label{eq:limitmk}
\end{equation}%
by using the fact that $q>1$. Note that (\ref{eq:tildem0small}) is
equivalent to%
\begin{equation*}
m_{0}\leq 2^{-\frac{\lambda q}{\left( q-1\right) ^{2}}-1}\cdot (DA)^{-\frac{1%
}{q-1}},
\end{equation*}%
that is, 
\begin{eqnarray}
\frac{\mu (B\cap \{u<a\})}{\mu (B)} &=&m_{0}\leq 2^{-\frac{\lambda q}{\left(
q-1\right) ^{2}}-1}D^{-\frac{1}{q-1}}A^{-1/\nu }  \notag \\
&\eqqcolon&\varepsilon _{0}(1-\varepsilon )^{2\theta }(1-\delta
)^{C_{L}\theta }\left( 1+\frac{w(B)\left( T_{\frac{3+\delta }{4}%
B,B}(u_{-})+||f||_{L^{\infty }(B)}\right) }{\varepsilon a}\right) ^{-\theta
},  \label{75}
\end{eqnarray}%
where $\varepsilon _{0},\theta ,C_{L}$ are universal constants given by 
\begin{equation}
\varepsilon _{0}:=2^{-\lambda q/\left( q-1\right) ^{2}-1}C^{-1/(q-1)}<1/2%
\text{, \ }\theta :=1/\nu \text{, \ and \ }C_{L}:=C_{0}+\beta _{2}+d_{2},
\label{eq:epsdef}
\end{equation}%
since by (\ref{74})%
\begin{eqnarray*}
2^{-\frac{\lambda q}{\left( q-1\right) ^{2}}-1}D^{-\frac{1}{q-1}} &=&2^{-%
\frac{\lambda q}{\left( q-1\right) ^{2}}-1}\left( C(1-\varepsilon
)^{-2}(1-\delta )^{-(C_{0}+\beta _{2}+d_{2})}\right) ^{-\frac{1}{q-1}} \\
&=&\varepsilon _{0}\left( (1-\varepsilon )^{2}(1-\delta )^{C_{0}+\beta
_{2}+d_{2}}\right) ^{1/\nu }.
\end{eqnarray*}%
Note that the constants $\varepsilon _{0},\theta ,C_{L}$ are all universal,
all of which are independent of the numbers $\varepsilon ,\delta ,$ the ball 
$B$ and the functions $f,u$.

The inequality (\ref{75}) is just the hypothesis (\ref{eq:vol_317}). With a
choice of $\varepsilon _{0},\theta ,C_{L}$ in (\ref{eq:epsdef}), the
assumption (\ref{eq:tildem0small}) is satisfied, and hence, we have (\ref%
{eq:limitmk}). Therefore, 
\begin{equation*}
\frac{\mu (\delta B\cap \{u<\varepsilon a\})}{\mu (\delta B)}=0,
\end{equation*}%
thus showing that (\ref{eq:a/2}) is true.

Finally, the implication (\ref{76}) follows directly from (\ref{eq:ABB->EP-1}%
) and (\ref{60}). The proof is complete.
\end{proof}

\section[wehi]{Proof of weak elliptic Harnack inequality}

\label{sec:pf}In this section, we prove Theorem \ref{thm:M1}.

\begin{proposition}
\label{prop:3.1} If $v\in \mathcal{F^{\prime }}$ and $v\geq 0$ in $B_{R}$
with $0<R<\overline{R}$, then 
\begin{equation}
T_{\frac{3}{4}B,B}(v_{-})\leq T_{\frac{3}{4}B_{R},B_{R}}(v_{-})  \label{tj3}
\end{equation}%
for any $B\subset \frac{3}{4}B_{R}$, where $T_{U,\Omega }(v)$ is defined by (%
\ref{tj2}).
\end{proposition}

\begin{proof}
Since $v\geq 0$ in $B_{R}$, we see that $v_{-}=0$ in $B_{R}$, and hence, 
\begin{align*}
T_{\frac{3}{4}B,B}(v_{-})& =\esup_{x\in \frac{3}{4}B}%
\int_{B^{c}}v_{-}(y)J(x,dy)=\esup_{x\in \frac{3}{4}B}%
\int_{B_{R}^{c}}v_{-}(y)J(x,dy) \\
& \leq \esup_{x\in B}\int_{B_{R}^{c}}v_{-}(y)J(x,dy)\leq \esup_{x\in \frac{3%
}{4}B_{R}}\int_{B_{R}^{c}}v_{-}(y)J(x,dy) \\
& =T_{\frac{3}{4}B_{R},B_{R}}(v_{-}),
\end{align*}%
thus showing (\ref{tj3}).
\end{proof}

We remark that an alternative version of the tail for a function $v$ outside
a ball $B(x_{0},R)$\ is defined in \cite{ChenKumagaiWang.2016.EH} by%
\begin{equation}
\text{Tail}_{w}(v;x_{0},R):=\int_{B(x_{0},R)^{c}}\frac{|v(z)|}{%
V(x_{0},d(x_{0},z))w(x_{0},d(x_{0},z))}\mu (dz).  \label{62}
\end{equation}%
If \emph{condition} $(\mathrm{J}_{\leq })$ holds, that is, if $%
J(dx,dy)=J(x,y)\mu (dx)\mu (dy)$ for a non-negative function $J(x,y)$ with 
\begin{equation}
J(x,y)\leq \frac{C}{V(x,d(x,y))w(x,d(x,y))}\text{ }  \label{63}
\end{equation}%
for any $(x,y)$ in $M\times M\setminus \mathrm{diag}$, for some constant $%
C\geq 0$, then for any function $v$ and any ball $B_{R}:=B(x_{0},R)$ with $%
0<R<\overline{R}$, 
\begin{equation}
T_{\frac{3}{4}B_{R},B_{R}}(v)\leq C^{\prime }\text{\thinspace Tail}%
_{w}(v;x_{0},R)  \label{tj4}
\end{equation}%
for a constant $C^{\prime }>0$ independent of $B_{R},v$.

Indeed, for any two points $x\in \frac{3}{4}B_{R}$ and $y\in B_{R}^{c}$,
since $d(x,y)\geq \frac{R}{4}$ and $d(x_{0},x)\leq \frac{3}{4}R$, it follows
by using (\ref{eq:vol_3}) and the triangle inequality that 
\begin{align}
\frac{V(x_{0},d(x_{0},y))}{V(x,d(x,y))}& \leq \frac{%
V(x_{0},d(x_{0},x)+d(x,y))}{V(x,d(x,y))}\leq C_{\mu }\left( \frac{%
d(x_{0},x)+d(x_{0},x)+d(x,y)}{d(x,y)}\right) ^{d_{2}}  \notag \\
& =C_{\mu }\left( 1+\frac{2d(x_{0},x)}{d(x,y)}\right) ^{d_{2}}\leq C_{\mu
}\left( 1+\frac{2\cdot \frac{3}{4}R}{\frac{1}{4}R}\right) ^{d_{2}}=C_{\mu
}7^{d_{2}},  \label{eq:vol_351}
\end{align}%
whilst by using (\ref{eq:vol_0}), 
\begin{align}
\frac{w(x_{0},d(x_{0},y))}{w(x,d(x,y))}& \leq \frac{%
w(x_{0},d(x_{0},x)+d(x,y))}{w(x,d(x,y))}\leq C_{2}\left( \frac{%
d(x_{0},x)+d(x,y)}{d(x,y)}\right) ^{\beta _{2}}  \notag \\
& =C_{2}\left( 1+\frac{d(x_{0},x)}{d(x,y)}\right) ^{\beta _{2}}\leq
C_{2}\left( 1+\frac{\frac{3}{4}R}{\frac{1}{4}R}\right) ^{\beta
_{2}}=C_{2}4^{\beta _{2}}.  \label{eq:vol_352}
\end{align}%
Therefore, by (\ref{63}), (\ref{eq:vol_351}), (\ref{eq:vol_352}) 
\begin{align*}
T_{\frac{3}{4}B_{R},B_{R}}(v)& =\esup_{x\in \frac{3}{4}B_{R}}%
\int_{B_{R}^{c}}|v(y)|J(x,dy)\leq \esup_{x\in \frac{3}{4}B_{R}}%
\int_{B_{R}^{c}}\frac{C|v(y)|}{V(x,d(x,y))w(x,d(x,y))}\mu (dy) \\
& \leq \int_{B_{R}^{c}}\frac{C(C_{\mu }7^{d_{2}})(C_{2}4^{\beta _{2}})|v(y)|%
}{V(x_{0},d(x_{0},y))w(x_{0},d(x_{0},y))}\mu (dy)=C^{\prime }\text{Tail}%
_{w}(v;x_{0},R),
\end{align*}%
thus showing (\ref{tj4}).

The inequality (\ref{tj4}) says that the tail of a function $v$ defined in
this paper is slightly weaker than that defined in \cite%
{ChenKumagaiWang.2016.EH}, and therefore, so is the weak elliptic Harnack
inequality introduced in Definition \ref{df:wehi1}.

\begin{proposition}
\label{prop:3.2} If $u\in \mathcal{F^{\prime }}\cap L^{\infty }$ and $%
\lambda >0$, then $\ln (u_{+}+\lambda )\in \mathcal{F^{\prime }}\cap
L^{\infty }$.
\end{proposition}

\begin{proof}
For $s\in \mathbb{R}$, we define 
\begin{equation*}
F(s)=\ln (s_{+}+\lambda ).
\end{equation*}%
Since $u\in L^{\infty }$, we see that $F(u)\in L^{\infty }$. For any $%
s_{1},s_{2}\in \mathbb{R}$, we assume $(s_{1})_{+}\geq (s_{2})_{+}$ without
loss of generality. Then 
\begin{equation*}
|F(s_{1})-F(s_{2})|=\ln \left( 1+\frac{(s_{1})_{+}-(s_{2})_{+}}{%
(s_{2})_{+}+\lambda }\right) \leq \frac{(s_{1})_{+}-(s_{2})_{+}}{%
(s_{2})_{+}+\lambda }\leq \frac{|s_{1}-s_{2}|}{\lambda }
\end{equation*}%
by using the elementary inequality 
\begin{equation*}
\ln (1+x)\leq x\text{ \ for any }x\geq 0.
\end{equation*}%
Thus, $F$ is Lipschitz on $\mathbb{R}$. Therefore, by \cite{ghh21TP} (see
also \cite[Proposition A.2 in Appendix]{GrigoryanHuHu.2018.AM433} for a
purely jump Dirichlet form), we conclude that 
\begin{equation*}
F(u)\in \mathcal{F^{\prime }},
\end{equation*}%
thus showing that 
\begin{equation*}
F(u)=\ln (u_{+}+\lambda )\in \mathcal{F^{\prime }}\cap L^{\infty }.
\end{equation*}%
The proof is complete.
\end{proof}

The following will be used shortly.

\begin{proposition}
\label{lemma:3.3} (see \cite[Lemma 3.7]{GrigoryanHuHu.2018.AM433}) Let a
function $u\in \mathcal{F^{\prime }}\cap L^{\infty }$ be non-negative in an
open set $B\subset M$ and $\phi \in \mathcal{F^{\prime }}\cap L^{\infty }$
be such that $\phi =0$ in $B^{c}$. Let $\lambda >0$ and set $u_{\lambda
}:=u+\lambda $. Then we have $\phi ^{2}u_{\lambda }^{-1}\in \mathcal{F}(B)$
and 
\begin{eqnarray*}
\mathcal{E}^{(J)}\Big(u,\phi ^{2}u_{\lambda }^{-1}\Big) &\leq &-\frac{1}{2}%
\iint_{B\times B}(\phi ^{2}(x)\wedge \phi ^{2}(y))\Big|\ln \frac{u_{\lambda
}(y)}{u_{\lambda }(x)}\Big|^{2}J(dx,dy) \\
&&+3\mathcal{E}^{(J)}(\phi ,\phi )-2\iint_{B\times B^{c}}u_{\lambda }(y)%
\frac{\phi ^{2}(x)}{u_{\lambda }(x)}J(dx,dy).
\end{eqnarray*}
\end{proposition}

We show the following \emph{crossover lemma}.

\begin{lemma}[the crossover lemma]
\label{thm:M4} Assume that conditions $(\mathrm{VD})$, $(\mathrm{Cap}_{\leq
})$ and $(\mathrm{PI})$ are satisfied. Let $u\in \mathcal{F^{\prime }}\cap
L^{\infty }$ be $f$-superharmonic and non-negative in a ball $%
B_{R}:=B(x_{0},R)$ with radius less than $\overline{R}$. If 
\begin{equation}
\lambda \geq w(B_{R})\Big(T_{\frac{3}{4}B_{R},B_{R}}(u_{-})+||f||_{L^{\infty
}(B_{R})}\Big),  \label{12}
\end{equation}%
then we have 
\begin{equation}
\left( \fint_{B_{r}}u_{\lambda }^{p}d\mu \right) ^{1/p}\left( \fint%
_{B_{r}}u_{\lambda }^{-p}d\mu \right) ^{1/p}\leq C\text{ \ where }u_{\lambda
}=u+\lambda  \label{eq:vol_316}
\end{equation}%
for any $B_{r}:=B(x_{0},r)$ with $0<r\leq \frac{R}{16(4\kappa +1)}$, where $%
C>0,p\in (0,1)$ are two constants independent of $B_{R},r,u,f$, and the
constant $\kappa \geq 1$ comes from condition $(\mathrm{PI})$.
\end{lemma}

\begin{proof}
The proof is motivated by \cite[Section 4]{Moser.1961.CPAM577} and \cite[%
Proposition 5.7]{BiroliMosco.1995.AMPA4125} for diffusions. The key is to
show that the logarithm function $\ln {u_{\lambda }}$ is a $\mathrm{BMO}$ $%
\mathrm{function}$ (cf. Definition \ref{df:BMO} in Appendix). Our result
covers both a diffusion and a jump process.

Let $B:=B(z,r)$ be an arbitrary ball contained in $\frac{3}{4(4\kappa +1)}%
B_{R}$. Without loss of generality, we may assume that%
\begin{equation}
r\leq 2\cdot \frac{3}{4(4\kappa +1)}R=\frac{3}{2(4\kappa +1)}R<R,
\label{rr2}
\end{equation}%
see for example \cite[Remark 3.16]{Bjorn.2011.EMS}. Then%
\begin{equation}
2\kappa B\subset \frac{3}{4}B_{R}=B(x_{0},\frac{3}{4}R),  \label{eq:vol_301}
\end{equation}%
since by the triangle inequality, for any point $x\in 2\kappa B=B(z,2\kappa
r)$, 
\begin{align*}
d(x,x_{0})& \leq d(x,z)+d(z,x_{0})<2\kappa r+\frac{3}{4(4\kappa +1)}R \\
& \leq 2\kappa \cdot \frac{3}{2(4\kappa +1)}R+\frac{3}{4(4\kappa +1)}R=\frac{%
3}{4}R.
\end{align*}

Let $u\in \mathcal{F^{\prime }}\cap L^{\infty }$ be $f$-superharmonic and
non-negative in $B_{R}$. Applying Proposition \ref{prop:3.1} with $B$
replaced by $2\kappa B$, we have 
\begin{equation}
T_{\frac{3}{2}\kappa B,2\kappa B}(u_{-})\leq T_{\frac{3}{4}%
B_{R},B_{R}}(u_{-}).  \label{eq:vol_302}
\end{equation}%
Let $\lambda $ be a number satisfying (\ref{12}). Without loss of
generality, we assume that 
\begin{equation*}
w(B_{R})\Big(T_{\frac{3}{4}B_{R},B_{R}}(u_{-})+||f||_{L^{\infty }(B_{R})}%
\Big)>0\ \ (\text{thus }\lambda >0).
\end{equation*}%
Otherwise, we consider $\lambda +\varepsilon $ for some $\varepsilon >0$ and
then let $\varepsilon \rightarrow 0$. We shall show that%
\begin{equation}
\ln {u_{\lambda }}\subset \mathrm{BMO}\left( \frac{3}{4(4\kappa +1)}%
B_{R}\right) .  \label{31}
\end{equation}

Indeed, note that $\ln (u_{+}+\lambda )\in \mathcal{F^{\prime }}\cap
L^{\infty }$ by using Proposition \ref{prop:3.2}. Applying condition $(%
\mathrm{PI})$ to the function $\ln (u_{+}+\lambda )$, we have by (\ref%
{eq:vol_5}) that 
\begin{align}
& \int_{B}\big(\ln (u_{+}+\lambda )-(\ln (u_{+}+\lambda ))_{B}\big)^{2}d\mu 
\notag \\
& \leq Cw(B)\left\{ \int_{\kappa B}d\Gamma ^{(L)}\langle \ln (u_{+}+\lambda
)\rangle +\iint_{(\kappa B)\times (\kappa B)}\big(\ln (u_{+}(x)+\lambda
)-\ln (u_{+}(y)+\lambda )\big)^{2}J(dx,dy)\right\}  \notag \\
& =Cw(B)\left\{ \int_{\kappa B}d\Gamma ^{(L)}\langle \ln u_{\lambda }\rangle
+\iint_{(\kappa B)\times (\kappa B)}\Big(\ln \frac{u_{\lambda }(x)}{%
u_{\lambda }(y)}\Big)^{2}J(dx,dy)\right\} ,  \label{eq:vol_305}
\end{align}%
where we have used the fact that $u\geq 0$ (thus $u_{+}=u$) in $B_{R}\supset
\kappa B$.

We estimate the right-hand side of (\ref{eq:vol_305}). Indeed, using
condition $(\mathrm{Cap}_{\leq })$ to the two concentric balls $(\kappa B,%
\frac{3}{2}\kappa B)$, we have by (\ref{eq:vol_3}), (\ref{eq:vol_0}) 
\begin{equation}
\mathcal{E}(\phi ,\phi )\leq C\frac{\mu \left( \frac{3}{2}\kappa B\right) }{%
w\left( \frac{3}{2}\kappa B\right) }\leq C^{\prime }\frac{\mu (B)}{w(B)}
\label{eq:vol_304}
\end{equation}%
for some $\phi \in \mathrm{cutoff}\left( \kappa B,\frac{3}{2}\kappa B\right) 
$.

On the other hand, using the Leibniz and chain rules of $d\Gamma
^{(L)}\langle \cdot \rangle $, we see that 
\begin{align}
\int \phi ^{2}d\Gamma ^{(L)}\langle \ln u_{\lambda }\rangle & =-\int \phi
^{2}d\Gamma ^{(L)}\langle u_{\lambda },u_{\lambda }^{-1}\rangle  \notag \\
& =-\int d\Gamma ^{(L)}\langle u_{\lambda },\phi ^{2}u_{\lambda
}^{-1}\rangle +2\int \phi u_{\lambda }^{-1}d\Gamma ^{(L)}\langle u_{\lambda
},\phi \rangle  \notag \\
& =-\mathcal{E}^{(L)}(u_{\lambda },\phi ^{2}u_{\lambda }^{-1})+2\int \phi
u_{\lambda }^{-1}d\Gamma ^{(L)}\langle u_{\lambda },\phi \rangle .
\label{eq:vol_306}
\end{align}%
By the Cauchy-Schwarz inequality, 
\begin{align}
2\int \phi u_{\lambda }^{-1}d\Gamma ^{(L)}\langle u_{\lambda },\phi \rangle
& =2\int \phi d\Gamma ^{(L)}\langle \ln {u_{\lambda }},\phi \rangle \leq 
\frac{1}{2}\int \phi ^{2}d\Gamma ^{(L)}\langle \ln {u_{\lambda }}\rangle
+2\int d\Gamma ^{(L)}\langle \phi \rangle  \notag \\
& =\frac{1}{2}\int \phi ^{2}d\Gamma ^{(L)}\langle \ln {u_{\lambda }}\rangle
+2\mathcal{E}^{(L)}(\phi ,\phi ),  \label{eq:vol_313}
\end{align}%
from which, it follows by (\ref{eq:vol_306}) that%
\begin{equation}
\int \phi ^{2}d\Gamma ^{(L)}\langle \ln u_{\lambda }\rangle \leq -2\mathcal{E%
}^{(L)}(u_{\lambda },\phi ^{2}u_{\lambda }^{-1})+4\mathcal{E}^{(L)}(\phi
,\phi ).  \label{33}
\end{equation}%
We estimate the first term on the right-hand side.

Indeed, since $\phi =0$ in $\left( \frac{3}{2}\kappa B\right) ^{c}\supset
(2\kappa B)^{c}$, using Proposition \ref{lemma:3.3} with $B$ being replaced
by $2\kappa B$, we obtain that $0\leq \phi ^{2}u_{\lambda }^{-1}\in \mathcal{%
F}(2\kappa B)$, and 
\begin{eqnarray}
\mathcal{E}^{(J)}(u,\phi ^{2}u_{\lambda }^{-1}) &\leq &-\frac{1}{2}%
\iint_{(2\kappa B)\times (2\kappa B)}(\phi ^{2}(x)\wedge \phi ^{2}(y))\Big|%
\ln \frac{u_{\lambda }(y)}{u_{\lambda }(x)}\Big|^{2}J(dx,dy)  \notag \\
&&+3\mathcal{E}^{(J)}(\phi ,\phi )-2\iint_{(2\kappa B)\times (2\kappa
B)^{c}}u_{\lambda }(y)\frac{\phi ^{2}(x)}{u_{\lambda }(x)}J(dx,dy).
\label{eq:vol_307}
\end{eqnarray}%
Noting that $\mathcal{E}(u_{\lambda },\phi ^{2}u_{\lambda }^{-1})\geq
(f,\phi ^{2}u_{\lambda }^{-1})$ since $u$ is $f$-superharmonic in $%
B_{R}\supset 2\kappa B$, we see by (\ref{eq:DF-LJ}), (\ref{eq:vol_307}) that 
\begin{align}
-\mathcal{E}^{(L)}(u_{\lambda },\phi ^{2}u_{\lambda }^{-1})& =-\mathcal{E}%
(u_{\lambda },\phi ^{2}u_{\lambda }^{-1})+\mathcal{E}^{(J)}(u_{\lambda
},\phi ^{2}u_{\lambda }^{-1})  \notag \\
& \leq -(f,\phi ^{2}u_{\lambda }^{-1})+\mathcal{E}^{(J)}(u_{\lambda },\phi
^{2}u_{\lambda }^{-1})  \notag \\
& \leq -(f,\phi ^{2}u_{\lambda }^{-1})-\frac{1}{2}\iint_{(2\kappa B)\times
(2\kappa B)}(\phi ^{2}(x)\wedge \phi ^{2}(y))\Big|\ln \frac{u_{\lambda }(y)}{%
u_{\lambda }(x)}\Big|^{2}J(dx,dy)  \notag \\
& +3\mathcal{E}^{(J)}(\phi ,\phi )-2\iint_{(2\kappa B)\times (2\kappa
B)^{c}}u_{\lambda }(y)\frac{\phi ^{2}(x)}{u_{\lambda }(x)}J(dx,dy).
\label{eq:vol_311}
\end{align}%
Since $\phi =1$ in $\kappa B$, we have%
\begin{equation}
-\iint_{(2\kappa B)\times (2\kappa B)}(\phi ^{2}(x)\wedge \phi ^{2}(y))\Big|%
\ln \frac{u_{\lambda }(y)}{u_{\lambda }(x)}\Big|^{2}J(dx,dy)\leq
-\iint_{(\kappa B)\times (\kappa B)}\Big|\ln \frac{u_{\lambda }(y)}{%
u_{\lambda }(x)}\Big|^{2}J(dx,dy),  \label{eq:vol_308}
\end{equation}%
whilst, since $\phi =0$ in $(\frac{3}{2}\kappa B)^{c}$ and $0\leq \phi \leq 1
$ in $M$, 
\begin{align}
-\iint_{(2\kappa B)\times (2\kappa B)^{c}}u_{\lambda }(y)\frac{\phi ^{2}(x)}{%
u_{\lambda }(x)}J(dx,dy)& =-\iint_{(\frac{3}{2}\kappa B)\times (2\kappa
B)^{c}}u_{\lambda }(y)\frac{\phi ^{2}(x)}{u_{\lambda }(x)}J(dx,dy)  \notag \\
& \leq \iint_{(\frac{3}{2}\kappa B)\times (2\kappa B)^{c}}u_{-}(y)\frac{1}{%
u_{\lambda }(x)}J(dx,dy)  \notag \\
& \leq \frac{1}{\lambda }\int_{(\frac{3}{2}\kappa B)}\left\{ \esup_{x\in (%
\frac{3}{2}\kappa B)}\int_{(2\kappa B)^{c}}u_{-}(y)J(x,dy)\right\} \mu (dx) 
\notag \\
& =\frac{1}{\lambda }\mu \left( \frac{3}{2}\kappa B\right) T_{\frac{3}{2}%
\kappa B,2\kappa B}(u_{-})  \notag \\
& \leq \frac{1}{\lambda }C_{\mu }\left( \frac{3}{2}\kappa \right)
^{d_{2}}\mu (B)T_{\frac{3}{4}B_{R},B_{R}}(u_{-}),  \label{eq:vol_309}
\end{align}%
where in the last inequality we have used condition $(\mathrm{VD})$ and
inequality (\ref{eq:vol_302}). From this, condition (\ref{12}), and using
the fact that $\frac{w(B_{R})}{w(B)}\geq C_{1}$ by (\ref{eq:vol_0}), (\ref%
{rr2}), we obtain%
\begin{eqnarray}
-\iint_{(2\kappa B)\times (2\kappa B)^{c}}u_{\lambda }(y)\frac{\phi ^{2}(x)}{%
u_{\lambda }(x)}J(dx,dy) &\leq &\frac{1}{\lambda }C_{\mu }\left( \frac{3}{2}%
\kappa \right) ^{d_{2}}\mu (B)T_{\frac{3}{4}B_{R},B_{R}}(u_{-})  \notag \\
&\leq &\frac{1}{w(B_{R})T_{\frac{3}{4}B_{R},B_{R}}(u_{-})}C_{\mu }\left( 
\frac{3}{2}\kappa \right) ^{d_{2}}\mu (B)T_{\frac{3}{4}B_{R},B_{R}}(u_{-}) 
\notag \\
&=&C_{\mu }\left( \frac{3}{2}\kappa \right) ^{d_{2}}\frac{\mu (B)}{w(B_{R})}%
\leq C\frac{\mu (B)}{w(B)}.  \label{eq:vol_310}
\end{eqnarray}%
Therefore, plugging (\ref{eq:vol_308}) and (\ref{eq:vol_310}) into (\ref%
{eq:vol_311}), we obtain 
\begin{equation}
-\mathcal{E}^{(L)}(u_{\lambda },\phi ^{2}u_{\lambda }^{-1})\leq -(f,\phi
^{2}u_{\lambda }^{-1})-\frac{1}{2}\iint_{(\kappa B)\times (\kappa B)}\Big|%
\ln \frac{u_{\lambda }(y)}{u_{\lambda }(x)}\Big|^{2}J(dx,dy)+3\mathcal{E}%
^{(J)}(\phi ,\phi )+2C\frac{\mu (B)}{w(B)}.  \label{32}
\end{equation}%
Plugging (\ref{32}), (\ref{eq:vol_304}) into (\ref{33}), it follows that%
\begin{eqnarray*}
\int \phi ^{2}d\Gamma ^{(L)}\langle \ln u_{\lambda }\rangle  &\leq &-2%
\mathcal{E}^{(L)}(u_{\lambda },\phi ^{2}u_{\lambda }^{-1})+4\mathcal{E}%
^{(L)}(\phi ,\phi ) \\
&\leq &-2(f,\phi ^{2}u_{\lambda }^{-1})-\iint_{(\kappa B)\times (\kappa B)}%
\Big|\ln \frac{u_{\lambda }(y)}{u_{\lambda }(x)}\Big|^{2}J(dx,dy)+6\mathcal{E%
}^{(J)}(\phi ,\phi ) \\
&&+4C\frac{\mu (B)}{w(B)}+4\mathcal{E}^{(L)}(\phi ,\phi ) \\
&\leq &-2(f,\phi ^{2}u_{\lambda }^{-1})-\iint_{(\kappa B)\times (\kappa B)}%
\Big|\ln \frac{u_{\lambda }(y)}{u_{\lambda }(x)}\Big|^{2}J(dx,dy)+C^{\prime }%
\frac{\mu (B)}{w(B)},
\end{eqnarray*}%
from which, using the fact that $\phi =1$ in $\kappa B$, we have 
\begin{equation}
\int_{\kappa B}d\Gamma ^{(L)}\langle \ln u_{\lambda }\rangle +\iint_{(\kappa
B)\times (\kappa B)}\Big|\ln \frac{u_{\lambda }(y)}{u_{\lambda }(x)}\Big|%
^{2}J(dx,dy)\leq -2(f,\phi ^{2}u_{\lambda }^{-1})+C^{\prime }\frac{\mu (B)}{%
w(B)}.  \label{eq:vol_314}
\end{equation}%
Since 
\begin{align*}
-2(f,\phi ^{2}u_{\lambda }^{-1})& =-2\int_{\frac{3}{2}\kappa B}f\phi
^{2}u_{\lambda }^{-1}d\mu \leq 2\int_{\frac{3}{2}\kappa B}|f|u_{\lambda
}^{-1}d\mu  \\
& \leq 2\int_{\frac{3}{2}\kappa B}\frac{||f||_{L^{\infty }(B_{R})}}{\lambda }%
d\mu \leq \frac{2\mu (\frac{3}{2}\kappa B)}{w(B_{R})}\ \ \ (\text{by using (%
\ref{12}))} \\
& \leq C\frac{\mu (B)}{w(B)}\ \ \ \ \ \ \text{(by condition }(\mathrm{VD})\ 
\text{and (\ref{eq:vol_0}))},
\end{align*}%
we have by plugging (\ref{eq:vol_314}) into (\ref{eq:vol_305}), 
\begin{equation*}
\int_{B}\big(\ln {u_{\lambda }}-(\ln {u_{\lambda }})_{B}\big)^{2}d\mu \leq
Cw(B)\cdot \Big(-2(f,\phi ^{2}u_{\lambda }^{-1})+C^{\prime }\frac{\mu (B)}{%
w(B)}\Big)\leq C\mu (B),
\end{equation*}%
which yields that, using the Cauchy-Schwarz inequality, 
\begin{equation*}
\left( \int_{B}\left\vert \ln {u_{\lambda }}-(\ln {u_{\lambda }}%
)_{B}\right\vert d\mu \right) ^{2}\leq \mu (B)\left( \int_{B}\big(\ln {%
u_{\lambda }}-(\ln {u_{\lambda }})_{B}\big)^{2}d\mu \right) \leq C^{\prime
\prime }\mu (B)^{2},
\end{equation*}%
that is, 
\begin{equation}
\fint_{B}\left\vert \ln {u_{\lambda }}-(\ln {u_{\lambda }})_{B}\right\vert
d\mu \leq C_{3}  \label{eq:vol_315}
\end{equation}%
for all balls $B$ in $\frac{3}{4(4\kappa +1)}B_{R}$ and all $\lambda $
satisfying (\ref{12}), where $C_{3}$ is a universal constant independent of $%
B_{R},B,\lambda $ and the functions $u,f$, thus proving (\ref{31}).

Applying Corollary \ref{cor:JN} in Appendix with function $\ln {u_{\lambda }}
$ and $B_{0}=\frac{3}{4(4\kappa +1)}B_{R}$, we have 
\begin{equation*}
\left\{ \fint_{B}\exp \left( \frac{c_{2}}{2b}\ln u_{\lambda }\right) d\mu
\right\} \left\{ \fint_{B}\exp \left( -\frac{c_{2}}{2b}\ln u_{\lambda
}\right) d\mu \right\} \leq (1+c_{1})^{2}
\end{equation*}%
for any ball $B$ satisfies $12B\subseteq \frac{3}{4(4\kappa +1)}B_{R}$ and
any 
\begin{equation}
b\geq ||\ln {u_{\lambda }}||_{\mathrm{BMO}\left( \frac{3}{4(4\kappa +1)}%
B_{R}\right) }.  \label{34}
\end{equation}%
In particular, for any $B_{r}:=B(x_{0},r)$ with $0<r\leq \frac{R}{16(4\kappa
+1)}$ so that $12B_{r}\subseteq \frac{3}{4(4\kappa +1)}B_{R}$ and for any
number $b$ satisfying (\ref{34}), 
\begin{equation}
\left\{ \fint_{B_{r}}\exp \left( \frac{c_{2}}{2b}\ln u_{\lambda }\right)
d\mu \right\} \left\{ \fint_{B_{r}}\exp \left( -\frac{c_{2}}{2b}\ln
u_{\lambda }\right) d\mu \right\} \leq (1+c_{1})^{2}.  \label{35}
\end{equation}

Finally, choosing $b=\frac{c_{2}}{2}+C_{3}$ so that (\ref{34}) is satisfied
and letting $p:=\frac{c_{2}}{2b}\in (0,1)$, we conclude from (\ref{35}) that 
\begin{equation*}
\left\{ \fint_{B_{r}}\exp \left( p\ln u_{\lambda }\right) d\mu \right\}
\left\{ \fint_{B_{r}}\exp \left( -p\ln u_{\lambda }\right) d\mu \right\}
\leq (1+c_{1})^{2},
\end{equation*}%
thus showing (\ref{eq:vol_316}). The proof is complete.
\end{proof}

We are now in a position to prove Theorem \ref{thm:M1}.

\begin{proof}[Proof of Theorem \protect\ref{thm:M1}]
We need to show the implication (\ref{20}). Indeed, by Lemma \ref{lemma:LG},
condition $(\mathrm{LG})$ is true. Let $B_{R}:=B(x_{0},R)$ be a metric ball
in $M$ with $0<R<\sigma \overline{R}$, where constant $\sigma $ comes from
condition $(\mathrm{LG})$. Let $B_{r}:=B(x_{0},r)$ with 
\begin{equation}
0<r\leq \delta R\text{ \ where }\delta :=\frac{1}{32(4\kappa +1)}  \label{23}
\end{equation}%
and $\kappa $ is the same constant as in condition $(\mathrm{PI})$. Let $%
u\in \mathcal{F^{\prime }}\cap L^{\infty }$ be a function that is
non-negative, $f$-superharmonic in $B_{R}$. We need to show that%
\begin{equation}
\left( \frac{1}{\mu (B_{r})}\int_{B_{r}}u^{p}d\mu \right) ^{1/p}\leq C\left( %
\einf_{{B}_{r}}u+w(B_{r})\left( T_{\frac{3}{4}B_{R},B_{R}}(u_{-})+||f||_{L^{%
\infty }(B_{R})}\right) \right)  \label{24}
\end{equation}%
for some universal numbers $p\in (0,1)$ and $C\geq 1$, both of which are
independent of $B_{R},r,u,f$.

To do this, let $\lambda $ be a number determined by 
\begin{equation}
\lambda =w(B_{R})\left( T_{\frac{3}{4}B_{R},B_{R}}(u_{-})+||f||_{L^{\infty
}(B_{R})}\right).  \label{25}
\end{equation}%
We claim that for any $r\in (0,\delta R]$ 
\begin{equation}
\left( \fint_{B_{r}}u_{\lambda }^{p}d\mu \right) ^{1/p}\leq C\einf_{{B_{r}}%
}u_{\lambda }\text{ with }u_{\lambda }=u+\lambda  \label{eq:vol_318}
\end{equation}%
for some constant $C$ independent of $B_{R},r$, $u,f$.

Indeed, by Lemma \ref{thm:M4}, there exist two positive constants $p\in
(0,1) $ and $c^{\prime }$, independent of $B_{R},r,u,f$, such that 
\begin{equation}
\left( \fint_{B_{r}}u_{\lambda }^{p}d\mu \right) ^{1/p}\left( \fint%
_{B_{r}}u_{\lambda }^{-p}d\mu \right) ^{1/p}\leq c^{\prime }
\label{eq:vol_329}
\end{equation}%
for any $0<r\leq 2\delta R$. Let $s=p/\theta $, where $\theta =\frac{1}{\nu }
$ with constant $\nu $ coming from condition $(\mathrm{FK})$. Without loss
of generality, assume $\theta \geq 1$. Thus $s\in (0,1)$. Let%
\begin{equation*}
b:=w(B_{r})\left( T_{\frac{3}{2}B_{r},2B_{r}}\big((u_{\lambda })_{-}\big)%
+||f||_{L^{\infty }(B_{r})}\right) .
\end{equation*}%
Define the function $g$ by%
\begin{equation}
g(a)=a^{s}(1+\frac{2b}{a})\text{ for any }a\in (0,+\infty ).  \label{gg}
\end{equation}%
Using the facts that $(u_{\lambda })_{-}\leq u_{-}$ in $M$\ and $%
2B_{r}\subset \frac{3}{4}B_{R}$, we have, by Proposition \ref{prop:3.1} with 
$B$ being replaced by $2B_{r}$, that 
\begin{align}
b& =w(B_{r})\left( T_{\frac{3}{2}B_{r},2B_{r}}\big((u_{\lambda })_{-}\big)%
+||f||_{L^{\infty }(B_{r})}\right) \leq w(B_{r})\left( T_{\frac{3}{4}%
(2B_{r}),2B_{r}}(u_{-})+||f||_{L^{\infty }(B_{r})}\right)  \notag \\
& \leq w(B_{R})\left( T_{\frac{3}{4}B_{R},B_{R}}(u_{-})+||f||_{L^{\infty
}(B_{R})}\right) =\lambda .  \label{26}
\end{align}%
Clearly, for any $a>\lambda $, 
\begin{equation}
\frac{\mu (B_{r}\cap \{u_{\lambda }<a\})}{\mu (B_{r})}=\frac{\mu (B_{r}\cap
\{u_{\lambda }^{-p}>a^{-p}\})}{\mu (B_{r})}\leq a^{p}\fint_{B_{r}}u_{\lambda
}^{-p}d\mu .  \label{27}
\end{equation}

Note that $u_{\lambda }\in \mathcal{F^{\prime }}\cap L^{\infty }$ is $f$%
-superharmonic, non-negative in $2B_{r}\subset B_{R}$. To look at whether
the hypotheses in condition $(\mathrm{LG})$ are satisfied or not, we
consider two cases.

\emph{Case }$\emph{1}$. Assume that there exists a number $a$ ($>\lambda $)
such that 
\begin{eqnarray}
\varepsilon _{0}2^{-(2+C_{L})\theta }\left( 1+\frac{2b}{a}\right) ^{-\theta
} &=&\varepsilon _{0}2^{-(2+C_{L})\theta }\left( 1+\frac{2w(B_{r})\left( T_{%
\frac{3}{2}B_{r},2B_{r}}\big((u_{\lambda })_{-}\big)+||f||_{L^{\infty
}(B_{r})}\right) }{a}\right) ^{-\theta }  \notag \\
&=&a^{p}\fint_{B_{r}}u_{\lambda }^{-p}d\mu ,  \label{lge}
\end{eqnarray}%
that is, 
\begin{equation}
\big(g(a)\big)^{1/s}=a\left( 1+\frac{2b}{a}\right) ^{1/s}=\varepsilon
_{1}^{1/p}\left( \fint_{B_{r}}u_{\lambda }^{-p}d\mu \right) ^{-1/p},
\label{eq:vol_319}
\end{equation}%
where the constant $C_{L}$ comes from condition $(\mathrm{LG})$ and $%
\varepsilon _{1}:=\varepsilon _{0}2^{-(2+C_{L})\theta }$. In this case, by
using (\ref{27}) and (\ref{lge}), we have 
\begin{align*}
\frac{\mu (B_{r}\cap \{u_{\lambda }<a\})}{\mu (B_{r})}& \leq a^{p}\fint%
_{B_{r}}u_{\lambda }^{-p}d\mu =\varepsilon _{0}2^{-(2+C_{L})\theta }\left( 1+%
\frac{2w(B_{r})\left( T_{\frac{3}{2}B_{r},2B_{r}}\big((u_{\lambda })_{-}\big)%
+||f||_{L^{\infty }(B_{r})}\right) }{a}\right) ^{-\theta } \\
& \leq \varepsilon _{0}2^{-(2+C_{L})\theta }\left( 1+\frac{2w(B_{r})\left(
T_{\frac{7}{8}B_{r},B_{r}}\big((u_{\lambda })_{-}\big)+||f||_{L^{\infty
}(B_{r})}\right) }{a}\right) ^{-\theta } \\
& =\varepsilon _{0}(1-1/2)^{2\theta }(1-1/2)^{C_{L}\theta }\left( 1+\frac{%
w(B_{r})\left( T_{\frac{3+1/2}{4}B_{r},B_{r}}\big((u_{\lambda })_{-}\big)%
+||f||_{L^{\infty }(B_{r})}\right) }{\frac{1}{2}a}\right) ^{-\theta },
\end{align*}%
since $T_{\frac{7}{8}B_{r},B_{r}}\big((u_{\lambda })_{-}\big)\leq T_{\frac{3%
}{2}B_{r},2B_{r}}\big((u_{\lambda })_{-}\big)$ by noting that $u_{\lambda }$
is non-negative in $2B_{r}$. Therefore, we see that the assumption (\ref%
{eq:vol_317}), with $B$ being replaced by $B_{r}$ and $u$ replaced by $%
u_{\lambda }$, is true. Thus, all the hypothesis in condition $\mathrm{LG}%
(\varepsilon ,\delta )$ are satisfied with $\varepsilon =\delta =1/2$.
Therefore, it follows that 
\begin{align*}
\einf_{{\frac{1}{2}B_{r}}}u_{\lambda }& \geq \frac{1}{2}a=\frac{1}{2}\left(
1+\frac{2b}{a}\right) ^{-1/s}\varepsilon _{1}^{1/p}\left( \fint%
_{B_{r}}u_{\lambda }^{-p}d\mu \right) ^{-1/p}\text{ \ (using (\ref%
{eq:vol_319}))} \\
& \geq \frac{1}{2}\left( 1+\frac{2b}{a}\right) ^{-1/s}\varepsilon _{1}^{1/p}{%
c^{\prime }}^{-1}\left( \fint_{B_{r}}u_{\lambda }^{p}d\mu \right) ^{1/p}%
\text{ \ (using (\ref{eq:vol_329}))} \\
& \geq \frac{1}{2}3^{-1/s}\varepsilon _{1}^{1/p}{c^{\prime }}^{-1}\left( %
\fint_{B_{r}}u_{\lambda }^{p}d\mu \right) ^{1/p}\text{ \ (using (\ref{26})
and }a>\lambda \text{),}
\end{align*}%
which gives that 
\begin{equation*}
\left( \fint_{B_{r}}u_{\lambda }^{p}d\mu \right) ^{1/p}\leq 2c^{\prime
}\varepsilon _{1}^{-1/p}3^{1/s}\einf_{{\frac{1}{2}B_{r}}}u_{\lambda }.
\end{equation*}%
Thus, the inequality (\ref{eq:vol_318}) is true in this case.

\emph{Case }$\emph{2}$\emph{.} Assume that (\ref{eq:vol_319}) is not
satisfied for any $a\in (\lambda ,+\infty )$. In this case, noting that 
\begin{equation*}
\lim_{a\rightarrow +\infty }g(a)=+\infty
\end{equation*}%
and $g$ is continuous on $(0,+\infty )$, we have that 
\begin{equation}
\big(g(a)\big)^{1/s}>\varepsilon _{1}^{1/p}\left( \fint_{B_{r}}u_{\lambda
}^{-p}d\mu \right) ^{-1/p},  \label{g2}
\end{equation}%
for any $a\in (\lambda ,+\infty )$.

If $\lambda =0$, then $b=0$ by (\ref{26}). By definition (\ref{gg}), we have 
$g(a)=a^{s}$ for any $a>0$. From this and using (\ref{g2}), it follows that 
\begin{equation*}
\varepsilon _{1}^{1/p}\left( \fint_{B_{r}}u_{\lambda }^{-p}d\mu \right)
^{-1/p}<g(a)^{1/s}=a\text{ for any }a\in (0,+\infty ).
\end{equation*}%
Letting $a\rightarrow 0$, we have $\left( \fint_{B_{r}}u_{\lambda }^{-p}d\mu
\right) ^{-1/p}=0$, which gives that%
\begin{equation*}
\left( \fint_{B_{r}}u_{\lambda }^{p}d\mu \right) ^{1/p}=0
\end{equation*}%
by using (\ref{eq:vol_329}), thus showing (\ref{eq:vol_318}).

In the sequel, we assume that $\lambda >0$. Since $g$ is continuous on $%
(0,+\infty )$, we have from (\ref{g2}) by letting $a\searrow \lambda $ that%
\begin{equation*}
\big(g(\lambda )\big)^{1/s}\geq \varepsilon _{1}^{1/p}\left( \fint%
_{B_{r}}u_{\lambda }^{-p}d\mu \right) ^{-1/p},
\end{equation*}%
from which, we see by using (\ref{26}) 
\begin{equation}
3^{1/s}\lambda \geq \lambda \left( 1+\frac{2b}{\lambda }\right) ^{1/s}=\big(%
g(\lambda )\big)^{1/s}\geq \varepsilon _{1}^{1/p}\left( \fint%
_{B_{r}}u_{\lambda }^{-p}d\mu \right) ^{-1/p}.  \label{eq:vol_321}
\end{equation}%
Thus, we have 
\begin{eqnarray}
\left( \fint_{B_{r}}u_{\lambda }^{p}d\mu \right) ^{1/p} &\leq
&c^{\prime}\left( \fint_{B_{r}}u_{\lambda }^{-p}d\mu \right) ^{-1/p}\text{ \
(using (\ref{eq:vol_329}))}  \notag \\
&\leq &c^{\prime}\varepsilon _{1}^{-1/p}3^{1/s}\lambda \text{ \ (using (\ref%
{eq:vol_321})).}  \label{eq:vol_324}
\end{eqnarray}

Therefore, combining Case $1$ and Case $2$, we always have 
\begin{equation}
\left( \fint_{B_{r}}u_{\lambda }^{p}d\mu \right) ^{1/p}\leq C(\lambda +\einf%
_{{\frac{1}{2}B_{r}}}u_{\lambda })\leq 2C\einf_{{\frac{1}{2}B_{r}}%
}u_{\lambda }  \label{eq:vol_325}
\end{equation}%
for any $0<r\leq 2\delta R$.

On the other hand, by condition $(\mathrm{VD})$, 
\begin{equation}
\fint_{B_{r}}u_{\lambda }^{p}d\mu \geq \frac{1}{\mu (B_{r})}\int_{\frac{1}{2}%
B_{r}}u_{\lambda }^{p}d\mu \geq \frac{1}{C_{\mu }\mu \left( \frac{1}{2}%
B_{r}\right) }\int_{\frac{1}{2}B_{r}}u_{\lambda }^{p}d\mu ,
\label{eq:vol_326}
\end{equation}%
from which, it follows by (\ref{eq:vol_325}) that 
\begin{equation*}
\left( \fint_{\frac{1}{2}B_{r}}u_{\lambda }^{p}d\mu \right) ^{1/p}\leq
C^{\prime }\einf_{{\frac{1}{2}B_{r}}}u_{\lambda }
\end{equation*}%
for $0<r\leq 2\delta R$, thus proving our claim (\ref{eq:vol_318}) by
renaming $r/2$ by $r$, as desired.

Therefore, we obtain by (\ref{eq:vol_318}) 
\begin{align}
\left( \frac{1}{\mu (B_{r})}\int_{B_{r}}u^{p}d\mu \right) ^{1/p}& \leq
\left( \frac{1}{\mu (B_{r})}\int_{B_{r}}u_{\lambda }^{p}d\mu \right)
^{1/p}\leq C^{\prime }\einf_{{B_{r}}}u_{\lambda }  \notag \\
& =C^{\prime }\left( \einf_{{B_{r}}}u+w(B_{R})\left( T_{\frac{3}{4}%
B_{R},B_{R}}(u_{-})+||f||_{L^{\infty }(B_{R})}\right) \right)
\label{eq:vol_327}
\end{align}%
for $0<r\leq \delta R$.

Finally, we show that the term $w(B_{R})$ on the right-hand side of (\ref%
{eq:vol_327}) can be replaced by a smaller one $w(B_{r})$ for any $0<r\leq
\delta R$, by adjusting the value of constant $C^{\prime }$.

Indeed, fix a number $r$ in $(0,\delta R)$. Let $i\geq 1$ be an integer such
that, setting $r_{k}=$ $\delta ^{k}R$ for any $k\geq 0$, 
\begin{equation}
r_{i+1}=\delta ^{i+1}R\leq r<\delta ^{i}R=r_{i}.  \label{328}
\end{equation}%
By Proposition \ref{prop:3.1}, we see that 
\begin{equation}
T_{\frac{3}{4}B_{r_{i-1}},B_{r_{i-1}}}(u_{-})\leq T_{\frac{3}{4}%
B_{R},B_{R}}(u_{-}).  \label{eq:vol_328}
\end{equation}%
By (\ref{eq:vol_0}) and (\ref{328}),%
\begin{equation}
w(B_{r_{i-1}})=\frac{w(x_{0},\delta ^{i-1}R)}{w(x_{0},r)}w(x_{0},r)\leq
C_{2}\left( \frac{\delta ^{i-1}R}{r}\right) ^{\beta _{2}}w(B_{r})\leq
C_{2}\delta ^{-2\beta _{2}}w(B_{r}).  \label{41}
\end{equation}

Since $u$ is $f$-superharmonic in $B_{r_{i-1}}$, applying (\ref{eq:vol_327})
with $R$ being replaced by $r_{i-1}$ and then using (\ref{eq:vol_328}), (\ref%
{41}), we conclude that 
\begin{align*}
\left( \frac{1}{\mu (B_{r})}\int_{B_{r}}u^{p}d\mu \right) ^{1/p}& \leq
C^{\prime }\left( \einf_{{B_{r}}}u+w(B_{r_{i-1}})\left( T_{\frac{3}{4}%
B_{r_{i-1}},B_{r_{i-1}}}(u_{-})+||f||_{L^{\infty }(B_{r_{i-1}})}\right)
\right) \\
& =C^{\prime }\left( \einf_{{B_{r}}}u+C_{2}\delta ^{-2\beta
_{2}}w(B_{r})\left( T_{\frac{3}{4}B_{R},B_{R}}(u_{-})+||f||_{L^{\infty
}(B_{R})}\right) \right) \\
& \leq C\left( \einf_{{B_{r}}}u+w(B_{r})\left( T_{\frac{3}{4}%
B_{R},B_{R}}(u_{-})+||f||_{L^{\infty }(B_{R})}\right) \right) ,
\end{align*}%
thus showing that condition $(\mathrm{wEH})$ holds. The proof is complete.
\end{proof}

\section{Other equivalent characterizations}

\label{sec:wEHI}In this section, we prove Theorem \ref{thm:M10}. Denote by 
\begin{equation}
\omega _{B}(A):=\frac{\mu (A\cap B)}{\mu (B)},  \label{42}
\end{equation}%
the \emph{occupation measure} of the set $A$ in $B$.

The following version of the weak elliptic Harnack inequality was introduced
in \cite[Proposition 3.6]{ChenKumagaiWang.2016.EH} when $f\equiv 0$, and we
label it by condition $(\mathrm{wEH}1)$.

\begin{definition}[condition (wEH$1$)]
\label{df:wehi2}We say that condition $(\mathrm{wEH}1)$ holds if there exist
two universal constants $\sigma \in (0,1)$ and $\delta _{1}\in (0,1/4)$ such
that, for any two concentric balls $B_{R}:=B(x_{0},R)\supset
B(x_{0},r)=:B_{r}$ with $R\in (0,\sigma \overline{R})$, $r\in (0,\delta
_{1}R)$, any function $f\in L^{\infty }(B_{R})$, any number $\eta _{1}\in
(0,1]$ and for any function $u\in \mathcal{F}^{\prime }\cap L^{\infty }$
which is non-negative, $f$-superharmonic in $B_{R}$, if for some $a>0$, 
\begin{equation*}
\omega _{B_{r}}(\left\{ u\geq a\right\} )=\frac{\mu (B_{r}\cap \left\{ u\geq
a\right\} )}{\mu (B_{r})}\geq \eta _{1},
\end{equation*}%
then 
\begin{equation}
\einf_{{B_{4r}}}u>\varepsilon _{1}a-w(B_{r})\left( T_{\frac{3}{4}%
B_{R},B_{R}}(u_{-})+||f||_{L^{\infty }(B_{R})}\right) ,  \label{eq:vol_402}
\end{equation}%
where $\varepsilon _{1}=\varepsilon _{1}(\eta _{1})\in (0,1)$ depends only
on $\eta _{1}$ (independent of $x_{0},r,R,f,u,a$).
\end{definition}

We show that condition $(\mathrm{wEH})$ defined in Definition \ref{df:wehi1}
is equivalent to condition $(\mathrm{wEH}1)$.

\begin{proposition}
\label{prop:4.2}Assume that $(\mathcal{E},\mathcal{F})$ is a Dirichlet form
in $L^{2}(M,\mu)$. If condition $(\mathrm{VD})$ holds, then 
\begin{equation*}
(\mathrm{wEH})\Leftrightarrow (\mathrm{wEH}1).
\end{equation*}
\end{proposition}

\begin{proof}
The proof was essentially given in \cite[Proof of Theorem 3.1 and Remark 3.9]%
{ChenKumagaiWang.2016.EH} wherein the jump kernel is assumed to exist and $%
f\equiv 0$. For the reader's convenience, we sketch the proof. We mention
that the jump kernel here may not exist.

We first show $(\mathrm{wEH})$ $\Rightarrow $ $(\mathrm{wEH}1)$.

\noindent Assume that condition $(\mathrm{wEH})$ holds. Let $u\in \mathcal{%
F^{\prime }}\cap L^{\infty }$ be non-negative, $f$-superharmonic in a ball $%
B_{R}(x_{0})$ with $R\in (0,\sigma \overline{R})$. Let $\eta _{1}$ be any
number in $(0,1]$ and $r$ any number in $(0,\delta R/4)$, where constant $%
\delta $ is the same as in condition $(\mathrm{wEH})$. Assume that 
\begin{equation}
\omega _{B_{r}}(\left\{ u\geq a\right\} )\geq \eta _{1}  \label{43}
\end{equation}%
for some $a>0$. We will show that condition $(\mathrm{wEH}1)$ holds with $%
\delta _{1}=\frac{\delta }{4}$ and 
\begin{equation}
\varepsilon _{1}(\eta _{1})=\left( C_{2}4^{\beta _{2}}C_{H}\right)
^{-1}\left( \frac{\eta _{1}}{C_{\mu }^{2}}\right) ^{1/p},  \label{45}
\end{equation}%
where constants $C_{2},\beta _{2}$ are the same as in (\ref{eq:vol_0}) and $%
C_{H},p$ the same as in condition $(\mathrm{wEH})$, while the number $C_{\mu
}$ comes from (\ref{eq:vol_1}). It suffices to show (\ref{eq:vol_402}).

Indeed, we have by (\ref{eq:vol_65}), with $r$ replaced by $4r$, that 
\begin{equation}
\left( \fint_{B_{4r}}u^{p}d\mu \right) ^{1/p}\leq C_{H}\left( \einf%
_{B_{4r}}u+w(x_{0},4r)\left( T_{\frac{3}{4}B_{R},B_{R}}(u_{-})+||f||_{L^{%
\infty }(B_{R})}\right) \right) .  \label{eq:vol_411}
\end{equation}%
Since $\mu (B_{4r})\leq C_{\mu }^{2}\mu (B_{r})$ by condition $(\mathrm{VD})$%
, we have by (\ref{43}) 
\begin{align}
\left( \fint_{B_{4r}}u^{p}d\mu \right) ^{1/p}& \geq \left( \frac{1}{C_{\mu
}^{2}\mu (B_{r})}\int_{B_{r}}u^{p}d\mu \right) ^{1/p}\geq \left( \frac{1}{%
C_{\mu }^{2}\mu (B_{r})}\int_{B_{r}\cap \{u\geq a\}}a^{p}d\mu \right) ^{1/p}
\notag \\
& =a\left( \frac{\omega _{B_{r}}(\{u\geq a\})}{C_{\mu }^{2}}\right)
^{1/p}\geq \left( \frac{\eta _{1}}{C_{\mu }^{2}}\right) ^{1/p}a.
\label{eq:vol_412}
\end{align}%
By the second inequality in (\ref{eq:vol_0}),%
\begin{equation}
w(B_{r})=\frac{w(x_{0},r)}{w(x_{0},4r)}w(x_{0},4r)\geq \frac{1}{%
C_{2}4^{\beta _{2}}}w(x_{0},4r).  \label{44}
\end{equation}

Therefore, plugging (\ref{eq:vol_412}) and (\ref{44}) into (\ref{eq:vol_411}%
), we obtain%
\begin{eqnarray*}
\left( \frac{\eta _{1}}{C_{\mu }^{2}}\right) ^{1/p}a &\leq &\left( \fint%
_{B_{4r}}u^{p}d\mu \right) ^{1/p}\leq C_{H}\left( \einf%
_{B_{4r}}u+w(x_{0},4r)\left( T_{\frac{3}{4}B_{R},B_{R}}(u_{-})+||f||_{L^{%
\infty }(B_{R})}\right) \right) \\
&\leq &C_{H}\left( \einf_{B_{4r}}u+C_{2}4^{\beta _{2}}w(x_{0},r)\left( T_{%
\frac{3}{4}B_{R},B_{R}}(u_{-})+||f||_{L^{\infty }(B_{R})}\right) \right) \\
&\leq &C_{2}4^{\beta _{2}}C_{H}\left( \einf_{B_{4r}}u+w(B_{r})\left( T_{%
\frac{3}{4}B_{R},B_{R}}(u_{-})+||f||_{L^{\infty }(B_{R})}\right) \right) ,
\end{eqnarray*}%
which gives that%
\begin{align*}
\einf_{B_{4r}}u& \geq \left( C_{2}4^{\beta _{2}}C_{H}\right) ^{-1}\left( 
\frac{\eta _{1}}{C_{\mu }^{2}}\right) ^{1/p}a-w(B_{r})\left( T_{\frac{3}{4}%
B_{R},B_{R}}(u_{-})+||f||_{L^{\infty }(B_{R})}\right) \\
& =\varepsilon _{1}a-w(B_{r})\left( T_{\frac{3}{4}%
B_{R},B_{R}}(u_{-})+||f||_{L^{\infty }(B_{R})}\right) ,
\end{align*}%
thus showing (\ref{eq:vol_402}) with $\varepsilon _{1}$ given by (\ref{45}).
Hence, condition $(\mathrm{wEH}1)$ holds.

We show the opposite implication $(\mathrm{wEH}1)$ $\Rightarrow $ $(\mathrm{%
wEH})$.

\noindent We will use the Krylov-Safonov covering lemma on the doubling
space as follows, see for example \cite[Lemma 3.8]{ChenKumagaiWang.2016.EH}
or \cite[Lemma 7.2]{KinnunenShanmugakubgam.2001.MM401}. Suppose that
condition $(\mathrm{VD})$ holds. Let $r$ be a number in $(0,\overline{R}/5)$
and $E\subset B_{r}(x_{0})$ a measurable set. For any number $\eta \in (0,1)$%
, we define 
\begin{equation*}
\lbrack E]_{\eta }=\bigcup_{0<\rho <r}\left\{ B_{5\rho }(x)\cap
B_{r}(x_{0}):x\in B_{r}(x_{0})\ \mathrm{and}\ \frac{\mu (E\cap B_{5\rho }(x))%
}{\mu (B_{\rho }(x))}>\eta \right\} .
\end{equation*}%
Then either 
\begin{equation*}
\lbrack E]_{\eta }=B_{r}(x_{0})
\end{equation*}%
or 
\begin{equation*}
\mu ([E]_{\eta })\geq \frac{1}{\eta }\mu (E).
\end{equation*}

Assume that condition $(\mathrm{wEH}1)$ holds. We show $(\mathrm{wEH})$.

To do this, let $\eta $ be any fixed number in $(0,1)$. Let $\sigma \in
(0,1) $ and $\delta _{1}\in (0,1/4)$ be the constants coming from condition $%
(\mathrm{wEH}1)$. Let $B_{R}:=B(x_{0},R)$ be any metric ball with $%
0<R<\sigma $$\overline{R}$ and $r$ any number in $(0,\frac{\delta _{1}}{10}%
R] $. Let $u\in \mathcal{F^{\prime }}\cap L^{\infty }$ be any function that
is non-negative, $f$-superharmonic in $B_{R}$. We define%
\begin{equation*}
A_{t}^{i}:=\left\{ x\in B_{r}(x_{0}):\ u(x)>t\varepsilon ^{i}-\frac{T}{%
1-\varepsilon }\right\}
\end{equation*}%
for any $t>0$ and $i\geq 0$, where constant $\varepsilon \in (0,1)$ will be
determined later and $T$ is given by 
\begin{equation}
T=C_{2}w(x_{0},5r)\left( T_{\frac{3}{4}B_{R},B_{R}}(u_{-})+||f||_{L^{\infty
}(B_{R})}\right)  \label{TT}
\end{equation}%
with constant $C_{2}$ as in (\ref{eq:vol_0}).

Obviously, we have $A_{t}^{i-1}\subset A_{t}^{i}$ for any $i\geq 1$. Let $x$
be any point in $B_{r}(x_{0})$ and $\rho $ be any number in $(0,r)$. If 
\begin{equation}
B_{5\rho }(x)\cap B_{r}(x_{0})\subset \lbrack A_{t}^{i-1}]_{\eta },
\label{uu-2}
\end{equation}%
which is equivalent to the fact that $\mu (A_{t}^{i-1}\cap B_{5\rho
}(x))>\eta \mu (B_{\rho }(x))$ by the definition of $[A_{t}^{i-1}]_{\eta }$,
then 
\begin{equation*}
\mu (A_{t}^{i-1}\cap B_{5\rho }(x))>\eta \mu (B_{\rho }(x))\geq C_{\mu
}^{-3}\eta \mu (B_{5\rho }(x)),
\end{equation*}%
since $\mu (B_{\rho }(x))\geq C_{\mu }^{-3}\mu (B_{5\rho }(x))$ by using
condition $(\mathrm{VD})$. Let $\varepsilon :=\varepsilon _{1}(C_{\mu
}^{-3}\eta )$. Since $B(x,\frac{R}{2})\subset B(x_{0},R)$, the function $u$
is non-negative, $f$-superharmonic in $B(x,\frac{R}{2})$. Noting that $5\rho
<5r\leq 5\frac{\delta _{1}}{10}R=\delta _{1}\frac{R}{2}$ and%
\begin{equation*}
\frac{\mu \left( B_{5\rho }(x)\cap \{u\geq t\varepsilon ^{i-1}-\frac{T}{%
1-\varepsilon }\}\right) }{\mu (B_{5\rho }(x))}\geq \frac{\mu \left(
B_{5\rho }(x)\cap A_{t}^{i-1}\right) }{\mu (B_{5\rho }(x))}\geq C_{\mu
}^{-3}\eta ,
\end{equation*}%
we apply condition $($\textrm{wEH}$1)$ on two concentric balls $B(x,\frac{R}{%
2}),B(x,5\rho )$ for $\eta _{1}=C_{\mu }^{-3}\eta $ and for those $t>0$ such
that 
\begin{equation*}
a:=t\varepsilon ^{i-1}-\frac{T}{1-\varepsilon }>0.
\end{equation*}

It follows that, using the fact that $w(x,5\rho )<w(x,5r)\leq
C_{2}w(x_{0},5r)$ by (\ref{eq:vol_0}), 
\begin{align}
\einf_{B_{20\rho }(x)}u& >\varepsilon \left( t\varepsilon ^{i-1}-\frac{T}{%
1-\varepsilon }\right) -w(x,5\rho )\left( T_{\frac{3}{4}B(x,\frac{R}{2}),B(x,%
\frac{R}{2})}(u_{-})+||f||_{L^{\infty }(B(x,\frac{R}{2}))}\right)  \notag \\
& \geq \varepsilon \left( t\varepsilon ^{i-1}-\frac{T}{1-\varepsilon }%
\right) -C_{2}w(x_{0},5r)\left( T_{\frac{3}{4}B_{R},B_{R}}(u_{-})+||f||_{L^{%
\infty }(B_{R})}\right)  \notag \\
& =\varepsilon \left( t\varepsilon ^{i-1}-\frac{T}{1-\varepsilon }\right)
-T=t\varepsilon ^{i}-\frac{T}{1-\varepsilon },  \label{uu-1}
\end{align}%
where we have used the fact that 
\begin{equation*}
T_{\frac{3}{4}B(x,\frac{R}{2}),B(x,\frac{R}{2})}(u_{-})\leq T_{\frac{3}{4}%
B_{R},B_{R}}(u_{-})
\end{equation*}%
by Proposition \ref{prop:3.1} since $B(x,\frac{R}{2})\subset \frac{3}{4}%
B_{R}=B(x_{0},\frac{3}{4}R)$ for any $x$ in $B_{r}(x_{0})$. Clearly, the
inequality (\ref{uu-1}) also holds for those $t$ when $t\varepsilon ^{i-1}-%
\frac{T}{1-\varepsilon }\leq 0$, and hence, it is true for any $t>0$,
provided that (\ref{uu-2}) is satisfied.

Therefore, for any ball $B_{5\rho }(x)$ satisfying (\ref{uu-2}), we have $%
B_{5\rho }(x)\cap B_{r}(x_{0})\subset A_{t}^{i}$, which implies that 
\begin{equation*}
\lbrack A_{t}^{i-1}]_{\eta }\subset A_{t}^{i}\text{ \ for any }t>0\text{ and
any }i\geq 1.
\end{equation*}%
By the Krylov-Safonov covering lemma with $E=A_{t}^{i-1}$, we must have that
for any $t>0$ and any $i\geq 1$, either $A_{t}^{i-1}=B_{r}(x_{0})$ (thus $%
A_{t}^{i}=B_{r}(x_{0})$) or 
\begin{equation}
\frac{1}{\eta }\mu (A_{t}^{i-1})\leq \mu ([A_{t}^{i-1}]_{\eta })\leq \mu
(A_{t}^{i}).  \label{eq:vol_414}
\end{equation}%
We choose an integer $j\geq 1$ such that 
\begin{equation*}
\eta ^{j}<\frac{\mu (A_{t}^{0})}{\mu (B_{r}(x_{0}))}\leq \eta ^{j-1}.
\end{equation*}%
Suppose that $A_{t}^{j-1}\neq B_{r}(x_{0})$. Using the fact that $%
A_{t}^{i-1}\subset A_{t}^{i}$, we have $A_{t}^{k}\neq B_{r}(x_{0})$ for all $%
0\leq k\leq j-1$. Hence, we obtain from (\ref{eq:vol_414}) that 
\begin{equation*}
\mu (A_{t}^{j-1})\geq \frac{1}{\eta }\mu (A_{t}^{j-2})\geq \cdot \cdot \cdot
\geq \frac{1}{\eta ^{j-1}}\mu (A_{t}^{0})\geq \eta \mu (B_{r}(x_{0})).
\end{equation*}%
Since $\eta \in (0,1)$, this inequality holds trivially when $%
A_{t}^{j-1}=B_{r}(x_{0})$. Therefore, using condition $(\mathrm{wEH}1)$
again, we have 
\begin{align*}
\einf_{B_{4r}(x_{0})}u& >\varepsilon _{1}(\eta )\left( t\varepsilon ^{j-1}-%
\frac{T}{1-\varepsilon }\right) -w(B_{r})\left( T_{\frac{3}{4}%
B_{R},B_{R}}(u_{-})+||f||_{L^{\infty }(B_{R})}\right) \\
& \geq \varepsilon _{1}(\eta )\left( t\varepsilon ^{j-1}-\frac{T}{%
1-\varepsilon }\right) -T\geq \varepsilon _{1}(\eta )t\varepsilon ^{j-1}-%
\frac{\varepsilon _{1}(\eta )+1}{1-\varepsilon }T \\
& \geq \varepsilon _{1}(\eta )t\left( \frac{\mu (A_{t}^{0})}{\mu
(B_{r}(x_{0}))}\right) ^{\frac{1}{\gamma }}-\frac{\varepsilon _{1}(\eta )+1}{%
1-\varepsilon }T,
\end{align*}%
where $\gamma =\log _{\varepsilon }\eta $. From this, it follows that, for
any $t>0$ and any $r\in (0,\frac{\delta _{1}}{10}R]$, 
\begin{equation*}
\frac{\mu (A_{t}^{0})}{\mu (B_{r}(x_{0}))}\leq \frac{c_{3}}{t^{\gamma }}%
\left( \einf_{B_{4r}(x_{0})}u+\frac{T}{1-\varepsilon }\right) ^{\gamma }
\end{equation*}%
for some positive constant $c_{3}$ depending only on $\eta $ (for example, $%
c_{3}=\frac{\varepsilon _{1}(\eta )+1}{\varepsilon _{1}(\eta )}$).

Therefore, for any $0<p<\gamma $ and any $a>0$, 
\begin{align*}
\fint_{B_{r}(x_{0})}u^{p}d\mu & =p\int_{0}^{\infty }t^{p-1}\frac{\mu
(B_{r}(x_{0})\cap \{u>t\})}{\mu (B_{r}(x_{0}))}dt\leq p\int_{0}^{\infty
}t^{p-1}\frac{\mu (A_{t}^{0})}{\mu (B_{r}(x_{0}))}dt \\
& \leq p\left[ \int_{0}^{a}t^{p-1}dt+c_{3}\left( \einf_{B_{4r}(x_{0})}u+%
\frac{T}{1-\varepsilon }\right) ^{\gamma }\int_{a}^{\infty }t^{p-1-\gamma }dt%
\right] \\
& \leq c_{4}(p,\eta ,\varepsilon )\left[ a^{p}+\left( \einf_{B_{4r}(x_{0})}u+%
\frac{T}{1-\varepsilon }\right) ^{\gamma }a^{p-\gamma }\right] .
\end{align*}%
By choosing $a$ such that 
\begin{equation*}
a=\einf_{B_{4r}(x_{0})}u+\frac{T}{1-\varepsilon },
\end{equation*}%
we conclude by using (\ref{TT}), (\ref{eq:vol_0}) that for any $0<r\leq 
\frac{\delta _{1}}{10}R$ 
\begin{align*}
\fint_{B_{r}(x_{0})}u^{p}d\mu & \leq 2c_{4}(p,\eta ,\varepsilon )\left( \einf%
_{B_{4r}(x_{0})}u+\frac{T}{1-\varepsilon }\right) ^{p} \\
& \leq 2c_{4}(p,\eta ,\varepsilon )\left( \einf_{B_{4r}(x_{0})}u+\frac{%
C_{2}^{2}5^{\beta _{2}}w(B_{r})\left( T_{\frac{3}{4}%
B_{R},B_{R}}(u_{-})+||f||_{L^{\infty }(B_{R})}\right) }{1-\varepsilon }%
\right) ^{p} \\
& \leq c_{5}(p,\eta ,\varepsilon )\left( \einf_{B_{r}(x_{0})}u+w(B_{r})%
\left( T_{\frac{3}{4}B_{R},B_{R}}(u_{-})+||f||_{L^{\infty }(B_{R})}\right)
\right) ^{p},
\end{align*}%
thus showing that condition $(\mathrm{wEH})$ holds with $\delta :=\frac{%
\delta _{1}}{10}$. The proof is complete.
\end{proof}

We introduce condition $(\mathrm{wEH}2)$ (cf. \cite[Lemma 7.2]%
{GrigoryanHuLau.2015.JMSJ1485} for the local Dirichlet form).

\begin{definition}[condition (\textrm{wEH}$2$)]
\label{df:wehi3}We say that condition $(\mathrm{wEH}2)$ holds if there exist
three universal constants $\sigma $, $\delta _{2}$ in $(0,1)$ and $C>0$ such
that, for any two concentric balls $B_{R}:=B(x_{0},R)\supset
B(x_{0},r)=:B_{r}$ with $R\in (0,\sigma \overline{R})$, $r\in (0,\delta
_{2}R)$, any number $a>0$, any function $f\in L^{\infty }(B_{R})$, and for
any $u\in \mathcal{F}^{\prime }\cap L^{\infty }$ which is non-negative, $f$%
-superharmonic in $B_{R}$, we have 
\begin{equation}
\einf_{B_{r}}u\geq a\exp \left( -\frac{C}{\omega _{B_{r}}(\{u\geq a\})}%
\right) -w(B_{r})\left( T_{\frac{3}{4}B_{R},B_{R}}(u_{-})+||f||_{L^{\infty
}(B_{R})}\right) .  \label{eq:vol_403}
\end{equation}%
We remark that the constants $\sigma $, $\delta _{2},C$ are all independent
of $B_{R},B_{r},a,f$ and $u$.
\end{definition}

\begin{remark}
\label{R-EH2}Let $a>0$. If 
\begin{equation*}
\omega _{B_{r}}(\{u\geq a\})=0,
\end{equation*}%
then (\ref{eq:vol_403}) is trivially satisfied since $u\geq 0$ in $B_{r}$.
On the other hand, if 
\begin{equation*}
\omega _{B_{r}}(\{u\geq a\})=1,
\end{equation*}%
then (\ref{eq:vol_403}) is also trivially satisfied since $u\geq a$ in $%
B_{r} $. Thus, in order to show (\ref{eq:vol_403}), it suffices to consider
the case $0<\omega _{B_{r}}(\{u\geq a\})<1$ only.
\end{remark}

We have the following.

\begin{proposition}
\label{prop:4.3}Let $(\mathcal{E},\mathcal{F})$ be a Dirichlet form in $%
L^{2} $. Then%
\begin{equation*}
(\mathrm{wEH})\Rightarrow (\mathrm{wEH}2).
\end{equation*}
\end{proposition}

\begin{proof}
Assume that condition $(\mathrm{wEH})$ holds with four constants $p,\delta
,\sigma $ in $(0,1)$ and $C_{H}\geq 1$. Fix a ball $B_{R}:=B(x_{0},R)$ with $%
R\in (0,\sigma \overline{R})$ and fix another concentric ball $%
B_{r}:=B(x_{0},r)$ with $0<r\leq \delta R$. Let $u\in \mathcal{F^{\prime }}%
\cap L^{\infty }$ be non-negative, $f$-superharmonic in $B_{R}$. Then 
\begin{equation}
\left( \frac{1}{\mu (B_{r})}\int_{B_{r}}u^{p}d\mu \right) ^{1/p}\leq
C_{H}\left( \einf_{{B_{r}}}u+w(B_{r})\left( T_{\frac{3}{4}%
B_{R},B_{R}}(u_{-})+||f||_{L^{\infty }(B_{R})}\right) \right) .
\label{eq:vol_406}
\end{equation}%
In order to show condition $(\mathrm{wEH}2)$, we shall prove that (\ref%
{eq:vol_403}) holds with 
\begin{equation}
\delta _{2}=\delta \text{ \ and \ }C:=\ln C_{H}+1/p.  \label{51}
\end{equation}

To see this, let $a$ be any positive number. By Remark \ref{R-EH2}, we may
assume 
\begin{equation*}
0<\omega _{B_{r}}(\{u\geq a\})<1.
\end{equation*}%
Note that, using the elementary inequality $\ln x\geq 1-\frac{1}{x}$ for any 
$0<x\leq 1$, 
\begin{align}
\left( \frac{1}{\mu (B_{r})}\int_{B_{r}}u^{p}d\mu \right) ^{1/p}& \geq
\left( \frac{1}{\mu (B_{r})}\int_{B_{r}\cap \{u\geq a\}}a^{p}d\mu \right)
^{1/p}=\left( a^{p}\omega _{B_{r}}(\{u\geq a\})\right) ^{1/p}  \notag \\
& =a\exp \left( \frac{1}{p}\ln \omega _{B_{r}}(\{u\geq a\})\right) \geq
a\exp \left( \frac{1}{p}\left( 1-\frac{1}{\omega _{B_{r}}(\{u\geq a\})}%
\right) \right)  \notag \\
& \geq a\exp \left( -\frac{1/p}{\omega _{B_{r}}(\{u\geq a\})}\right) =a\exp
\left( -\frac{C-\ln C_{H}}{\omega _{B_{r}}(\{u\geq a\})}\right)  \notag \\
& =a\exp \left( \frac{\ln C_{H}}{\omega _{B_{r}}(\{u\geq a\})}\right) \cdot
\exp \left( -\frac{C}{\omega _{B_{r}}(\{u\geq a\})}\right)  \notag \\
& \geq aC_{H}\exp \left( -\frac{C}{\omega _{B_{r}}(\{u\geq a\})}\right) 
\text{ \ (since }\omega _{B_{r}}(\{u\geq a\})\leq 1\text{)}.
\label{eq:vol_407}
\end{align}%
Plugging (\ref{eq:vol_407}) into (\ref{eq:vol_406}) and then dividing by $%
C_{H}$ on the both sides, we conclude that 
\begin{equation*}
a\exp \left( -\frac{C}{\omega _{B_{r}}(\{u\geq a\})}\right) \leq \einf_{{%
B_{r}}}u+w(B_{r})\left( T_{\frac{3}{4}B_{R},B_{R}}(u_{-})+||f||_{L^{\infty
}(B_{R})}\right) ,
\end{equation*}%
thus showing that (\ref{eq:vol_403}) holds with constants $\delta _{2},C$
chosen as in (\ref{51}). The proof is complete.
\end{proof}

We introduce condition $(\mathrm{wEH}3)$.

\begin{definition}[condition (\textrm{wEH}$3$)]
\label{df:wehi4}We say that condition $(\mathrm{wEH}3)$ holds if there exist
two universal constants $\sigma ,\delta _{3}$ in $(0,1)$ such that, for any
two concentric balls $B_{R}:=B(x_{0},R)\supset B(x_{0},r)=:$ $B_{r}$ with $%
R\in (0,\sigma \overline{R})$, $r\in (0,\delta _{3}R]$, any number $\eta
_{3}\in (0,1]$, any function $f\in L^{\infty }(B_{R})$, and for any $u\in 
\mathcal{F}^{\prime }\cap L^{\infty }$ which is non-negative, $f$%
-superharmonic in $B_{R}$, if for some $a>0$, 
\begin{equation*}
\omega _{B_{r}}(\{u\geq a\})=\frac{\mu (B_{r}\cap \left\{ u\geq a\right\} )}{%
\mu (B_{r})}\geq \eta _{3}
\end{equation*}%
and 
\begin{equation*}
w(B_{r})\left( T_{\frac{3}{4}B_{R},B_{R}}(u_{-})+||f||_{L^{\infty
}(B_{R})}\right) \leq F(\eta _{3})a
\end{equation*}%
for a map $F:(0,1]\longmapsto (0,1]$, then we have 
\begin{equation}
\einf_{B_{r}}u\geq F(\eta _{3})a.  \label{eq:vol_404}
\end{equation}
\end{definition}

We show condition $(\mathrm{wEH}2)$ alone implies condition $(\mathrm{wEH}3)$
for any Dirichlet form in $L^{2}$.

\begin{proposition}
\label{P1}Let $(\mathcal{E},\mathcal{F})$ be a Dirichlet form in $L^{2}$,
then 
\begin{equation*}
(\mathrm{wEH}2)\Rightarrow (\mathrm{wEH}3).
\end{equation*}
\end{proposition}

\begin{proof}
Assume that condition $(\mathrm{wEH}2)$ holds with constants $\sigma ,\delta
_{2},C$. We shall show that condition $(\mathrm{wEH}3)$ holds with the same $%
\sigma $ and with $\delta _{3},F$ being given by 
\begin{equation}
\delta _{3}=\delta _{2}\text{ \ and \ }F(\eta _{3})=\frac{1}{2}\exp \left( -%
\frac{C}{\eta _{3}}\right) .  \label{52}
\end{equation}

To see this, fix a ball $B_{R}:=B(x_{0},R)$ with $R\in (0,\sigma \overline{R}%
)$ and fix another concentric ball $B_{r}:=B(x_{0},r)$ with $0<r\leq \delta
_{2}R$. Let $\eta _{3}\in (0,1]$ and $r\in (0,\delta _{2}R]$ be any two
numbers. Let $u\in \mathcal{F^{\prime }}\cap L^{\infty }$ be any function
that is non-negative, $f$-superharmonic in $B_{R}$. If for some $a>0$, 
\begin{equation*}
\omega _{B_{r}}(\{u\geq a\})\geq \eta _{3}
\end{equation*}%
and if 
\begin{equation*}
w(B_{r})\left( T_{\frac{3}{4}B_{R},B_{R}}(u_{-})+||f||_{L^{\infty
}(B_{R})}\right) \leq F(\eta _{3})a=\frac{1}{2}\exp \left( -\frac{C}{\eta
_{3}}\right) a,
\end{equation*}%
then by condition $(\mathrm{wEH}2)$, 
\begin{align*}
\einf_{B_{r}}u& \geq a\exp \left( -\frac{C}{\omega _{B_{r}}(\{u\geq a\})}%
\right) -w(B_{r})\left( T_{\frac{3}{4}B_{R},B_{R}}(u_{-})+||f||_{L^{\infty
}(B_{R})}\right) \\
& \geq a\exp \left( -\frac{C}{\eta _{3}}\right) -\frac{1}{2}\exp \left( -%
\frac{C}{\eta _{3}}\right) a=\frac{1}{2}\exp \left( -\frac{C}{\eta _{3}}%
\right) a=F(\eta _{3})a.
\end{align*}%
This proves that (\ref{eq:vol_404}) is true, and so condition $(\mathrm{wEH}%
3)$ holds. The proof is complete.
\end{proof}

The following shows that condition $(\mathrm{wEH}3)$ implies condition $(%
\mathrm{wEH}1)$.

\begin{proposition}
\label{prop:4.4} Let $(\mathcal{E},\mathcal{F})$ be a Dirichlet form in $%
L^{2}$. If condition $(\mathrm{VD})$ holds, then 
\begin{equation*}
(\mathrm{wEH}3)\Rightarrow (\mathrm{wEH}1).
\end{equation*}
\end{proposition}

\begin{proof}
Assume that condition $(\mathrm{wEH}3)$ holds with constants $\sigma ,\delta
_{3}$ in $(0,1)$ and a map $F:(0,1]\rightarrow (0,1]$. We show that
condition $(\mathrm{wEH}1)$ holds with the same $\sigma $ and with constants%
\begin{equation}
\delta _{1}=\frac{\delta _{3}}{8}\text{ \ and \ }\varepsilon
_{1}:=\varepsilon _{1}(\eta _{1})=\frac{F(\eta _{1}/C_{\mu }^{3})}{%
C_{2}8^{\beta _{2}}},  \label{53}
\end{equation}%
so that $\delta _{1}\in (0,\frac{1}{4})$ and $\varepsilon _{1}=\varepsilon
_{1}(\eta _{1})\in (0,1)$, where constants $C_{\mu }$ is the same as in (\ref%
{eq:vol_1}) and $C_{2},\beta _{2}$ same as in (\ref{eq:vol_0}).

To see this, fix two concentric balls $B_{R}:=B(x_{0},R)\supset
B(x_{0},r)=:B_{r}$ with $R\in (0,\sigma \overline{R})$, $r\in (0,\delta
_{1}R]$. Let $\eta _{1}\in (0,1]$ be any fixed number. Let $u\in \mathcal{%
F^{\prime }}\cap L^{\infty }$ be any function that is non-negative, $f$%
-superharmonic in $B_{R}$. Suppose that for some $a>0$, 
\begin{equation}
\omega _{B_{r}}(\{u\geq a\})\geq \eta _{1}.  \label{eq:vol_408}
\end{equation}%
We need to show that (\ref{eq:vol_402}) is satisfied.

Indeed, since $r\leq \delta _{1}R=\delta _{3}R/8$ so that%
\begin{equation}
B_{\frac{8r}{\delta _{3}}}\subseteq B_{R},  \label{54-1}
\end{equation}%
the function $u$ is non-negative and $f$-superharmonic in $B_{\frac{8r}{%
\delta _{3}}}$. By (\ref{eq:vol_408}) and condition $(\mathrm{VD})$ 
\begin{align}
\omega _{\delta _{3}B_{\frac{8r}{\delta _{3}}}}(u\geq a)& =\omega
_{B_{8r}}(\{u\geq a\})=\frac{\mu (B_{8r}\cap \{u\geq a\})}{\mu (B_{8r})}\geq 
\frac{\mu (B_{r}\cap \{u\geq a\})}{\mu (B_{8r})}  \notag \\
& =\frac{\omega _{B_{r}}(\{u\geq a\})\mu (B_{r})}{\mu (B_{8r})}\geq \frac{%
\eta _{1}\mu (B_{r})}{\mu (B_{8r})}\geq \frac{\eta _{1}}{C_{\mu }^{3}}:=\eta
_{3}.  \label{54}
\end{align}%
We distinguish two cases.

\emph{Case }$1$ when 
\begin{equation}
w(B_{r})\left( T_{\frac{3}{4}B_{R},B_{R}}(u_{-})+||f||_{L^{\infty
}(B_{R})}\right) \leq \varepsilon _{1}a=\frac{F(\eta _{1}/C_{\mu }^{3})}{%
C_{2}8^{\beta _{2}}}a.  \label{55}
\end{equation}%
In this case, we have 
\begin{align*}
w(\delta _{3}B_{\frac{8r}{\delta _{3}}})\left( T_{\frac{3}{4}B_{\frac{8r}{%
\delta _{3}}},B_{\frac{8r}{\delta _{3}}}}(u_{-})+||f||_{L^{\infty }(B_{\frac{%
8r}{\delta _{3}}})}\right) & =w(B_{8r})\left( T_{B_{\frac{6r}{\delta _{3}}%
},B_{\frac{8r}{\delta _{3}}}}(u_{-})+||f||_{L^{\infty }(B_{\frac{8r}{\delta
_{3}}})}\right) \\
& \leq w(B_{8r})\left( T_{\frac{3}{4}B_{R},B_{R}}(u_{-})+||f||_{L^{\infty
}(B_{R})}\right) \\
& \leq C_{2}8^{\beta _{2}}w(B_{r})\left( T_{\frac{3}{4}%
B_{R},B_{R}}(u_{-})+||f||_{L^{\infty }(B_{R})}\right) \\
& \leq F(\eta _{1}/C_{\mu }^{3})a=F(\eta _{3})a,
\end{align*}%
where in the first inequality we have used the fact%
\begin{equation*}
T_{B_{\frac{6r}{\delta _{3}}},B_{\frac{8r}{\delta _{3}}}}(u_{-})\leq T_{%
\frac{3}{4}B_{R},B_{R}}(u_{-})
\end{equation*}%
since $B_{\frac{6r}{\delta _{3}}}\subseteq \frac{3}{4}B_{R}$ and $u$ is
non-negative in $B_{R}$, whilst in the second inequality we have used the
fact that%
\begin{equation*}
\frac{w(B_{8r})}{w(B_{r})}=\frac{w(x_{0},8r)}{w(x_{0},r)}\leq C_{2}\left( 
\frac{8r}{r}\right) ^{\beta _{2}}=C_{2}8^{\beta _{2}}
\end{equation*}%
by virtue of (\ref{eq:vol_0}). Therefore, applying $(\mathrm{wEH}3) $ with $%
B_{R}$ being replaced by $B_{\frac{8r}{\delta _{3}}}$, we obtain 
\begin{equation}
\einf_{B_{4r}}u\geq F(\eta _{3})a.  \label{eq:vol_409}
\end{equation}%
Noting that 
\begin{equation}
\varepsilon _{1}=\frac{F(\eta _{1}/C_{\mu }^{3})}{C_{2}8^{\beta _{2}}}%
<F(\eta _{1}/C_{\mu }^{3})=F(\eta _{3}),  \label{eq:vol_415}
\end{equation}%
we see that 
\begin{eqnarray}
F(\eta _{3})a &\geq &F(\eta _{3})a-w(B_{r})\left( T_{\frac{3}{4}%
B_{R},B_{R}}(u_{-})+||f||_{L^{\infty }(B_{R})}\right)  \notag \\
&>&\varepsilon _{1}a-w(B_{r})\left( T_{\frac{3}{4}%
B_{R},B_{R}}(u_{-})+||f||_{L^{\infty }(B_{R})}\right) .  \label{eq:vol_410}
\end{eqnarray}%
Plugging (\ref{eq:vol_410}) into (\ref{eq:vol_409}), it follows that%
\begin{equation*}
\einf_{B_{4r}}u\geq F(\eta _{3})a>\varepsilon _{1}a-w(B_{r})\left( T_{\frac{3%
}{4}B_{R},B_{R}}(u_{-})+||f||_{L^{\infty }(B_{R})}\right) ,
\end{equation*}%
thus showing that (\ref{eq:vol_402}) is true in this case.

\emph{Case }$2$ when 
\begin{equation*}
w(B_{r})\left( T_{\frac{3}{4}B_{R},B_{R}}(u_{-})+||f||_{L^{\infty
}(B_{R})}\right) >\varepsilon _{1}a.
\end{equation*}%
In this case, we immediately see that 
\begin{equation*}
\einf_{B_{4r}}u\geq 0>\varepsilon _{1}a-w(B_{r})\left( T_{\frac{3}{4}%
B_{R},B_{R}}(u_{-})+||f||_{L^{\infty }(B_{R})}\right) ,
\end{equation*}%
thus showing that (\ref{eq:vol_402}) is true again.

Therefore, we always have that (\ref{eq:vol_402}) holds, no matter the \emph{%
Case }$1$ happens or not. This proves condition $(\mathrm{wEH}1)$. The proof
is complete.
\end{proof}

The following another version of the weak elliptic Harnack inequality was
essentially introduced in \cite[Lemma 4.5]{GrigoryanHuHu.2018.AM433} when $%
f\equiv 0$, and the jump kernel $J$ exists and satisfies the upper and lower
bounds.

\begin{definition}[condition (\textrm{wEH}$4$)]
We say that condition $(\mathrm{wEH}4)$ holds if there exist three universal
constants $\sigma ,\varepsilon _{4},\delta _{4}$ in $(0,1)$ such that, for
any ball $B_{R}:=B(x_{0},R)$ with $R\in (0,\sigma \overline{R})$, any
function $f\in L^{\infty }(B_{R})$, and for any $u\in \mathcal{F^{\prime }}%
\cap L^{\infty }$ which is non-negative and $f$-superharmonic in $B_{R}$, if
for some $a>0$, 
\begin{equation*}
\omega _{\delta _{4}B_{R}}(\{u\geq a\})=\frac{\mu (\delta _{4}B_{R}\cap
\left\{ u\geq a\right\} )}{\mu (\delta _{4}B_{R})}\geq \frac{1}{2}
\end{equation*}%
and 
\begin{equation*}
w(\delta _{4}B_{R})\left( T_{\frac{3}{4}B_{R},B_{R}}(u_{-})+||f||_{L^{\infty
}(B_{R})}\right) \leq \varepsilon _{4}a,
\end{equation*}%
then 
\begin{equation}
\einf_{\frac{1}{2}(\delta _{4}B_{R})}u\geq \varepsilon _{4}a.
\label{eq:vol_405}
\end{equation}
\end{definition}

\begin{proposition}
\label{prop:4.1} Let $(\mathcal{E},\mathcal{F})$ be a Dirichlet form in $%
L^{2}$, then 
\begin{equation*}
(\mathrm{wEH}3)\Rightarrow (\mathrm{wEH}4).
\end{equation*}
\end{proposition}

\begin{proof}
In fact, condition $(\mathrm{wEH}4)$ is a special case of condition $(%
\mathrm{wEH}3)$ with $\eta _{3}=\frac{1}{2}$, $\varepsilon _{4}=F(1/2)$ and $%
\delta _{4}=\delta _{3}$. The proof is complete.
\end{proof}

We are now in a position to prove Theorem \ref{thm:M10}.

\begin{proof}[Proof of Theorem \protect\ref{thm:M10}]
We have the following conclusions: 
\begin{eqnarray*}
(\mathrm{wEH}) &\Leftrightarrow &(\mathrm{wEH1})\text{ \ (Proposition \ref%
{prop:4.2}),} \\
(\mathrm{wEH}) &\Rightarrow &(\mathrm{wEH2})\text{ \ (Proposition \ref%
{prop:4.3}),} \\
&\Rightarrow &(\mathrm{wEH3})\text{ \ (Proposition \ref{P1}),} \\
&\Rightarrow &(\mathrm{wEH1})\text{ \ (Proposition \ref{prop:4.4}),}
\end{eqnarray*}%
thus showing that the equivalences in (\ref{56}) are all true.

Finally, the implication $(\mathrm{wEH}3)\Rightarrow (\mathrm{wEH}4)$ in (%
\ref{57}) follows immediately by Proposition \ref{prop:4.1}. The proof of
Theorem \ref{thm:M10} is complete.
\end{proof}

\section{Consequences of weak Harnack inequality}

In this section, we look at two consequences of the weak Harnack inequality.
One is that we obtain the H\"{o}lder continuity of any harmonic function if
conditions $(\mathrm{wEH})$ and $(\mathrm{TJ})$ hold for any regular
Dirichlet form without killing part, see Lemma \ref{T2} below. The H\"{o}%
lder continuity of harmonic functions was investigated in various settings,
see for example \cite[Theorem 5.3]{Sivestre.2006.IUMJ1155} for a certain
class of integro-differential equations in $%
%TCIMACRO{\U{211d} }%
%BeginExpansion
\mathbb{R}
%EndExpansion
^{n}$ (see also \cite[Theorem 1.7]{Dyda.2020.AP317} in $%
%TCIMACRO{\U{211d} }%
%BeginExpansion
\mathbb{R}
%EndExpansion
^{n}$ under a weaker assumption), and \cite[Theorem 2.1]%
{ChenKumagaiWang.2016.EH} for a pure-jump Dirichlet form. Here we have
extended this conclusion to a more general situation where the jump kernel
does not necessarily exist. Although the proof is standard, we sketch the
proof for completeness of this paper.

The other consequence of the weak Harnack inequality is that we can obtain a
Lemma of growth for any globally non-negative, superharmonic function (Lemma %
\ref{L3} below), which leads to a lower bound of the mean exit time on a
ball (Lemma \ref{L80} below). The lower bound of the mean exit time plays an
important role in obtaining the heat kernel estimate.

Recall that for an open subset $\Omega $ of $M$, a function $u\in \mathcal{F}
$ is \emph{harmonic} in $\Omega $ if for any non-negative $\varphi \in 
\mathcal{F}(\Omega )$, 
\begin{equation*}
\mathcal{E}(u,\varphi )=0.
\end{equation*}%
For any ball $B\subseteq M$ and any function $u\in L^{\infty }(B,\mu )$, we
define 
\begin{equation*}
\eosc_{B}u:=\esup_{B}u-\einf_{B}u.
\end{equation*}

\begin{lemma}
\label{lemma:EHR}Let $(\mathcal{E},\mathcal{F})$ be a regular Dirichlet form
in $L^{2}(M,\mu )$ without killing part. If conditions $(\mathrm{wEH})$ and $%
(\mathrm{TJ})$ hold, then there exist two constants $\beta \in (0,1)$ and $%
C>0$ such that, for any $x_{0}\in M$, $0<r<\sigma \overline{R}$ and any
harmonic function $u$ in $B(x_{0},r)$, 
\begin{equation}
\eosc_{B(x_{0},\rho )}u\leq C||u||_{L^{\infty }}\left( \frac{\rho }{r}%
\right) ^{\beta },\ \ 0<\rho \leq r.  \label{600}
\end{equation}%
We remark that constants $C,\beta $ are independent of $\overline{R}%
,u,x_{0},r,\rho $.
\end{lemma}

\begin{proof}
Fix a ball $B\left( x_{0},r\right) $ for $0<r<\sigma \overline{R}$. Set%
\begin{equation*}
B_{\rho }:=B\left( x_{0},\rho \right) \text{ \ for any }\rho >0.
\end{equation*}%
Let $u$ be a harmonic function in $B_{r}$. Without loss of generality, we
assume that $||u||_{L^{\infty }(M)}<\infty $. Let%
\begin{equation*}
M_{0}:=||u||_{L^{\infty }}\text{, \ \ }m_{0}:=\einf_{M}u\text{, \ }%
K:=M_{0}-m_{0}
\end{equation*}%
so that $0\leq K\leq 2\Vert u\Vert _{L^{\infty }}$.

We will construct two sequences $\{m_{n}\}_{n\geq 0}$, $\{M_{n}\}_{n\geq 0}$
of positive numbers such that for each $n$,%
\begin{eqnarray}
m_{n-1} &\leq &m_{n}\leq M_{n}\leq M_{n-1}\text{ \ \ and \ \ }%
M_{n}-m_{n}=K\theta ^{-n\beta },  \notag \\
m_{n} &\leq &u(x)\leq M_{n}\quad \text{ for any }x\in B_{r\theta ^{-n}},
\label{601}
\end{eqnarray}%
where $\theta ,\beta $ are two constants to be determined so that 
\begin{equation}
\theta \geq \delta ^{-1}\text{, \ \ }\beta \in (0,1)\text{, \ \ and \ }\frac{%
2-\lambda }{2}\theta ^{\beta }\leq 1,  \label{602}
\end{equation}%
where $\lambda :=(2^{1+1/p}C_{H})^{-1}\in (0,1)$ and $p,\delta \in (0,1)$
and $C_{H}\geq 1$ come from condition $(\mathrm{wEH})$. Once this is true,
then we are done by noting that (\ref{600}) follows, since for any $0<\rho
<r $, there is some integer $j\geq 0$ such that 
\begin{equation*}
\theta ^{-j-1}\leq \frac{\rho }{r}<\theta ^{-j},
\end{equation*}%
from which, we see by (\ref{601}) that 
\begin{equation*}
\eosc_{B_{\rho }}u\leq \eosc_{B_{r\theta ^{-j}}}u\leq M_{j}-m_{j}=K\theta
^{-j\beta }\leq 2\theta ^{\beta }\Vert u\Vert _{L^{\infty }}\left( \frac{%
\rho }{r}\right) ^{\beta }.
\end{equation*}

We will show (\ref{601}) inductively. Indeed, assume that there exists an
integer $k\geq 1$ such that (\ref{601}) holds for any $n\leq k-1$. We need
to construct $m_{k},M_{k}$ such that (\ref{601}) still holds for $n=k$ and
for $\theta ,\beta $ satisfying (\ref{602}).

To do this, set for any $x\in M$ 
\begin{equation}
v(x)=\left( u(x)-\frac{M_{k-1}+m_{k-1}}{2}\right) \frac{2\theta ^{(k-1)\beta
}}{K}.  \label{vv1}
\end{equation}%
Clearly, we have by (\ref{601}) for $n=k-1$ that%
\begin{equation}
|v(x)|\leq \frac{M_{k-1}-m_{k-1}}{2}\frac{2\theta ^{(k-1)\beta }}{K}=\frac{%
K\theta ^{-(k-1)\beta }}{2}\frac{2\theta ^{(k-1)\beta }}{K}=1  \label{vv}
\end{equation}%
for almost all $x\in B_{r\theta ^{-(k-1)}}$.

Note that for any point $y\in B(x_{0},r\theta ^{-(k-1)})^{c}$, there is some
integer $j\geq 1$ such that 
\begin{equation*}
r\theta ^{-k+j}\leq d\left( y,x_{0}\right) <r\theta ^{-k+j+1}.
\end{equation*}%
For simplicity, set $M_{-n}=M_{0}$ and $m_{-n}=m_{0}$ for any $n\geq 1$. By (%
\ref{601}), for any $y\in B(x_{0},r\theta ^{-(k-j-1)})\setminus
B(x_{0},r\theta ^{-(k-j)})$ ($j\geq 1$), 
\begin{align*}
\frac{K}{2\theta ^{(k-1)\beta }}v(y)& =u(y)-\frac{M_{k-1}+m_{k-1}}{2}\leq
M_{k-j-1}-\frac{M_{k-1}+m_{k-1}}{2} \\
& =M_{k-j-1}-m_{k-j-1}+m_{k-j-1}-\frac{M_{k-1}+m_{k-1}}{2} \\
& \leq M_{k-j-1}-m_{k-j-1}-\frac{M_{k-1}-m_{k-1}}{2} \\
& \leq K\theta ^{-(k-j-1)\beta }-\frac{K}{2}\theta ^{-(k-1)\beta },
\end{align*}%
from which, it follows that 
\begin{equation}
v(y)\leq 2\theta ^{j\beta }-1\leq 2\left( \frac{d\left( y,x_{0}\right) }{%
r\theta ^{-k}}\right) ^{\beta }-1\text{ \ for any }y\in B(x_{0},r\theta
^{-(k-1)})^{c}.  \label{603}
\end{equation}%
On the other hand, we similarly have that, for any $y\in B(x_{0},r\theta
^{-(k-j-1)})\setminus B(x_{0},r\theta ^{-(k-j)})$ ($j\geq 1$), 
\begin{align*}
\frac{K}{2\theta ^{(k-1)\beta }}v(y)& =u(y)-\frac{M_{k-1}+m_{k-1}}{2}\geq
m_{k-j-1}-\frac{M_{k-1}+m_{k-1}}{2} \\
& =m_{k-j-1}-M_{k-j-1}+M_{k-j-1}-\frac{M_{k-1}+m_{k-1}}{2} \\
& \geq -\left( M_{k-j-1}-m_{k-j-1}\right) +\frac{M_{k-1}-m_{k-1}}{2} \\
& \geq -K\theta ^{-(k-j-1)\beta }+\frac{K}{2}\theta ^{-(k-1)\beta },
\end{align*}%
which gives that 
\begin{equation}
v(y)\geq 1-2\theta ^{j\beta }\geq 1-2\left( \frac{d\left( y,x_{0}\right) }{%
r\theta ^{-k}}\right) ^{\beta }\text{ \ for any }y\in B(x_{0},r\theta
^{-(k-1)})^{c}.  \label{603-1}
\end{equation}

We distinguish two cases: either%
\begin{equation}
\mu (B_{r\theta ^{-k}}\cap \left\{ v\leq 0\right\} )\geq \mu (B_{r\theta
^{-k}})/2,  \label{603-2}
\end{equation}%
or%
\begin{equation}
\mu (B_{r\theta ^{-k}}\cap \left\{ v>0\right\} )\geq \mu (B_{r\theta
^{-k}})/2.  \label{603-3}
\end{equation}%
If (\ref{603-2}) holds, we will show that for almost every $z\in B_{r\theta
^{-k}}$ 
\begin{equation}
v(z)\leq 1-\lambda .  \label{603-4}
\end{equation}%
Temporally assume that (\ref{603-4}) holds true. Then by (\ref{vv1}), we see
that for any point $z\in B_{r\theta ^{-k}}$, 
\begin{align*}
u(z)& =\frac{K}{2\theta ^{(k-1)\beta }}v(z)+\frac{M_{k-1}+m_{k-1}}{2}\leq 
\frac{K(1-\lambda )}{2\theta ^{(k-1)\beta }}+\frac{M_{k-1}+m_{k-1}}{2} \\
& =\frac{K(1-\lambda )}{2}\theta ^{-(k-1)\beta }+\frac{M_{k-1}-m_{k-1}}{2}%
+m_{k-1} \\
& =\frac{K(1-\lambda )}{2}\theta ^{-(k-1)\beta }+\frac{K}{2}\theta
^{-(k-1)\beta }+m_{k-1}\text{ \ (using (\ref{601}))} \\
& =\frac{(2-\lambda )\theta ^{\beta }}{2}K\theta ^{-k\beta }+m_{k-1}\leq
K\theta ^{-k\beta }+m_{k-1},
\end{align*}%
where in the last inequality we have used the fact that $\frac{2-\lambda }{2}%
\theta ^{\beta }\leq 1$ in (\ref{602}). Therefore, setting%
\begin{equation*}
m_{k}=m_{k-1}\text{ \ \ and \ \ }M_{k}=m_{k}+K\theta ^{-k\beta }\leq M_{k-1},
\end{equation*}%
we obtain that $m_{k}\leq u(z)\leq M_{k}$ for a.e $z\in B_{r\theta ^{-k}}$,
thus showing that (\ref{601}) holds when $n=k$, which finishes the induction
step from $n\leq k-1$ to $n=k$ \ in the case when (\ref{603-2}) holds.

We turn to show (\ref{603-4}). Indeed, consider $h:=1-v$. Clearly, the
function $h$ is harmonic in $B_{r\theta ^{-(k-1)}}$ and also is non-negative
in $B_{r\theta ^{-(k-1)}}$ by using (\ref{vv}). Applying $(\mathrm{wEH})$ to
the function $h$ in $B_{r\theta ^{-(k-1)}}$ and $f=0$, we find that 
\begin{equation}
\left( \fint_{B_{r\theta ^{-k}}}h^{p}du\right) ^{1/p}\leq C_{H}\left( \einf%
_{B_{r\theta ^{-k}}}h+w\left( x_{0},r\theta ^{-k}\right) T_{\frac{3}{4}%
B_{r\theta ^{-(k-1)}},B_{r\theta ^{-(k-1)}}}(h_{-})\right) ,  \label{604}
\end{equation}%
where we have used the fact that $\theta ^{-1}\leq \delta $ so that $r\theta
^{-k}\leq \delta \cdot r\theta ^{-(k-1)}$. Note that by (\ref{603-2}), 
\begin{eqnarray}
\left( \fint_{B_{r\theta ^{-k}}}h^{p}du\right) ^{1/p} &\geq &\left( \frac{1}{%
\mu \left( B_{r\theta ^{-k}}\right) }\int_{B_{r\theta ^{-k}}\cap \left\{
v\leq 0\right\} }(1-v)^{p}du\right) ^{1/p}  \notag \\
&\geq &\left( \frac{\mu (B_{r\theta ^{-k}}\cap \left\{ v\leq 0\right\} )}{%
\mu \left( B_{r\theta ^{-k}}\right) }\right) ^{1/p}\geq 2^{-1/p}.
\label{605}
\end{eqnarray}%
Also note that by (\ref{603}), 
\begin{equation*}
h_{-}(y)=(1-v(y))_{-}=(v(y)-1)_{+}\leq 2\left[ \left( \frac{d\left(
y,x_{0}\right) }{r\theta ^{-k}}\right) ^{\beta }-1\right]
\end{equation*}%
for any $y\in B(x_{0},r\theta ^{-(k-1)})^{c}=B_{r\theta ^{-(k-1)}}^{c}$.
From this, we have by condition $(\mathrm{TJ})$ 
\begin{align}
T_{\frac{3}{4}B_{r\theta ^{-(k-1)}},B_{r\theta ^{-(k-1)}}}(h_{-})&
=\sup_{x\in \frac{3}{4}B_{r\theta ^{-(k-1)}}}\int_{B_{r\theta
^{-(k-1)}}^{c}}h_{-}(y)J(x,dy)  \notag \\
& \leq 2\sup_{x\in \frac{3}{4}B_{r\theta ^{-(k-1)}}}\int_{B_{r\theta
^{-(k-1)}}^{c}}\left[ \left( \frac{d\left( y,x_{0}\right) }{r\theta ^{-k}}%
\right) ^{\beta }-1\right] J(x,dy)  \notag \\
& =2\sup_{x\in \frac{3}{4}B_{r\theta ^{-(k-1)}}}\sum_{j=1}^{\infty
}\int_{B_{r\theta ^{-k+j+1}}\backslash B_{r\theta ^{-k+j}}}\left[ \left( 
\frac{d\left( y,x_{0}\right) }{r\theta ^{-k}}\right) ^{\beta }-1\right]
J(x,dy)  \notag \\
& \leq 2\sum_{j=1}^{\infty }\sup_{x\in \frac{3}{4}B_{r\theta
^{-(k-1)}}}\int_{B_{r\theta ^{-k+j+1}}\backslash B_{r\theta ^{-k+j}}}\left(
\theta ^{(j+1)\beta }-1\right) J(x,dy)  \notag \\
& \leq 2\sum_{j=1}^{\infty }\left( \theta ^{(j+1)\beta }-1\right) \sup_{x\in 
\frac{3}{4}B_{r\theta ^{-(k-1)}}}\int_{B_{r\theta ^{-k+j}}^{c}}J(x,dy) 
\notag \\
& \leq 2\sum_{j=1}^{\infty }\left( \theta ^{(j+1)\beta }-1\right) \sup_{x\in 
\frac{3}{4}B_{r\theta ^{-(k-1)}}}\int_{B(x,r\theta ^{-k+j}/4)^{c}}J(x,dy) 
\notag \\
& \leq 2C\sum_{j=1}^{\infty }\left( \theta ^{(j+1)\beta }-1\right)
\sup_{x\in \frac{3}{4}B_{r\theta ^{-(k-1)}}}\frac{1}{w(x,r\theta ^{-k+j}/4)},
\label{606}
\end{align}%
where $C>0$ is the same constant as in (\ref{eq:vol_16}). Since 
\begin{equation*}
\frac{w(x,r\theta ^{-k+j}/4)}{w(x_{0},r\theta ^{-k})}\geq \frac{w(x,r\theta
^{-k+j})}{C_{2}4^{\beta _{2}}w(x_{0},r\theta ^{-k})}\geq \frac{C_{1}\theta
^{j\beta _{1}}}{C_{2}4^{\beta _{2}}}
\end{equation*}%
for any $x\in \frac{3}{4}B_{r\theta ^{-(k-1)}}$ by using (\ref{eq:vol_0}),
it follows from (\ref{606}) that 
\begin{equation}
w\left( x_{0},r\theta ^{-k}\right) T_{\frac{3}{4}B_{r\theta
^{-(k-1)}},B_{r\theta ^{-(k-1)}}}(h_{-})\leq \frac{2CC_{2}4^{\beta _{2}}}{%
C_{1}}\sum_{j=1}^{\infty }\frac{\theta ^{(j+1)\beta }-1}{\theta ^{j\beta
_{1}}}.  \label{607}
\end{equation}

Therefore, substituting (\ref{605}), (\ref{607}) into (\ref{604}), we obtain 
\begin{align}
\einf_{B_{r\theta ^{-k}}}h& \geq \left( C_{H}2^{1/p}\right) ^{-1}-w\left(
x_{0},r\theta ^{-k}\right) T_{\frac{3}{4}B_{r\theta ^{-(k-1)}},B_{r\theta
^{-(k-1)}}}(h_{-})  \notag \\
& \geq \left( C_{H}2^{1/p}\right) ^{-1}-\frac{2CC_{2}4^{\beta _{2}}}{C_{1}}%
\sum_{j=1}^{\infty }\frac{\theta ^{(j+1)\beta }-1}{\theta ^{j\beta _{1}}}.
\label{51-1}
\end{align}%
Since $\theta ^{-1}\leq \delta $, we see that for any $\beta \in \left(
0,\beta _{1}/2\right) $ 
\begin{equation*}
\sum_{j=l+1}^{\infty }\theta ^{-j\beta _{1}}\left( \theta ^{(j+1)\beta
}-1\right) \leq \sum_{j=l+1}^{\infty }\theta ^{-j\beta _{1}}\theta
^{(j+1)\beta _{1}/2}=\frac{\theta ^{-l\beta _{1}/2}}{1-\theta ^{-\beta
_{1}/2}}\leq \frac{\delta ^{l\beta _{1}/2}}{1-\delta ^{\beta _{1}/2}}\leq
\left( \frac{8CC_{2}4^{\beta _{2}}}{C_{1}}C_{H}2^{1/p}\right) ^{-1},
\end{equation*}%
provided that the number $l$ is sufficiently large, which depends only on $%
\delta $ but is independent of $\beta ,\theta $. For such a number $l$, we
now choose $\beta \in \left( 0,\beta _{1}/2\right) $ to be so small that 
\begin{eqnarray*}
\sum_{j=1}^{l}\theta ^{-j\beta _{1}}\left( \theta ^{(j+1)\beta }-1\right) 
&\leq &\theta ^{-\beta _{1}}\sum_{j=1}^{l}\left( \theta ^{(j+1)\beta
}-1\right) \leq l\theta ^{-\beta _{1}}\left( \theta ^{(l+1)\beta }-1\right) 
\\
&\leq &l\delta ^{\beta _{1}}\left( \theta ^{(l+1)\beta }-1\right) \leq
\left( \frac{8CC_{2}4^{\beta _{2}}}{C_{1}}C_{H}2^{1/p}\right) ^{-1}.
\end{eqnarray*}%
It follows that%
\begin{equation*}
\sum_{j=1}^{\infty }\frac{\theta ^{(j+1)\beta }-1}{\theta ^{j\beta _{1}}}%
=\sum_{j=1}^{l}\theta ^{-j\beta _{1}}\left( \theta ^{(j+1)\beta }-1\right)
+\sum_{j=l+1}^{\infty }\theta ^{-j\beta _{1}}\left( \theta ^{(j+1)\beta
}-1\right) \leq 2\left( \frac{8CC_{2}4^{\beta _{2}}}{C_{1}}%
C_{H}2^{1/p}\right) ^{-1},
\end{equation*}%
from which, we see by (\ref{51-1}) that 
\begin{equation*}
\einf_{B_{r\theta ^{-k}}}h\geq \left( C_{H}2^{1/p}\right) ^{-1}-\frac{%
2CC_{2}4^{\beta _{2}}}{C_{1}}\cdot 2\left( \frac{8CC_{2}4^{\beta _{2}}}{C_{1}%
}C_{H}2^{1/p}\right) ^{-1}=\left( 2C_{H}2^{1/p}\right) ^{-1}=\lambda .
\end{equation*}%
Therefore, $v\leq 1-\lambda $ in $B_{r\theta ^{-k}}$, thus showing (\ref%
{603-4}) when (\ref{603-2}) is satisfied, as desired.

It remains to consider the case when (\ref{603-3}) is satisfied. We need to
show 
\begin{equation}
v\geq -1+\lambda \text{ \ in }B_{r\theta ^{-k}}\text{.}  \label{51-2}
\end{equation}%
Indeed, consider the function $h=1+v$. Similar to the argument above,
setting $M_{k}=M_{k-1}$ and $m_{k}=M_{k}-K\theta ^{-k\beta }$, one can
obtain (\ref{51-2}). The proof is complete.
\end{proof}

From the above, we immediately get the H\"{o}lder continuity of harmonic
functions.

\begin{lemma}
\label{T2}Let $(\mathcal{E},\mathcal{F})$ be a regular Dirichlet form in $%
L^{2}$ without killing part. If conditions $(\mathrm{wEH})$ and $(\mathrm{TJ}%
)$ hold, then there exist three constants $C>0,\theta \in (0,1]$ and $%
\varepsilon \in (0,1)$ such that, for any ball $B\left( x_{0},r\right) $
with $r<\sigma \overline{R}$ and for any globally bounded function $u$,
which is harmonic in $B\left( x_{0},r\right) $, 
\begin{equation}
|u(x)-u(y)|\leq C\left( \frac{d(x,y)}{r}\right) ^{\theta }\Vert u\Vert
_{L^{\infty }}.  \label{608}
\end{equation}%
for almost every points $x,y\in B\left( x_{0},\varepsilon r\right) $.
\end{lemma}

\begin{proof}
Let the function $u\in L^{\infty }$ be harmonic in $B\left( x_{0},r\right) $
with $r<\sigma \overline{R}$. By Lemma \ref{lemma:EHR}, 
\begin{equation}
\eosc_{B(x_{0},\rho )}u\leq C||u||_{L^{\infty }}\left( \frac{\rho }{r}%
\right) ^{\beta },\ \ 0<\rho \leq r.  \label{609}
\end{equation}%
We show that (\ref{608}) holds for $\theta =\beta ,\varepsilon =1/4$.

Indeed, let $x$ be any point $x$ in $B(x_{0},r/4)$, the function $u$ is
harmonic in $B(x,\frac{3}{4}r)\subseteq B(x_0,r)$. Let $y$ be a point in $%
B(x_{0},r/4)$. Applying (\ref{609}) with $x_{0}$ replaced by $x$, $r$ by $%
\frac{3}{4}r$ and with $\rho =\frac{3}{2}d(x,y)$, we obtain 
\begin{equation*}
|u(x)-u(y)|\leq \eosc_{B(x,\frac{3}{2}d(x,y))}u\leq C||u||_{L^{\infty
}}\left( \frac{3d(x,y)/2}{3r/4}\right) ^{\beta }=C2^{\beta }\left( \frac{%
d(x,y)}{r}\right) ^{\beta }||u||_{L^{\infty }},
\end{equation*}%
thus showing (\ref{608}). The proof is complete.
\end{proof}

Another consequence of the weak elliptic Harnack inequality is that it
implies a lower bound of the mean exit time on a ball, as we will see below.

Recall that the operator $\mathcal{L}^{\Omega }$ is the generator of the
Dirichlet form $(\mathcal{E},\mathcal{F}(\Omega ))$ for any non-empty open
subset $\Omega $ of $M$. For a ball $B\subset M$, let the function $E^{B}$
be a weak solution of the Poisson-type equation $-\mathcal{L}^{B}u=1$ in $B$%
, that is%
\begin{equation}
\mathcal{E}(E^{B},\varphi )=(1,\varphi )\text{ \ for any }0\leq \varphi \in 
\mathcal{F}(B).  \label{81}
\end{equation}%
We say that \emph{condition }$(\mathrm{E}_{\geq })$ holds if there exist
three constants $C>0$ and $\sigma ,\delta $ in $(0,1)$ such that, for all
balls $B\subset M$ with radius less than $\sigma \overline{R}$, 
\begin{equation}
\einf_{x\in \delta B}E^{B}(x)\geq Cw(B).  \label{E>}
\end{equation}%
We say that \emph{condition }$(\mathrm{E}_{\leq })$ holds if there exist two
constants $C>0$ and $\sigma $ in $(0,1)$ such that, for all balls $B\subset
M $ with radius less than $\sigma \overline{R}$, 
\begin{equation}
||E^{B}||_{L^{\infty }}\leq Cw(B).  \label{eq:vol_29}
\end{equation}

\begin{lemma}
\label{L80}Let $(\mathcal{E},\mathcal{F})$ be a regular Dirichlet form in $%
L^{2}$. Then%
\begin{equation}
(\mathrm{VD})+(\mathrm{Cap}_{\leq })+(\mathrm{FK})+(\mathrm{wEH})\Rightarrow
(\mathrm{E}_{\geq })+(\mathrm{E}_{\leq }).  \label{82}
\end{equation}
\end{lemma}

\begin{proof}
Note that by \cite[Theorem 9.4 p.1542]{GrigoryanHuLau.2015.JMSJ1485}%
\begin{equation}
(\mathrm{FK})\Rightarrow (\mathrm{E}_{\leq }),  \label{84}
\end{equation}%
(observing that we only use condition $(\mathrm{FK})$ at this stage).

It remains to show the implication%
\begin{equation}
(\mathrm{VD})+(\mathrm{Cap}_{\leq })+(\mathrm{FK})+(\mathrm{wEH})\Rightarrow
(\mathrm{E}_{\geq }).  \label{83}
\end{equation}%
Let $\delta $ be the same constant as in condition $(\mathrm{wEH})$. Without
loss of generality, assume that $\delta <\frac{2}{3}$. Let $B:=B(x_{0},R)$
be a ball in $M$ with $R<\sigma \overline{R}$. Let $u$ be the unique weak
solution such that 
\begin{equation}
\mathcal{E}(u,\varphi )=(1_{\delta B},\varphi )\text{ \ for any }0\leq
\varphi \in \mathcal{F}(B).  \label{613}
\end{equation}%
It is known that $u\in \mathcal{F}(B)$, $u\geq 0$ in $M$, and $u$ is
superharmonic in $B$, see for example \cite[Lemma 5.1]%
{GrigoryanHu.2014.CJM641}. Applying (\ref{eq:vol_65}) in condition $(\mathrm{%
wEH})$ on the function $u$ and the ball $B$, and with $f=0$, $r=\delta R$,
and noting that $u_{-}=0$ in $M$, we obtain 
\begin{equation}
\left( \fint_{\delta B}u^{p}d\mu \right) ^{1/p}\leq C_{H}\einf_{\delta B}u.
\label{614}
\end{equation}

On the other hand, we have by condition $(\mathrm{Cap}_{\leq })$ 
\begin{equation}
\mathcal{E}(\phi ,\phi )\leq C\frac{\mu (B)}{w(B)}  \label{612}
\end{equation}%
for some $\phi \in \mathrm{cutoff}\left( (2/3)B,B\right) $.

Taking $\varphi =\phi $ in (\ref{613}) and using condition $(\mathrm{VD})$,
we see that 
\begin{equation}
\mathcal{E}(u,\phi )=(1_{\delta B},\phi )=\int_{\delta B}\phi d\mu =\mu
(\delta B)\geq C_{\mu }^{-1}\delta ^{d_{2}}\mu (B).  \label{615}
\end{equation}%
Taking $\varphi =u$ in (\ref{613}) and using the Cauchy-Schwarz inequality
and (\ref{612}), it follows that 
\begin{equation}
\mathcal{E}(u,\phi )\leq \sqrt{\mathcal{E}(u,u)\mathcal{E}(\phi ,\phi )}=%
\sqrt{(1_{\delta B},u)}\sqrt{\mathcal{E}(\phi ,\phi )}\leq C\sqrt{%
\int_{\delta B}ud\mu }\sqrt{\frac{\mu (B)}{w(B)}}.  \label{616}
\end{equation}

Therefore, combining (\ref{615}) and (\ref{616}), we obtain 
\begin{equation}
\int_{\delta B}ud\mu \geq C\mu (B)w(B).  \label{617}
\end{equation}

Since by (\ref{84}) 
\begin{equation*}
||u||_{L^{\infty }}\leq ||E^{B}||_{L^{\infty }}\leq Cw(B),
\end{equation*}%
we conclude by (\ref{614}) that 
\begin{align*}
\int_{\delta B}ud\mu & =\int_{\delta B}u^{p}\cdot u^{1-p}d\mu \leq \left(
Cw(B)\right) ^{1-p}\int_{\delta B}u^{p}d\mu =\left( Cw(B)\right) ^{1-p}\mu
(\delta B)\fint_{\delta B}u^{p}d\mu \\
& \leq C^{\prime }w(B)^{1-p}\mu (B)(\einf_{\delta B}u)^{p}\leq C^{\prime
\prime }w(B)^{1-p}\mu (B)(\einf_{\delta B}E^{B})^{p},
\end{align*}%
thus showing (\ref{E>}) by (\ref{617}). The proof is complete.
\end{proof}

Finally, we show that the weak elliptic Harnack inequality also implies a
Lemma of growth, termed \emph{condition} $(\mathrm{LG}_{0})$, for any \emph{%
global non-negative superharmonic }function.

We say that \emph{condition} $(\mathrm{LG}_{0})$ holds if there exist four
constants $\sigma ,\epsilon _{0},\eta ,\delta \in (0,1)$ such that, for any
ball $B:=B(x_{0},R)$ with radius $R\in (0,\sigma \overline{R})$ and for any $%
u\in \mathcal{F}^{\prime }\cap L^{\infty }$ that is superharmonic in $B$ and
non-negative globally in $M$, if

\begin{equation}
\frac{\mu (\delta B\cap \{u<a\})}{\mu (\delta B)}\leq \epsilon _{0}
\label{701}
\end{equation}%
for some $a>0$, then

\begin{equation}
\einf_{\delta B}u\geq \eta a.  \label{702}
\end{equation}%
We remark that the superharmonic function $u$ in condition $(\mathrm{LG}%
_{0}) $ is required to be non-negative \emph{globally} in $M$, instead of
being non-negative \emph{locally} in condition\emph{\ }$(\mathrm{LG})$ given
in Definition \ref{def-LG}.

\begin{lemma}
\label{L3}Let $(\mathcal{E},\mathcal{F})$ be a regular Dirichlet form in $%
L^{2}$. Then%
\begin{equation}
(\mathrm{wEH})\Rightarrow (\mathrm{LG}_{0}).  \label{85}
\end{equation}
\end{lemma}

\begin{proof}
Let $u\in \mathcal{F}^{\prime }\cap L^{\infty }$ be superharmonic in $B$ and
non-negative globally in $M$. Assume that (\ref{701}) holds, namely,%
\begin{equation*}
\frac{\mu (\delta B\cap \{u\geq a\})}{\mu (\delta B)}=1-\frac{\mu (\delta
B\cap \{u<a\})}{\mu (\delta B)}\geq 1-\epsilon _{0}.
\end{equation*}%
Since $u_{-}=0$ in $M$, we see that 
\begin{equation*}
T_{\frac{3}{4}B_{R},B_{R}}(u_{-})\equiv 0.
\end{equation*}%
Applying (\ref{eq:vol_65}) with $r=\delta R$ and $f\equiv 0$, it follows
that 
\begin{equation*}
\einf_{\delta B}u\geq C_{H}^{-1}\left( \fint_{\delta B}u^{p}d\mu \right)
^{1/p}\geq C_{H}^{-1}a\left( \frac{\mu (\delta B\cap \{u\geq a\})}{\mu
(\delta B)}\right) ^{1/p}\geq C_{H}^{-1}(1-\epsilon _{0})^{1/p}a,
\end{equation*}%
thus showing that (\ref{702}) is true with $\eta =C_{H}^{-1}(1-\epsilon
_{0})^{1/p}$. The proof is complete.
\end{proof}

Lemma \ref{L80} above gives a direct, simpler proof of obtaining a lower
bound of the mean exit time from the weak elliptic Harnack inequality. We
remark that this conclusion can also be obtained in a more indirect way,
without recourse to condition $(\mathrm{FK})$. Indeed, the implication%
\begin{equation*}
(\mathrm{VD})+(\mathrm{Cap}_{\leq })+(\mathrm{LG}_{0})\Rightarrow (\mathrm{E}%
_{\geq })
\end{equation*}%
has been proved in a forthcoming paper \cite{ghh21TP} for any regular
Dirichlet form in $L^{2}$. Combining this with (\ref{85}), we have 
\begin{equation*}
(\mathrm{VD})+(\mathrm{Cap}_{\leq })+(\mathrm{wEH})\Rightarrow (\mathrm{VD}%
)+(\mathrm{Cap}_{\leq })+(\mathrm{LG}_{0})\Rightarrow (\mathrm{E}_{\geq }),
\end{equation*}%
from which, we also obtain condition $(\mathrm{E}_{\geq })$ from the weak
elliptic Harnack inequality but without using $(\mathrm{FK})$. We do need
condition $(\mathrm{FK})$ in Lemma \ref{L80}, not only in deriving condition 
$(\mathrm{E}_{\leq })$ but also in deriving condition $(\mathrm{E}_{\geq })$.

\section{An example}

\label{sect-1}In this section we give an example to illustrate Theorem \ref%
{thm:M1}. We show that the assumptions $(\mathrm{VD})$, $(\mathrm{RVD})$, $(%
\mathrm{Gcap})$, $(\mathrm{TJ})$, $(\mathrm{PI})$ are all satisfied so that
the weak Harnack inequality holds, but the jump kernel does not exist. This
example is essentially taken from \cite[Section 15]{BendikovGrigoryanHu18},
see also \cite{ghh21TP}.

\begin{example}[Ultra-metric space]
\label{ex2}\textrm{Let $\beta ,\alpha _{1},\alpha _{2}$ be three positive
numbers. Let $\left( M_{i},d_{i},\mu _{i}\right) $ for $i=1,2$ be two
ultrametric spaces, where $d_{i}$ is an ultra-metric:%
\begin{equation*}
d_{i}(x,y)\leq \max \{d_{i}(x,z),d_{i}(z,y)\}\text{ \ for all }x,y,z\in
M_{i},
\end{equation*}%
and the measure $\mu _{i}$ is Ahlfors-regular: 
\begin{equation}
C^{-1}r^{\alpha _{i}}\leq \mu _{i}\left( B\left( x_{i},r\right) \right) \leq
Cr^{\alpha _{i}}\text{ for all }x_{i}\in M_{i}\text{ and all }r>0  \label{e1}
\end{equation}%
for some constant $C\geq 1$. Let $J_{i}$ be a function on $M_{i}\times M_{i}$
for $i=1,2$ such that for $\mu _{i}$-almost all $x_{i},y_{i}\in M_{i}$, 
\begin{equation}
J_{i}\left( x_{i},y_{i}\right) =d_{i}\left( x_{i},y_{i}\right) ^{-\left(
\alpha _{i}+\beta \right) }.  \label{e2}
\end{equation}
}

\textrm{Consider the product space $M:=M_{1}\times M_{2}$ equipped with
product measure $\mu :=\mu _{1}\times \mu _{2}$ and the metric 
\begin{equation*}
d(x,y):=\max \left\{ d_{1}\left( x_{1},y_{1}\right) ,d_{2}\left(
x_{2},y_{2}\right) \right\} \text{ for }x=\left( x_{1},x_{2}\right)
,y=\left( y_{1},y_{2}\right) \text{ in }M.
\end{equation*}%
Clearly, $(M,d,\mu )$ is an ultrametric space and for any point $x=\left(
x_{1},x_{2}\right) $ in $M$, the metric ball $B(x,r)$ in $M$ can be written
as 
\begin{equation}
B(x,r)=B(x_{1},r)\times B(x_{2},r)\text{ for any }r>0.  \label{bb}
\end{equation}%
From this, we see that for any point $x=\left( x_{1},x_{2}\right) $ in $M$
and any $r>0$, 
\begin{equation}
V(x,r)=\mu (B(x,r))=\mu _{1}\left( B\left( x_{1},r\right) \right) \mu
_{2}\left( B\left( x_{2},r\right) \right) \asymp r^{\alpha _{1}+\alpha
_{2}}=r^{\alpha },  \label{e3}
\end{equation}%
where $\alpha :=\alpha _{1}+\alpha _{2}$. For simplicity, let the scaling
function }$w(x,r)$\textrm{\ be defined by 
\begin{equation*}
w(x,r)=a(x)r^{\beta }\ \ \text{for any point}\ x\in M\ \text{and any }\ r>0,
\end{equation*}%
where $a(x)$ is a measurable function on $M$ with $C^{-1}\leq a(x)\leq C$
for all $x\in M$ ($C\geq 1$). Clearly, such a function $w$ satisfies (\ref%
{eq:vol_0}) and 
\begin{equation}
C^{-1}r^{\beta }\leq w(x,r)\leq Cr^{\beta }.  \label{ew}
\end{equation}%
}

\textrm{Define the measure $J$ on $\mathcal{B}(M\times M)$ by $%
J(dx,dy)=J(x,dy)\mu (dx)$, where $J(x,dy)$ is a transition function on $%
M\times \mathcal{B}(M)$ given by 
\begin{equation}
J(x,dy)=J_{1}(x_{1},y_{1})\mu _{1}(dy_{1})\delta
_{x_{2}}(dy_{2})+J_{2}(x_{2},y_{2})\mu _{2}(dy_{2})\delta _{x_{1}}(dy_{1})
\label{e4}
\end{equation}%
for any points $x=\left( x_{1},x_{2}\right) ,y=\left( y_{1},y_{2}\right) $
in $M$, where $\delta _{b}(dx)$ is the Dirac measure concentrated at point $%
b $. By (\ref{e3}) and (\ref{e4}), we have for any $r>0$ and any point $%
x=(x_{1},x_{2})\in M$, 
\begin{align}
\int_{B(x,r)^{c}}J(x,dy)& =\int_{B(x,r)^{c}}\left( J_{1}(x_{1},y_{1})\mu
_{1}(dy_{1})\delta _{x_{2}}(dy_{2})+J_{2}(x_{2},y_{2})\mu _{2}(dy_{2})\delta
_{x_{1}}(dy_{1})\right)  \notag \\
& =\int_{B(x_{1},r)^{c}}J_{1}(x_{1},y_{1})\mu
_{1}(dy_{1})+\int_{B(x_{2},r)^{c}}J_{2}(x_{2},y_{2})\mu _{2}(dy_{2})  \notag
\\
& \leq \frac{C}{r^{\beta }}+\frac{C}{r^{\beta }}=\frac{2C}{r^{\beta }}\leq 
\frac{C^{\prime }}{w(x,r)}\text{ \ \ \ \ (using (\ref{e1}), (\ref{e2}) and (%
\ref{ew}))},  \label{tj}
\end{align}%
which is exactly condition $(\mathrm{TJ})$. }

\textrm{Let $\left( \mathcal{E},\mathcal{F}\right) $ be a Dirichlet form in $%
L^{2}\left( M,\mu \right) $ defined by 
\begin{equation*}
\mathcal{E}(u,v)=\iint_{M\times M}\left( u\left( x\right) -u\left( y\right)
\right) \left( v\left( x\right) -v\left( y\right) \right) J\left(
x,dy\right) \mu (dx),\quad u,v\in \mathcal{F},
\end{equation*}%
where the space $\mathcal{F}$ is the closure of the set 
\begin{equation*}
\left\{ \sum_{i=0}^{n}c_{i}\mathbf{1}_{B_{i}}:n\in \mathbb{N},c_{i}\in 
\mathbb{R},B_{i}\text{ is a compact ball }\right\}
\end{equation*}%
under the inner product 
\begin{equation*}
\sqrt{\mathcal{E}(\cdot ,\cdot )+(\cdot ,\cdot )_{L^{2}\left( M,\mu \right) }%
}.
\end{equation*}
}

\textrm{By \cite[Theorem 2.2]{BendikovGrigoryanHu18}, the form $\left( 
\mathcal{E},\mathcal{F}\right) $ is regular and non-local. By (\ref{e3}),
the measure $\mu $ satisfies conditions $(\mathrm{VD})$ and $(\mathrm{RVD})$%
, whilst condition $(\mathrm{Gcap})$ automatically holds since it follows
directly from condition $(\mathrm{TJ})$ and the ultrametric property. Hence,
conditions $(\mathrm{VD})$, $(\mathrm{RVD})$, $(\mathrm{Gcap})$, $(\mathrm{TJ%
})$ in Theorem \ref{thm:M1} are satisfied. }

\textrm{It remains to verify condition $(\mathrm{PI})$. Indeed, let $%
B:=B(x^{\prime },r)$ be a metric ball in $M$. Writing up $x^{\prime
}=(x_{0},y_{0})$ with $x_{0}\in M_{1}$, $y_{0}\in M_{2}$, we see $%
B=B(x_{0},r)\times B(y_{0},r)$ by using (\ref{bb}). By (\ref{e4}), (\ref{e2})%
\begin{eqnarray*}
\int_{B}\int_{B}(u(x)-u(y))^{2}J(x,dy)\mu (dx) &=&\int_{B}\left\{
\int_{B(x_{0},r)}\frac{(u(x_{1},x_{2})-u(y_{1},x_{2}))^{2}}{%
d_{1}(x_{1},y_{1})^{\alpha _{1}+\beta }}\mu _{1}(dy_{1})\right. \\
&&+\left. \int_{B(y_{0},r)}\frac{(u(x_{1},x_{2})-u(x_{1},y_{2}))^{2}}{%
d_{2}(x_{2},y_{2})^{\alpha _{2}+\beta }}\mu _{2}(dy_{2})\right\} \mu (dx).
\end{eqnarray*}%
The first integral on the right-hand side is estimated as follows: for any $%
(x_{1},x_{2})\in B(x_{0},r)\times B(y_{0},r)$, 
\begin{eqnarray*}
&&\int_{B(x_{0},r)}\frac{(u(x_{1},x_{2})-u(y_{1},x_{2}))^{2}}{%
d_{1}(x_{1},y_{1})^{\alpha _{1}+\beta }}\mu _{1}(dy_{1}) \\
&\geq &\int_{B(x_{0},r)}\frac{(u(x_{1},x_{2})-u(y_{1},x_{2}))^{2}}{r^{\alpha
_{1}+\beta }}\mu _{1}(dy_{1}) \\
&\geq &C^{-1}\int_{B(y_{0},r)}\int_{B(x_{0},r)}\frac{%
(u(x_{1},x_{2})-u(y_{1},x_{2}))^{2}}{r^{\alpha _{1}+\alpha _{2}+\beta }}\mu
_{1}(dy_{1})\mu _{2}(dy_{2}) \ \ \text{(using (\ref{e1}))} \\
&=&C^{-1}\int_{B}\frac{(u(x_{1},x_{2})-u(y_{1},x_{2}))^{2}}{r^{\alpha +\beta
}}\mu (dy)
\end{eqnarray*}%
by using the fact that $\alpha _{1}+\alpha _{2}=\alpha $, from which, we have%
\begin{equation*}
\int_{B}\int_{B(x_{0},r)}\frac{(u(x_{1},x_{2})-u(y_{1},x_{2}))^{2}}{%
d_{1}(x_{1},y_{1})^{\alpha _{1}+\beta }}\mu _{1}(dy_{1})\mu (dx)\geq
C^{-1}\int_{B}\int_{B}\frac{(u(x_{1},x_{2})-u(y_{1},x_{2}))^{2}}{r^{\alpha
+\beta }}\mu (dy)\mu (dx).
\end{equation*}%
Similarly, the second integral is estimated by 
\begin{eqnarray*}
&&\int_{B}\int_{B(y_{0},r)}\frac{(u(x_{1},x_{2})-u(x_{1},y_{2}))^{2}}{%
d_{2}(x_{2},y_{2})^{\alpha _{2}+\beta }}\mu _{2}(dy_{2})\mu (dx) \\
&\geq &C^{-1}\int_{B}\int_{B}\frac{(u(x_{1},x_{2})-u(x_{1},y_{2}))^{2}}{%
r^{\alpha +\beta }}\mu (dy)\mu (dx) \\
&=&C^{-1}\int_{B}\int_{B}\frac{(u(y_{1},y_{2})-u(y_{1},x_{2}))^{2}}{%
r^{\alpha +\beta }}\mu (dx)\mu (dy)\text{ \ (swapping }(x_{1},x_{2})\text{
with }(y_{1},y_{2})\text{).}
\end{eqnarray*}%
Therefore, we conclude from above that, using the elementary inequality $%
a^{2}+b^{2}\geq (a+b)^{2}/2$,%
\begin{eqnarray*}
\int_{B}\int_{B}(u(x)-u(y))^{2}J(x,dy)\mu (dx) &\geq &C^{-1}\left\{
\int_{B}\int_{B}\frac{(u(x_{1},x_{2})-u(y_{1},x_{2}))^{2}}{r^{\alpha +\beta }%
}\mu (dy)\mu (dx)\right. \\
&&+\left. \int_{B}\int_{B}\frac{(u(y_{1},y_{2})-u(y_{1},x_{2}))^{2}}{%
r^{\alpha +\beta }}\mu (dx)\mu (dy)\right\} \\
&\geq &C^{-1}\int_{B}\int_{B}\frac{(u(x_{1},x_{2})-u(y_{1},y_{2}))^{2}}{%
2r^{\alpha +\beta }}\mu (dx)\mu (dy) \\
&\geq &\frac{C^{\prime }r^{-\beta }}{\mu (B)}\int_{B}\int_{B}(u(x)-u(y))^{2}%
\mu (dx)\mu (dy)\ \ \text{(using (\ref{e3}))} \\
&=&2C^{\prime }r^{-\beta }\int_{B}(u-{u_{B}})^{2}d\mu \\
&\geq &\frac{C}{w(B)}\int_{B}(u-{u_{B}})^{2}d\mu \ \ \text{(using (\ref{ew}%
)),}
\end{eqnarray*}%
thus showing that condition $\mathrm{(PI)}$ with $\kappa $ $=1$ is
satisfied. }

\textrm{Therefore, all the hypotheses in Theorem \ref{thm:M1} are satisfied,
and the weak elliptic Harnack inequality follows. We mention that the jump
kernel does not exist by (\ref{e4}) in this case. }
\end{example}

\section[John-Nirenberg inequality]{Appendix}

In this appendix, we collect some known results that have been cited in this
paper. Recall the John-Nirenberg inequality for $\text{BMO}$ functions on a
doubling space.

\begin{definition}[$\text{BMO}$ function]
\label{df:BMO} For a locally integrable function $u$ on an open set $\Omega $%
, the seminorm $||u||_{\mathrm{BMO}(\Omega )}$ is defined by 
\begin{equation*}
||u||_{\mathrm{BMO}(\Omega )}:=\sup_{B\subset \Omega }\fint_{B}|u-u_{B}|d\mu
,
\end{equation*}%
where the supremum is taken over all the balls contained in $\Omega $. The
space $\mathrm{BMO}(\Omega )$ consists of all locally integrable functions $%
u $ on $\Omega $ such that $||u||_{\mathrm{BMO}(\Omega )}<\infty $.
\end{definition}

The following was addressed in \cite[Theorem 5.2]%
{AaltoBerkovitsKansanenYue.2011.SM21}.

\begin{lemma}[John-Nirenberg inequality]
\label{lemma:JN}Let $(M,d,\mu )$ be a metric measure space satisfying
condition $(\mathrm{VD})$. If $u\in \mathrm{BMO}(\Omega )$ for a non-empty
open subset $\Omega $ of $M$, then 
\begin{equation*}
\mu (\{x\in B:|u-u_{B}|>\lambda \})\geq c_{1}\mu (B)\exp \left( -\frac{%
c_{2}\lambda }{||u||_{\mathrm{BMO}(\Omega )}}\right)
\end{equation*}%
for any ball with $12B\subseteq \Omega $ and any $\lambda >0$, where
constants $c_{1},c_{2}$ are independent of $u,\lambda ,\Omega $ and ball $B$.
\end{lemma}

The following is a folklore, see for example \cite[Corollary 5.6]%
{BiroliMosco.1995.AMPA4125}.

\begin{lemma}
\label{cor:JN} Let $(M,d,\mu )$ be a metric measure space satisfying
condition $(\mathrm{VD})$. Let $B_{0}:=B(x_{0},R)$ be a ball in $M$. Then
for any $u\in \mathrm{BMO}(B_{0})$ 
\begin{equation}
\left\{ \fint_{B}\exp \left( \frac{c_{2}}{2b}u\right) d\mu \right\} \left\{ %
\fint_{B}\exp \left( -\frac{c_{2}}{2b}u\right) d\mu \right\} \leq
(1+c_{1})^{2}  \label{58}
\end{equation}%
for any ball $B$ with $12B\subseteq B_{0}$ and any $b\geq ||u||_{\mathrm{BMO}%
(B_{0})}$, where the constants $c_{1},c_{2}$ are the same as in Lemma \ref%
{lemma:JN}.
\end{lemma}

The following has been proved in a forthcoming paper \cite{ghh21TP}.

\begin{proposition}
\label{P61}Let $(\mathcal{E},\mathcal{F})$ be a regular Dirichlet form in $%
L^{2}$ without killing part. Assume that a function $F\in C^{2}(\mathbb{R})$
satisfies 
\begin{equation*}
F^{\prime \prime }\geq 0,~\sup_{\mathbb{R}}|F^{\prime }|<\infty ,~\sup_{%
\mathbb{R}}F^{\prime \prime }<\infty .
\end{equation*}%
Then for any $u,\varphi \in \mathcal{F}^{\prime }\cap L^{\infty }$, both
functions $F(u),F^{\prime }(u)\varphi $ belong to the space $\mathcal{F}%
^{\prime }\cap L^{\infty }$. Moreover, if further $\varphi \geq 0$ in $M$,
then 
\begin{equation}
\mathcal{E}(F(u),\varphi )\leq \mathcal{E}(u,F^{\prime }(u)\varphi ).
\label{eq61}
\end{equation}
\end{proposition}

The following is taken from in \cite[Lemma 2.12]{MaRockner.1992.209}.

\begin{lemma}
\label{lem:A1} Let $(\mathcal{E},\mathcal{F})$ be a Dirichlet form in $L^{2}$%
. If each $f_{n}\in \mathcal{F}$ and 
\begin{equation*}
f_{n}\overset{L^{2}}{\rightarrow }f,\text{ \ }\sup_{n}\mathcal{E}%
(f_{n})<\infty ,
\end{equation*}%
then $f\in \mathcal{F}$, and there exists a subsequence, still denoted by $%
\{f_{n}\}$, such that $f_{n}\overset{\mathcal{E}}{\rightharpoonup }f$
weakly, that is, 
\begin{equation*}
\mathcal{E}(f_{n},\varphi )\rightarrow \mathcal{E}(f,\varphi )
\end{equation*}%
as $n\rightarrow \infty $ for any $\varphi \in \mathcal{F}$. Moreover, we
have 
\begin{equation*}
\mathcal{E}(f)\leq \liminf_{n\rightarrow \infty }\mathcal{E}(f_{n}).
\end{equation*}
\end{lemma}

%\bibliographystyle{siam}
%\bibliography{ref}

\end{document}